\documentclass[12pt]{amsart}
\usepackage{verbatim,amsmath,amssymb,mathabx}
\usepackage[active]{srcltx}
\usepackage[all]{xy}
\usepackage{graphicx}

\newtheorem{thm}{Theorem}[section]
\newtheorem{prop}[thm]{Proposition}
\newtheorem{lem}[thm]{Lemma}

\newtheorem{ass}[thm]{Assumption}
\newtheorem{cor}[thm]{Corollary}

\newtheorem*{thmA}{Theorem A}

\newtheorem*{thmB}{Theorem B}

\theoremstyle{definition}
\newtheorem{dfn}[thm]{Definition}
\newtheorem{ex}[thm]{Example}

\def\C{\mathbb{C}}   

\def\R{\mathbb{R}}

\def\0{\emptyset}
 \def\Bc{\mathcal{B}} 
\def\Ec{\mathcal{E}}

 \def\Vc{\mathcal{V}}  
\def\Xc{\mathcal{X}}   

\def\Z{\mathbb{Z}}
\renewcommand\emptyset{\varnothing}

\def\eps{\varepsilon}
\def\ol{\overline}
\def\d{\partial}

\def\ge{\geqslant}

\def\oC{\xymatrix{*+<0.3pc>[o][F]{\C}}}
\def\Im{{\rm Im}}
\def\S{\mathbb{S}}

\def\<{\langle}
\def\>{\rangle}

\def\ivy{\mathrm{Ivy}}
\def\pmcg{\mathrm{PMod}}
\def\push{\mathcal{P}\hbox{\textit{ush}}}
\def\fs{\mathrm{FS}}

\begin{document}

\title[Invariant Spanning Trees]
{Invariant Spanning Trees for Quadratic Rational Maps}

\author[A.~Shepelevtseva]{Anastasia Shepelevtseva}

\author[V.~Timorin]{Vladlen~Timorin$^\star$}

\address[V.~Timorin, A.~Shepelevtseva]
{National Research University Higher School of Economics\\
6 Usacheva St., 119048 Moscow, Russia}

\address[A.~Shepelevtseva]
{Scuola Normale Superiore, 7 Piazza dei Cavalieri, 56126 Pisa, Italy}

\thanks{The study has been funded within the framework of the HSE University Basic Research Program and the Russian Academic Excellence Project '5-100.}

\email[Anastasia~Shepelevtseva]{asyashep@gmail.com}
\email[Vladlen~Timorin]{vtimorin@hse.ru}



\begin{abstract}
The main objects of study are Thurston equivalence classes of quadratic post-critically finite branched coverings.
For these maps, we introduce and study invariant spanning trees.
We give a computational procedure for searching for invariant spanning trees.
This procedure uses bisets over the fundamental group of a punctured sphere.
We also introduce a new combinatorial invariant of Thurston classes --- the ivy graph.
\end{abstract}

\maketitle

\section{Introduction}
Rational maps acting on the Riemann sphere are among central objects in complex dynamics.
\emph{Thurston's characterization theorem} allows to study these algebraic objects by topological tools.
It views rational maps within a much wider class of topological branched coverings.
(Branched self-coverings of the sphere whose critical points have finite orbits are called \emph{Thurston maps}.)
There is a natural equivalence relation on Thurston maps such that different rational functions are almost never equivalent.
(All exceptions are known and well-understood.)
Thurston's theorem provides a topological criterion for a Thurston map being equivalent to a rational function.
Thus, classification of Thurston maps up to equivalence is an important problem.
This fundamental problem has applications beyond complex dynamics, e.g., in group theory;
 it is a focus of recent developments, see e.g. \cite{BN06,BD17,CGNPP15,KL18,H17}.
We approach the problem via analogs of Hubbard trees for quadratic rational maps: \emph{invariant spanning trees}.

We will write $\S^2$ for the oriented topological 2-sphere.
By a \emph{graph} in the sphere, we here mean a 1-dimensional cell complex embedded into $\S^2$.
By \emph{vertices} and \emph{edges}, we mean 0-cells and 1-cells, respectively.
For a graph $G$, we will write $V(G)$ for the set of vertices of $G$ and $E(G)$ for the set of edges of $G$.
A \emph{tree} is a simply connected graph.
A vertex $x$ of a tree $T$ is called a \emph{branch point} if $T-\{a\}$ has more than 2 components.
Suppose that $P\subset\S^2$ is some finite subset.
A tree $T$ in $\S^2$ such that $P\subset V(T)$ is called a \emph{spanning tree} for $P$ if $V(T)-P$ consists of branch points.

Let $f:\S^2\to\S^2$ be an orientation preserving branched covering of degree 2.
The map $f$ has two critical points $c_1(f)$ and $c_2(f)$.
Let $v_1(f)$ and $v_2(f)$ be the corresponding critical values.
The \emph{post-critical set} of $f$ is defined as the smallest closed $f$-stable set including $\{v_1(f),v_2(f)\}$.
The post-critical set of $f$ will be denoted by $P(f)$.
If $P(f)$ is finite, then $f$ is said to be \emph{post-critically finite}.
Recall that a \emph{Thurston map} is a post-critically finite orientation preserving branched covering.
In this paper, we will only consider degree two Thurston maps.

\begin{dfn}[Invariant spanning tree]
Let $f:\S^2\to\S^2$ be a Thurston map.
A spanning tree $T$ for $P(f)$ is called an \emph{invariant spanning tree} for $f$ if:
\begin{enumerate}
 \item we have $f(T)\subset T$;
 \item vertices of $T$ map to vertices of $T$.
\end{enumerate}
\end{dfn}
This notion is close to what is called ``invariant trees'' in \cite{H17}.
Note that the restriction of $f$ to an edge of $T$ is injective unless the edge contains a critical point of $f$.
Consideration of invariant spanning trees is justified by the following examples
(we describe those of them, which are quadratic maps, in more detail later, see Section \ref{s:ex}):
\begin{enumerate}
  \item \emph{Hubbard trees} \cite{hubbdoua85,BBH92,poi93} can be connected to infinity to form invariant spanning trees.
  \item Invariant spanning trees can be constructed for \emph{formal matings} by joining the two Hubbard trees.
  (However, the tree structure sometimes does not survive in the corresponding topological matings).
  \item Classical \emph{captures} in the sense of \cite{Wittner,Rees_description}
    often come with invariant spanning trees.
  In fact, the original approach of Wittner used invariant trees.
  \item Sufficiently high iterates of expanding Thurston maps possess invariant spanning trees by \cite{H17}.
  This result has also been extended in \cite{H17} to rational maps with Sierpinski carpet Julia set.
  \item Extended Newton graphs constructed in \cite{LMS15} for post-critically finite Newton maps are often invariant trees.
  Invariant spanning trees can be obtained from them by erasing some of the vertices.
  \item With each \emph{critically fixed} rational map $f$, a certain bipartite graph is associated in \cite{CGNPP15}.
  Every edge if this graph is invariant, thus every spanning tree of this graph is an invariant tree for $f$.
  Again, an invariant spanning tree can be obtained by erasing some of the vertices.
  \item Some invariant spiders in the sense of \cite{HS94} are invariant spanning trees
  (possibly after removal of the critical leg in case of a strictly preperiodic critical point).
  Note that invariant spiders may fail to be trees.
\end{enumerate}
Section \ref{s:ex} deals with examples of invariant spanning trees.
However, we restrict our attention to the case of degree 2 Thurston maps.

Let $x$ be a vertex of a tree $T\subset\S^2$, and $e$ be an edge of $T$.
If $x$ is in the closure of $e$, then we say that $x$ is \emph{incident} to $e$.
We also say that $e$ is incident to $x$.
The following result shows how to recover the Thurston equivalence class of $f$ from an invariant spanning tree of $f$.
Recall that a \emph{ribbon graph} (also known as a fat graph, or a cyclic graph) is an abstract graph
 in which the edges incident to each particular vertex are cyclically ordered.
By \cite{MA41}, ribbon trees are the same as isomorphism classes of embedded trees in $\S^2$.
Here an isomorphism of embedded trees is an orientation preserving self-homeomorphism of $\S^2$ that takes one tree to another.
Recall that the orientation of $\S^2$ is assumed to be fixed.
For a spanning tree $T$ for $P(f)$, we write $C(T)$ for the set of critical points of $f$ in $T$.

\begin{thmA}
Suppose that $f$, $g:\S^2\to\S^2$ are two Thurston maps of degree 2.
Let $T_f$ and $T_g$ be invariant spanning trees for $f$ and $g$, respectively.
Suppose that there is a cellular homeomorphism $\tau:T_f\to T_g$ with the following properties:
\begin{enumerate}
\item The map $\tau$ is an isomorphism of ribbon graphs.
\item We have $\tau\circ f=g\circ\tau$ on $V(T_f)\cup C(T_f)$.
\item The critical values of $f$ are mapped to critical values of $g$ by $\tau$.
\end{enumerate}
Suppose also that $\tau$ can be extended to edges of $f^{-1}(T_f)$ incident to points in $C(T_f)$ so that
 to preserve the cyclic order of edges incident to a given vertex of $C(T_f)$ and so that to satisfy $(2)$.
Then $f$ and $g$ are Thurston equivalent.
\end{thmA}

In other words, to know the Thurston equivalence class of $f$, it suffices to know the following data:
\begin{enumerate}
\item the ribbon graph structure of $T_f$;
 \item the restriction of the map $f$ to the set $V(T_f)\cup C(T_f)$;
 \item the cyclic order, in which pullbacks of certain edges of $T_f$ appear around a point of $C(T_f)$.
\end{enumerate}
These data are discrete and can be encoded symbolically.
Theorem A is likely to extend to higher degrees using the notion of an angled tree, cf. \cite{poi93}.

An important algebraic invariant of a Thurston map is its \emph{biset} over the fundamental group of $\S^2-P(f)$.
A biset is a convenient algebraic structure that carries a complete information about the Thurston class of $f$.
The (perhaps better known) \emph{iterated monodromy group} of $f$ can be immediately recovered from the biset.
A formal definition of a biset will be given in Section \ref{s:img}.
For now, we just emphasize that bisets admit compact symbolic descriptions somewhat similar to
  presentations of groups by generators and relations or presentation of linear maps by matrices.

\begin{thmB}
Suppose that $f$ is a Thurston map of degree 2, and $T$ is an invariant spanning tree for $f$.
There is an explicit presentation of the biset of $f$ based only on the data $(1)-(3)$ listed above.
\end{thmB}

In Section \ref{ss:ThmB-aut}, we make the statement of Theorem B more precise.
We provide an automaton representing the biset of $f$ in Theorem \ref{t:ThmB-hom}.
The construction is algorithmic.
Theorem B solves, in a particular case, the problem of combinatorial encoding of Thurston maps by means of invariant graphs.
Different contexts of this problem are addressed in \cite{CFP01,BM17,LMS15} for specific families of rational maps.
A relationship between the properties of invariant trees and the properties of the
  iterated monodromy group has been studied in \cite{H17}.

\subsection{Dynamical tree pairs}
It is not always easy to find an invariant spanning tree for a Thurston map $f$.
However, for any spanning tree $T$ for $P(f)$, it is easy to find another spanning tree $T^*$ that maps onto $T$.
The tree $T^*$ is not uniquely defined; there are several ways of choosing suitable subtrees in the graph $f^{-1}(T)$.
Suppose that we want to find an invariant spanning tree for $f$.
It is natural to look at an iterative process, a single step of which is the transition from $T$ to $T^*$.
Such an iterative process will be described below under the name of \emph{ivy iteration}.

We now proceed with a more formal exposition.
Let $f:\S^2\to\S^2$ be a Thurston map of degree two.
Consider two spanning trees $T^*$ and $T$ for $P(f)$ such that $f(T^*)\subset T$.
Moreover, we assume that
\begin{enumerate}
 \item the vertices of $T^*$ are mapped to vertices of $T$ under $f$;
 \item all critical values of $f$ are vertices of $T$;
\end{enumerate}
If these assumptions are fulfilled, then $(T^*,T)$ is called a \emph{dynamical tree pair} for $f$.
Clearly, an invariant spanning tree $T$ for $f$ gives rise to a dynamical tree pair $(T,T)$.
Thus, dynamical tree pairs generalize invariant spanning trees.
Observe also that the restriction of $f$ to every edge of $T^*$ is injective unless the edge contains a critical point of $f$.

A spanning tree $T$ for $P(f)$ gives rise to a distinguished generating set $\Ec_T$ of the fundamental group $\pi_1(\S^2-P(f),y)$ with $y\in\S^2-T$.
Namely, $\Ec_T$ consists of the identity element and the homotopy classes of smooth loops based at $y$ intersecting $T$ only once and transversely (we will make this more precise later).

In Section \ref{s:img}, we state a theorem (Theorem \ref{t:ThmB-hom}) generalizing Theorem B.
It follows from Theorem \ref{t:ThmB-hom} that the biset of $f$ is determined by a dynamical tree pair $(T^*,T)$.
More precisely, the biset can be explicitly presented knowing the following discrete data:
\begin{enumerate}
 \item the ribbon graph structures on $T^*$, $T$;
 \item the map $f:V(T^*)\cup C(T^*)\to V(T)$;
 \item how elements of $\Ec_{T^*}$ are expressed through elements of $\Ec_T$
 (or how both $\Ec_{T^*}$, $\Ec_T$ are expressed through some other generating set of $\pi_1(\S^2-P(f),y)$).
\end{enumerate}
Vice versa, given $T$ and a presentation of the biset of $f$ in the basis associated with $T$, these data can be recovered.

\subsection{The ivy iteration}
The principal objective of this paper is to introduce a computational procedure for
finding invariant (or, more generally, periodic) spanning trees of degree 2 Thurston maps.

Let $f:\S^2\to\S^2$ be a degree 2 Thurston map.
Ivy iteration operates on isotopy classes (rel. $P(f)$) of spanning trees for $P(f)$, which we call \emph{ivy objects}.
Let $\ivy(f)$ denote the set of all ivy objects for $f$.
Let $T$ be a spanning tree, and $[T]$ be the corresponding ivy object.
A symbolic presentation of the biset of $f$ plus a symbolic encoding of the ribbon tree structure on $T$
  give rise to several choices of a spanning tree $T^*$ such that $(T^*,T)$ is a dynamical tree pair for $f$.
Roughly speaking, several choices for $T^*$ are related with different ways of choosing a spanning subtree in $f^{-1}(T)$.
Consider the \emph{pullback relation} $[T]\multimap [T^*]$ on $\ivy(f)$.\footnote{On a somewhat similar note,
the pullback relation on isotopy classes of simple closed curves in $\S^2-P(f)$ is discussed in \cite{Pil03,KPS16}.}
It equips $\ivy(f)$ with a structure of an abstract directed graph.
A subset $C\subset\ivy(f)$ is said to be \emph{pullback invariant} if $[T]\in C$ and $[T]\multimap [T^*]$ imply $[T^*]\in C$.

With the help of a computer, we found finite pullback invariant subsets in $\ivy(f)$
 for several simplest quadratic Thurston maps $f$.
Within these pullback invariant subsets, we found all invariant ivy objects.
Invariant ivy objects, obviously, correspond to invariant (up to homotopy) spanning trees for $f$.
Having a picture for a pullback invariant subset, we can also see many periodic ivy objects of various periods.
Some of these examples will be described in Section \ref{s:ex}.
Note that, if we found a spanning tree for $P(f)$ that is $f$-invariant up to homotopy, then
this tree is a genuine invariant spanning tree for some map homotopic to $f$.
This is good enough since we are interested in classification of Thurston maps up to Thurston equivalence
(in particular, homotopic maps are in the same class).

\subsubsection*{How the ivy iteration compares with known combinatorial algorithms}
The ivy iteration can be considered in the following general context.
The biset associated with $f$ is a way of compactly representing the ``combinatorics'' of $f$.
Unfortunately, the same biset may have very different presentations in different bases.
(A linear algebra analog is that the same linear map has different matrices in different bases.)
Usually, a combinatorial description of $f$ yields a presentation of its biset.
However, given different combinatorial descriptions of a Thurston map, the problem is whether they describe the same thing.
This problem translates into comparison of bisets: do different presentations correspond to the same biset?

Up to date, there are several algorithmic approaches to the comparison of bisets.
We restrict our attention to bisets associated with rational, i.e., unobstructed, Thurston maps.
A natural idea is to look for the ``best'' presentation of a biset.
(This idea is somewhat similar to finding a Jordan normal form --- or some other normal form --- of a linear map.)
This general idea works well for polynomials.
In fact, the \emph{combinatorial spider algorithm} of Nekrashevich \cite{Nek09} aims at solving the problem.
Given a presentation of a biset by a \emph{twisted kneading automaton}, the algorithm searches for the best presentation,
  which is associated with a \emph{kneading automaton}.
The algorithm of \cite{Nek09} is a combinatorial implementation of the \emph{spider algorithm}
  originally developed by Hubbard and Schleicher \cite{HS94} for quadratic polynomials.
A version of the combinatorial spider algorithm was used in \cite{BN06} to identify twisted rabbits
  with the rabbit, co-rabbit, or the airplane polynomials, as well as to classify the twists of $z^2+i$.
In fact, the authors consider not only bisets over the fundamental group but also bisets over the pure mapping class group,
  which turns out to be useful for distinguishing twisted rabbits.
In \cite{BN06}, Bartholdi and Nekrashevich develop another approach to the same problem based on
  inspecting the correspondence on the moduli space associated with the Thurston pullback map.
This second approach relies on the fact that, in examples under consideration, there are few (namely, four) points in $P(f)$
(if $P(f)$ consists of only four points, then the moduli space has complex dimension one).

The second approach of \cite{BN06} has been further developed in \cite{KL18},
where all non-Euclidean Thurston maps with 4 or fewer post-critical points are classified.
The authors also provide an algorithm for identifying the twists of all such maps.
An extension of these results to Thurston maps with bigger post-critical sets is currently unavailable,
  not only because there are too many objects to classify but also because the technique is not easy to adapt.
Thus, to the best of our knowledge, purely combinatorial tools available up to date for comparing bisets
  are restricted either to specific types of Thurston maps (say, topological polynomials, expanding maps, etc.)
  or to maps with few post-critical points.

The process of finding invariant trees described in \cite{H17} is in principle algorithmic.
However, it applies under additional assumptions on the dynamics of the map (sphere or Sierpinski carpet Julia set)
  and produces an invariant tree only for a sufficiently high iterate of the map.

On the other hand, there are also ``floating-point'' algorithms, see e.g. \cite[Section V.2]{BD17}.
Most of these algorithms aim at turning the Thurston iteration into an efficient computation.
For example, given two bisets without obstructions, one can compare them as follows (see Corollary V.9 of \cite{BD17}).
For each of the two bisets, compute the coefficients of the corresponding rational map using a version of Thurston's algorithm.
Then the two rational maps can be compared by comparing the corresponding coefficients.
This approach works as an efficient computation but fails to provide good combinatorial tags to rational maps.

The ivy iteration may be regarded as an attempt to generalize the spider algorithm.
In fact, for quadratic polynomials, the spider algorithm (applied to a not necessarily invariant spider of an actual polynomial)
is the same as the ivy iteration, except that, the issue of arbitrary choices is resolved by specifying a kneading sequence.
A spider is a tree of a very specific shape; it is a star.
The ivy iteration may be in principle applied to arbitrary spanning trees, and to non-polynomial Thurston maps.
In this sense it is more general.
However, it is not as good as the spider algorithm because it is not an algorithm at all.
An important ingredient (an analog of a kneading sequence that would allow to make specific choices) is missing so far.
On the positive side, the ivy iteration can be implemented as a purely combinatorial procedure.
At any step of the iteration, we obtain a presentation of the biset associated with a Thurston map.
If the iteration converges, then we obtain a good presentation, and the hope is that there are only few good ones.
Moreover, the result is a nice visual tag associated with a rational map.
To the best of our knowledge, the ivy iteration does not coincide with other known computational procedures,
although it is conceptually unsophisticated and is based on the same general idea: that of taking pullbacks.

\subsection*{Terminological conventions}
In this paper, we talk about graphs in the sphere as well as abstract graphs.
The former notion belongs to topology, and the latter --- to combinatorics.
We try to clearly distinguish these notions.
A graph (without a specification) usually means a graph in the sphere.
When referring to an abstract graph, we always say ``abstract''.
We sometimes consider oriented edges of graphs in the sphere.
These are edges, for which some orientation (=direction) is specified.
On the other hand, we talk of directed edges in abstract directed graphs.
In this sense, a directed edge is a fundamental notion, which can be defined as an ordered pair of vertices.
It is not ``an edge equipped with a direction''.
Thus our terminological discrepancy between ``oriented edges'' and ``directed edges'' is intentional.

\section{Examples of invariant spanning trees}
\label{s:ex}
In this section, we describe some examples of invariant spanning trees.
We confine ourselves with degree two rational maps.

\subsection{Quadratic polynomials}
\label{ex:quadpoly}
Let $p(z)=z^2+c$ be a post-critically finite quadratic polynomial.
We will write $J(p)$ for the Julia set of $p$ and $K(p)$ for the filled Julia set of $p$.
Set $X$ to be the \emph{forward $p$-orbit} of $c$, i.e., the set $\{p^{\circ n}(c)\,|\, n\ge 0\}$.
The landing point of the dynamical external ray $R_p(0)$ of $p$ with argument $0$ is denoted by $x_\beta$.
This point is usually called the \emph{$\beta$-fixed point}.
We will use the terminology of \cite{M,poi93,poi10}, in particular, the notion of a regulated hull.
Define $T$ as the union of $\{\infty\}\cup R_p(0)$ and the regulated hull of $X\cup\{x_\beta\}$.
Then $T$ is an invariant spanning tree for $p$.

Recall that the regulated hull of $X$ is called the \emph{Hubbard tree} of $p$.
Thus $T$ is strictly bigger than the Hubbard tree of $p$.
It is important that the tree $T$ contains both critical values of $p$.
In our terminology, the Hubbard tree itself is not a spanning tree for $p$.
To specify a graph structure on $T$, we need to define vertices.
By definition, the vertices of $T$ are post-critical points and branch points of $T$.
On the other hand, $x_\beta$ is never a vertex of $T$.

As an example, an invariant spanning tree $T$ for the basilica polynomial $p(z)=z^2-1$ looks as follows:
$$
\xymatrix{
{\stackrel{-1}{\bullet}} \ar@{-}[r] &{\stackrel{0}{\circ}} \ar@{-}[r] &{\stackrel{\infty}{\bullet}}
}
$$
Critical values are shown as solid, and other vertices of $T$ as circles.
For the rabbit polynomial $z^2+c$, where $c\approx -0.122561 + 0.744862i$,
the tree $T$ looks as follows:
$$
\xymatrix{
 & {\stackrel{v}{\bullet}} \ar@{-}[d] & & &\\
   {\stackrel{w}{\circ}} \ar@{-}[r] & {\stackrel{x_\alpha}{\circ}} \ar@{-}[dr]& &\\
 & &{\stackrel{0}{\circ}} \ar@{-}[r] & {\stackrel{\infty}{\bullet}}
}
$$
Here $x_\alpha$ is the \emph{$\alpha$-fixed point}, i.e.,
  the landing point of external rays with arguments $\frac 17$, $\frac 27$ and $\frac 47$.
The edge $0\infty$ contains the $\beta$-fixed point $x_\beta$ and the points $-x_\alpha$, $-w$.
The latter three points are not in $V(T)$ since they are neither post-critical nor branch points.

\subsection{Matings}
\label{ss:mat}
Let $p$ and $q$ be two post-critically finite quadratic polynomials.
Consider a compactification $\oC$ of $\C$ obtained by adding a circle of infinity.
More precisely, the circle at infinity is parameterized by the arguments of external rays.
For an angle $\theta\in\R/\Z$, we let $R_p(\theta)$ be the corresponding external ray in the dynamical plane of $p$.
We will write $E_p(\theta)$ for the corresponding point in the circle at infinity.
Let $p$ and $q$ act on different copies of $\oC$, say, $p$ acts on $\oC_p$ and $q$ on $\oC_q$.
Then $E_p(\theta)$ and $E_q(\theta)$ will refer to points in $\oC_p$ and $\oC_q$, respectively.
Consider the disjoint union $Y=\oC_p\sqcup\oC_q$.
Let $\sim$ be an equivalence relation on $Y$ defined as follows.
We have $x\sim y$ and $x\ne y$ if and only if one of the two points, say, $x$, has the form $E_p(\theta)$,
  and the other point $y$ has the form $E_q(-\theta)$.
The quotient space $\S^2_{p\amalg q}=Y/\sim$ is called the \emph{formal mating space} of $p$ and $q$.
It is easy to see that $\S^2_{p\amalg q}$ is homeomorphic to $\S^2$.
The map $F:Y\to Y$ defined as $p$ on $\oC_p$ and $q$ on $\oC_q$ descends to the quotient space.
Thus we have a naturally defined map $f:\S^2_{p\amalg q}\to \S^2_{p\amalg q}$.
We write $f=p\amalg q$ and call $f$ the \emph{formal mating} of $p$ and $q$.
To construct an invariant spanning tree for $f$, it suffices to construct invariant spanning trees for $p$ and $q$ as above,
  and then take the union of the two trees.
Below, the thus constructed invariant spanning tree is shown for $p\amalg q$, where $p$ is the rabbit polynomial,
  and $q$ is the basilica polynomial.
$$
\xymatrix{
 & {\stackrel{v}{\bullet}} \ar@{-}[d] & & & & \\
   {\stackrel{w}{\circ}} \ar@{-}[r] & {\stackrel{x_\alpha}{\circ}} \ar@{-}[dr]& & & \\
 & &{\stackrel{0}{\circ}} \ar@{-}[r]
& {\stackrel{\ol 0}{\circ}} \ar@{-}[r]& {\stackrel{\ol{-1}}{\bullet}}
}
$$
Here, $\ol 0$ and $\ol{-1}$ refer to the points $0$ and $-1$ in the dynamical plane $\oC_q$
(more precisely, in the image of this plane in the space $\S^2_{p\amalg q}$).
The point $\infty=\ol\infty$ (more precisely, the image of $E_p(0)$ and $E_q(0)$ in $\S^2_{p\amalg q}$)
is not a critical value anymore.
Moreover, this point is not a vertex of the invariant spanning tree shown above.
It belongs to the edge connecting $0$ with $\ol 0$.

\subsection{Captures}
\label{ss:cap}
The following definition of a capture is equivalent to the one from \cite{Rees_description}.
However, we phrase the definition somewhat differently.
Introduce a smooth structure on $\S^2$.
We also fix a smooth spherical metric on $\S^2$.
Given a vector $v_x$ at some point $x\in\S^2$ and $\eps>0$, there is a vector field $D(v_x,\eps)$ such that
\begin{enumerate}
  \item outside of the $\eps$-neighborhood of $x$ with respect to the spherical metric, $D(v_x,\eps)=0$;
  \item at point $x$, the vector $D(v_x,\eps)_x$ coincides with $v_x$.
\end{enumerate}
We may consistently choose vector fields $D(v_x,\eps)$ for all $x$, $v_x$ and $\eps$
 so that they depend continuously (or even smoothly) on all parameters.
Consider a smooth path $\beta :[0,1] \to \S^2$ and choose a small $\eps>0$.
Define the map $\sigma_{\beta}:\S^2\to \S^2$ as the time $[0,1]$ flow of the non-autonomous vector field $D(\dot\beta(t),\eps)$.
Here $\dot\beta(t)$ is the velocity vector of $\beta$ at the point $\beta(t)$.
The map $\sigma_\beta$ is a self-homeomorphism of $\S^2$ with the following properties:
\begin{enumerate}
 \item we have $\sigma_{\beta}(\beta(0)) = \beta(1)$;
 \item the map $\sigma_{\beta}$ is the identity outside of the $\eps$-neighborhood $U_\eps(\beta)$ of $\beta[0,1]$;
 \item the map $\sigma_\beta$ is homotopic to the identity modulo $\S^2-U_\eps(\beta)$.
\end{enumerate}
The homeomorphism $\sigma_{\beta}$ depends on $\beta$, $\eps$ and on a particular choice of $D(v_x,\eps)$.
However, if the path $\beta$ is fixed,
then any two such homeomorphisms $\sigma_{\beta}$ and $\tilde {\sigma}_{\beta}$ are homotopic relative to
$\S^2-U_\eps(\beta)$.

We can consider a composition $\sigma_{\beta} \circ p$, where $p$ is a post-critically finite quadratic polynomial,
 and the choice of $\beta$ depends on $p$.
Set $\beta(0)=\infty$, and place $\beta(1)$ at some strictly preperiodic point that is not postcritical.
If $U_\eps(\beta)$ does not contain finite post-critical points of the map $p$ and iterated images of $\beta(1)$,
then all such maps $\sigma_{\beta} \circ p$ with fixed $\beta$ are equivalent.
In other words, the Thurston equivalence class of $f=\sigma_{\beta} \circ p$ depends only on $\beta$ and $p$.
The post-critical set of $f$ is the union of $P(p)$ and the forward orbit of $\beta(1)$, including $\beta(1)$.
Note that $\beta(1)$ is a critical value of $f$, the image of the critical point $\infty$.
In fact, the homotopy class of $f$ does not change if we deform $\beta$ within the same homotopy class relative to $P(f)$.
When talking about $\sigma_\beta\circ p$, we will always assume that the set $\beta[0,1)$ is disjoint from $P(p)$ and from the forward orbit of $\beta(1)$.
The path $\beta$ is called a \emph{capture path} for $p$.

\begin{dfn}
The map $\sigma_{\beta} \circ p$ defined as above is called the (generalized) \emph{capture} of $p$ associated with $\beta$.
The capture $\sigma_\beta\circ p$ is said to be \emph{simple} if there is only one $t_0\in [0,1]$ with $\beta(t_0)\in J(p)$.
In the latter case, the corresponding capture path is called a \emph{simple capture path}.
\end{dfn}

Suppose that $\beta(1)$ is eventually mapped to a periodic critical point of $p$, i.e., to $0$ if $p(z)=z^2+c$.
Then a simple capture path $\beta:[0,1]\to\S^2$ looks as follows.
There is a parameter $t_0\in (0,1)$ such that $\beta[0,t_0)$ is in the basin of infinity,
 $\beta(t_0,1]$ is in the Fatou component eventually mapped to a super-attracting periodic basin,
 and $\beta(t_0)$ is a point of the Julia set.
We may arrange $\beta|_{[0,t_0)}$ to go along an external ray, and $\beta|_{(t_0,1]}$ to go along an internal ray.
If $\beta(1)\in J(p)$, then $\beta[0,1]$ can be chosen as the union of an external ray and its landing point.
Different simple capture paths lead to at most two different Thurston equivalence classes of captures
 provided that $p$ and $\beta(1)$ are fixed, cf. \cite[Section 2.8]{ReesV3}.

Generalized captures were first defined by M.Rees in \cite{Rees_description}.
Simple captures go back to B.Wittner \cite{Wittner}.
Both Wittner and Rees used the word ``capture'' to mean simple capture.
We, on the contrary, use the word ``capture'' to mean a generalized capture.
It is worth noting that the original approach of Wittner also used invariant trees.
The study of captures is motivated by the following theorem of M.Rees:

\begin{thm}[Polynomial-and-Path Theorem, Section 1.8 of \cite{Rees_description}]
 Suppose that $R$ is a rational function of degree two with a periodic critical point $c_1$.
 Suppose also that the other critical point $c_2$ of $R$ is not periodic but is eventually mapped to $c_1$.
 Then $R$ is equivalent to some capture $\sigma_\beta\circ p$.
 Moreover, the quadratic polynomial $p$ has a periodic critical point of the same period as $c_1$.
\end{thm}

Suppose that $\beta$ is a simple capture path for $p$, and $f=\sigma_\beta\circ p$ is the corresponding capture.
Let $T$ be the minimal subtree of the extended Hubbard tree of $p$ that includes $P(f)$.
Then $T$ satisfies the property $p(T)\subset T$.
Note that it may happen that $f(T)\not\subset T$, so that $T$ is not an invariant spanning tree for $f$.
For example, let $p$ be the airplane polynomial.
Choose $\beta(1)$ to be an iterated $p$-preimage of $0$ on an edge of the Hubbard tree of $p$.
Then $T$ coincides with the Hubbard tree set-theoretically but has more vertices.
Some edge $e$ of $T$ maps under $p$ so that the $\beta(1)\in p(e)$ but $\beta(1)$ is not an endpoint of $p(e)$.
The latter is a consequence of the fact that there are no vertices of $T$ mapping to $\beta(1)$.
The homeomorphism $\sigma_\beta$ displaces $p(e)$ so that $\sigma_\beta(p(e))$ no longer contains $\beta(1)$.
Thus $T$ is not forward invariant under $f=\sigma_\beta\circ p$.

It may seem plausible that $T$ can be deformed slightly into a genuine invariant spanning tree.
Unfortunately, this is not always true.
It is known that different simple captures (even those for which $\beta(1)$ is the same)
  may yield different Thurston equivalence classes, see e.g. \cite[Section 2.8]{ReesV3}.
If $T$ were deformable into an invariant spanning tree, then, by Theorem A,
  all simple captures with given $\beta(1)$ would be Thurston equivalent, a contradiction.

In the following lemma, by a \emph{support} of a homeomorphism $\sigma:\S^2\to\S^2$ we mean
  the closure of the set of points $x\in\S^2$ with $\sigma(x)\ne x$.

\begin{lem}
  \label{l:capT}
  Let $p$, $\beta$ and $T$ be as above.
Assume that the support of $\sigma_\beta$ is a sufficiently narrow neighborhood of $\beta[0,1]$, i.e.,
 a subset of the $\eps$-neighborhood of $\beta[0,1]$ for sufficiently small $\eps>0$.
Then $T$ is an invariant spanning tree for $f=\sigma_\beta\circ p$ whenever $\beta[0,1]\cap p(T)=\0$.
\end{lem}

Recall our assumption that the capture path $\beta$ is simple.

\begin{proof}
  Suppose that $\beta[0,1]\cap p(T)=\0$.
Then the support of $\sigma_\beta$ can be made disjoint from $p(T)$.
It follows that $\sigma_\beta=id$ on $p(T)$, therefore, $f(T)=\sigma_\beta(p(T))=p(T)\subset T$.
\end{proof}

\section{Proof of Theorem A}
Let $f:\S^2\to\S^2$ be a Thurston map of degree two.
It will be convenient to mark the critical points of $f$, i.e., to distinguish between $c_1(f)$ and $c_2(f)$.

\begin{dfn}[Marked Thurston maps]
 A (critically) marked Thurston map of degree two is an ordered triple $(f,c_1,c_2)$, where
  $f$ is a Thurston map of degree two, and $\{c_1,c_2\}$ is the set of all critical points of $f$.
 Thus, if $c_1\ne c_2$, then $(f,c_1,c_2)$ and $(f,c_2,c_1)$ are different marked Thurston maps.
 To lighten the notation, we will sometimes write $f$ for a Thurston map $(f,c_1,c_2)$.
 In this case, we will write $c_1(f)$, $c_2(f)$ to emphasize the dependence on $f$.
\end{dfn}

We now recall the definition of Thurston equivalence.

\begin{dfn}[Thurston equivalence]
\label{d:Th-eqiv}
 Let $f$ and $g$ be two Thurston maps.
 They are said to be \emph{Thurston equivalent} if there are two orientation preserving homeomorphisms $\phi$, $\psi:\S^2\to\S^2$ with the following properties:
 \begin{enumerate}
  \item We have $\phi=\psi$ on $P(f)$, and $\phi(P(f))=P(g)$.
  \item The maps $\phi$ and $\psi$ are isotopic modulo $P(f)$.
  \item We have $\psi\circ f=g\circ\phi$.
 \end{enumerate}
 If $f$ and $g$ are marked Thurston maps of degree two, then we additionally require that $\phi(v_i(f))=v_i(g)$ for $i=1,2$.
\end{dfn}

For example, two topologically conjugate Thurston maps are Thurston equivalent.
The following is another particular case of Thurston equivalence.

\begin{lem}
 \label{l:deform}
 Let $f_t$, $t\in [0,1]$ be a continuous family of Thurston maps with $P(f_t)=P(f_0)$.
 Then all $f_t$ are Thurston equivalent.
\end{lem}

Thurston maps $f_0$ and $f_1$ from Lemma \ref{l:deform} are said to be \emph{homotopic}.
This lemma is known but we will sketch a proof for completeness.

\begin{proof}[Sketch of a proof]
By the covering homotopy theorem, there is a homotopy $\phi_t$ with $\phi_0=id$ and $f_t\circ\phi_t=f_0$.
It is easy to see that $\phi_t$ are orientation preserving homeomorphisms.
Setting $g=f_t$, $f=f_0$, $\phi=\phi_t$, $\psi=id$, we see that the requirements of Definition \ref{d:Th-eqiv} are fulfilled.
\end{proof}

\subsection{Cyclic sets and pseudoaccesses}
Recall Theorem A.
We are given two quadratic Thurston maps $f$, $g$ with invariant spanning trees $T_f$, $T_g$.
There is an isomorphism $\tau:T_f\to T_g$ of ribbon graphs that conjugates $f$ with $g$ on $V(T_f)\cup C(T_f)$
  and maps critical values to critical values.
Also, $\tau$ extends to a germ of $f^{-1}(T_f)$ at each point of $C(T_f)$ so that the extension 
 still preserves the cyclic order of edges around any vertex, and still takes the dynamics of $f$ to the dynamics of $g$. 
We want to prove that $\tau$ extends to a Thurston equivalence between $f$ and $g$.

Modify the trees $T_f$, $T_g$ by adding to their vertices the critical points of $f$ belonging to $T_f$, $T_g$, respectively.
We will write $\ol T_f$, $\ol T_g$ for the modified trees.
These are also ribbon graphs.
Note that $\tau$ takes the vertices of $\ol T_f$ to the vertices of $\ol T_g$.
Moreover, $\tau$ induces an isomorphism of ribbon graphs.

\begin{figure}
  \centering
  \includegraphics[height=5cm]{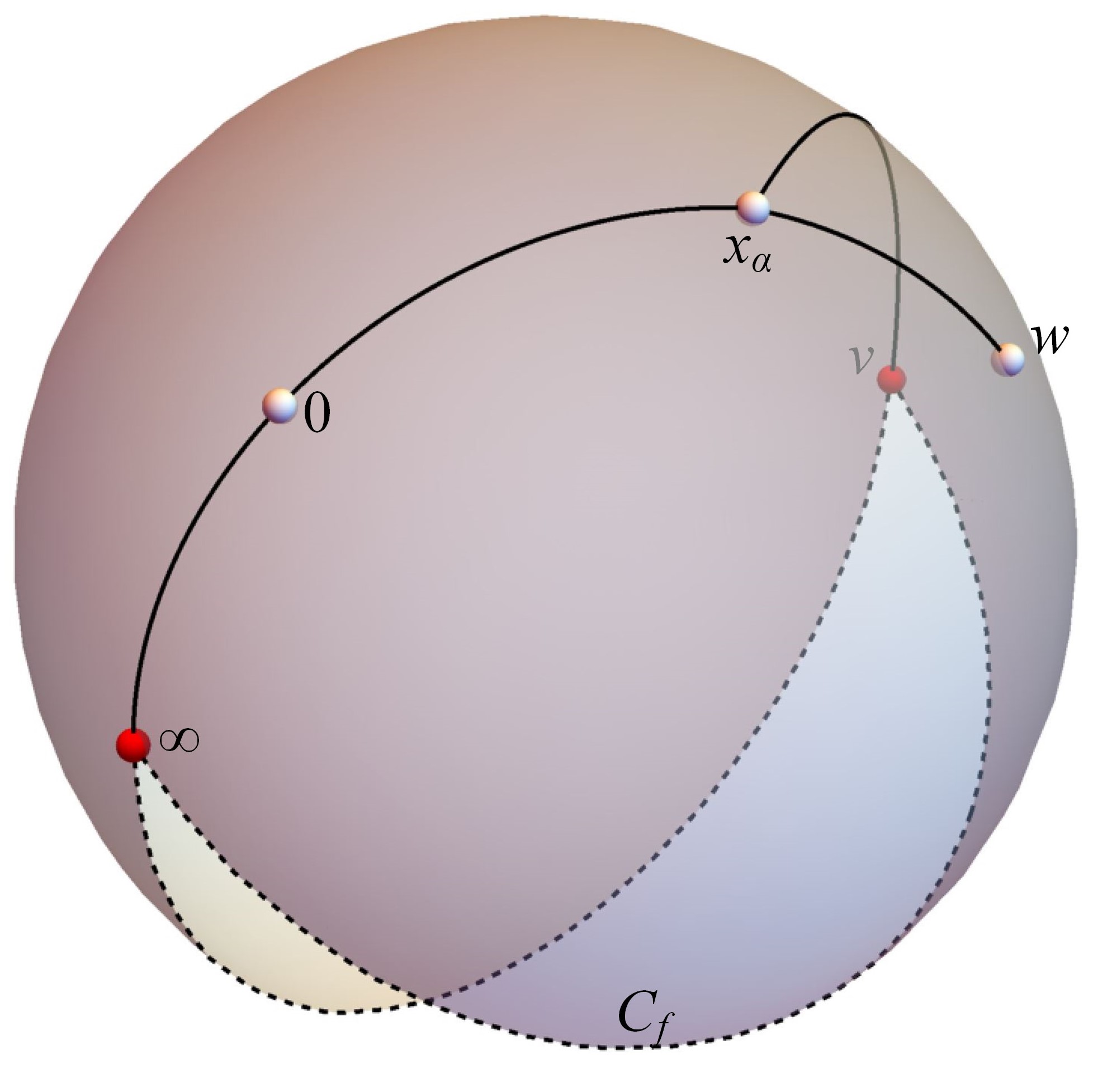}\hspace{1cm}
  \includegraphics[height=5cm]{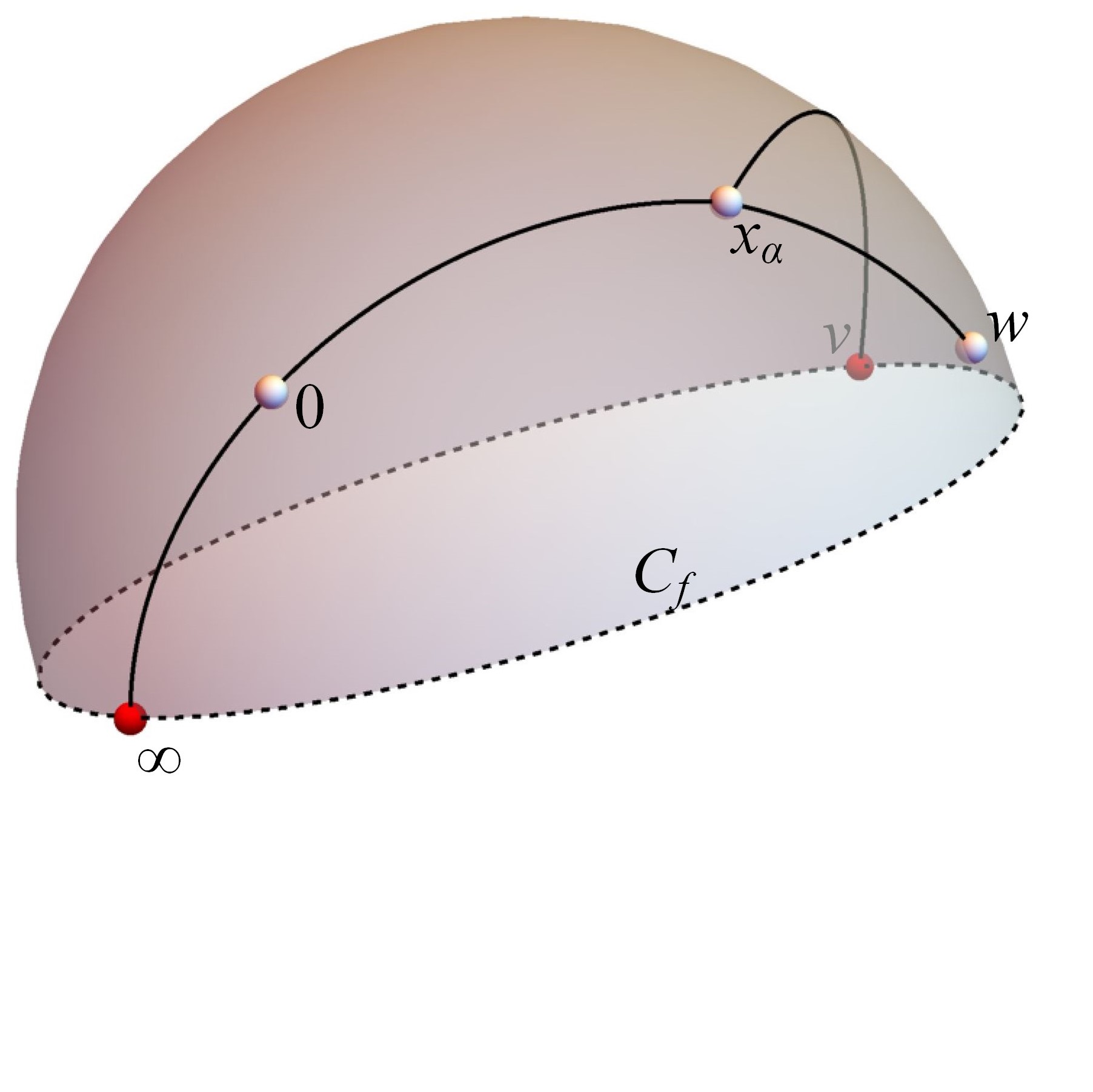}\hfill\linebreak
  \includegraphics[height=5cm]{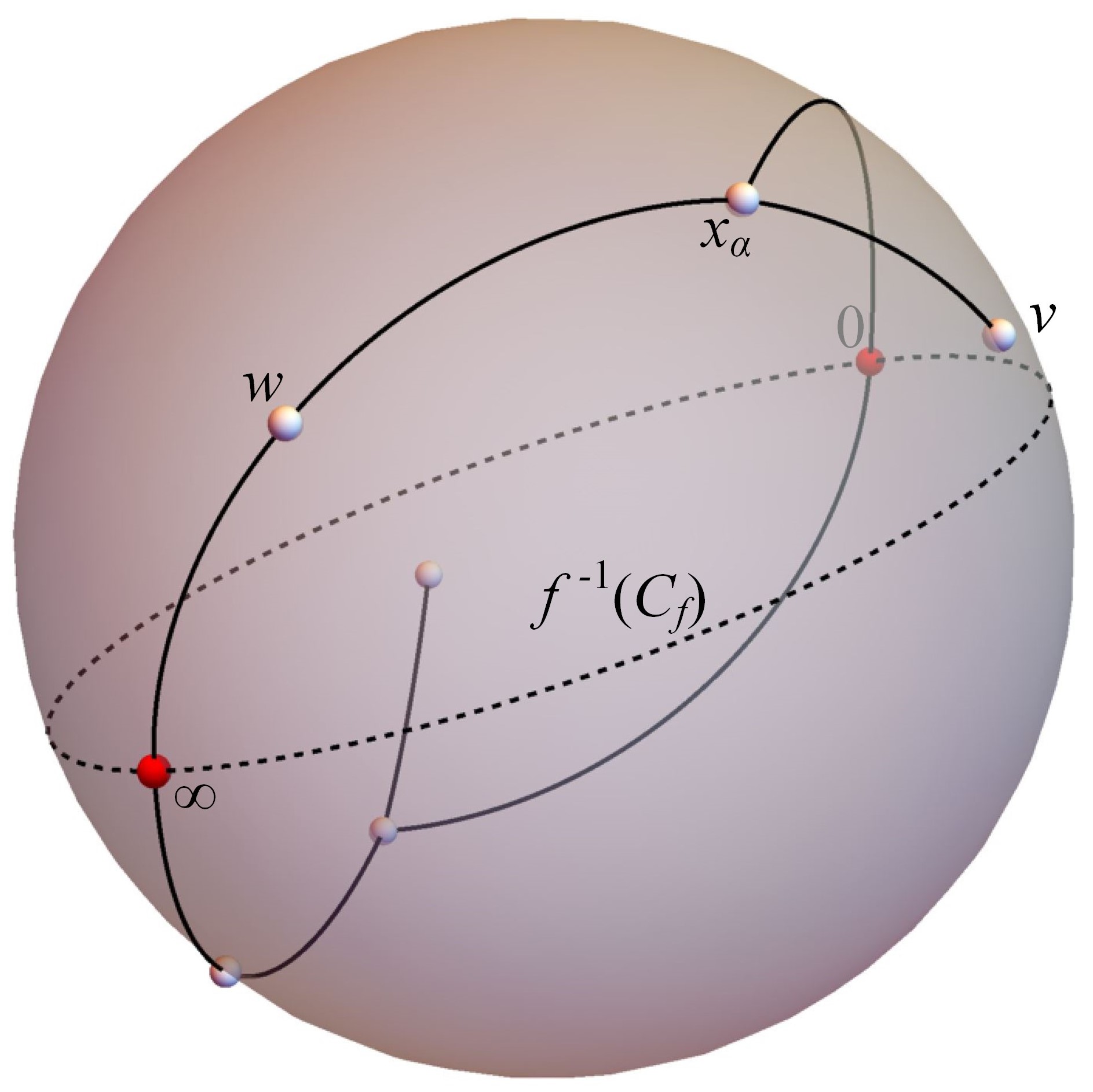}\hspace{1cm}
  \includegraphics[height=5cm]{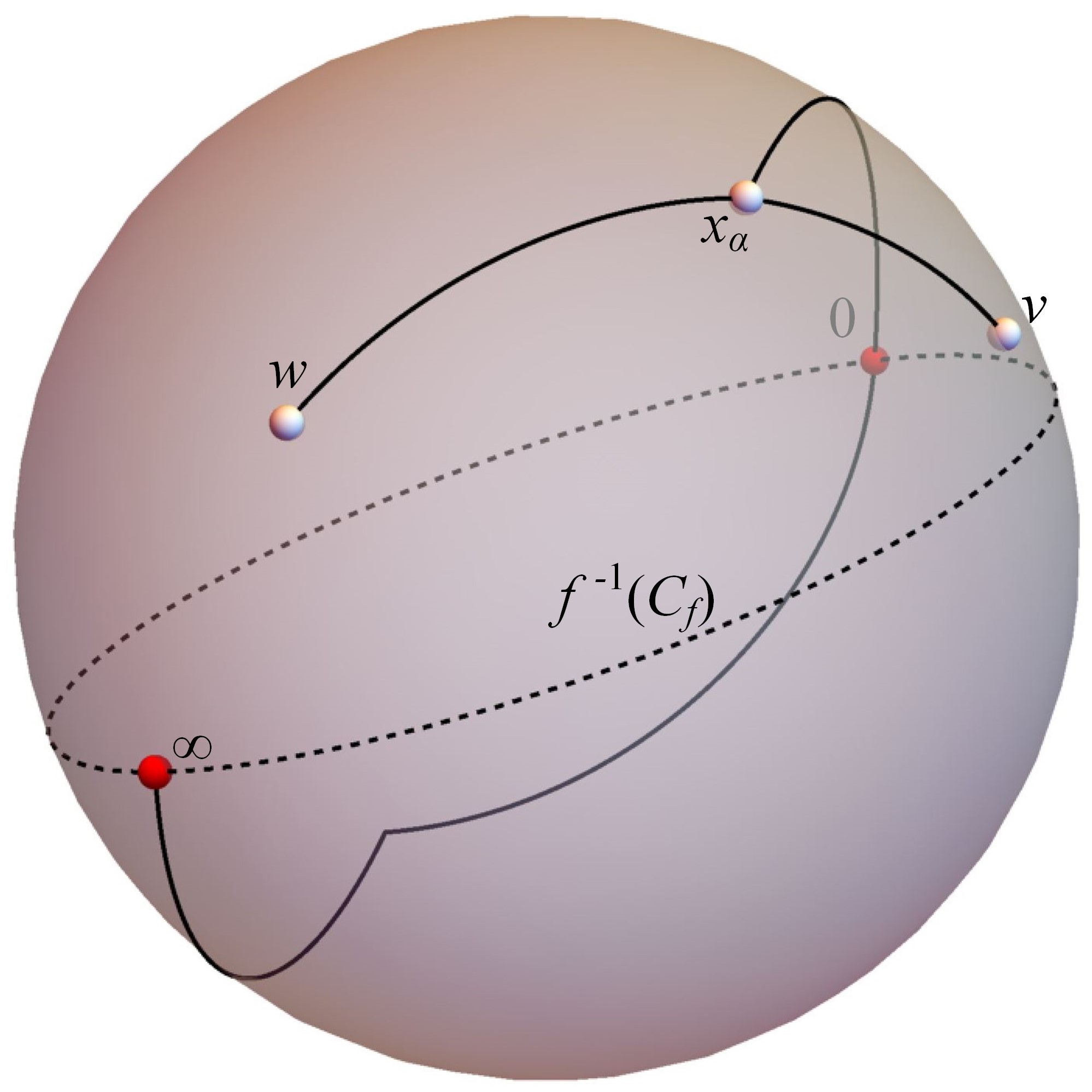}
  \caption{\footnotesize Recovering $G_f=f^{-1}(T_f)$ from $T_f$.
  Top left: we start with $T_f$ and make a cut along a simple curve $C_f$ connecting the critical values outside of $T_f$.
  As $T_f$, we took an invariant spanning tree for the rabbit polynomial $p$; we set $v=p(0)$ and $w=p(v)$.
  Top right: widening the cut, we obtain a hemisphere $U^0_f$ with a copy $T_f^0$ of $T_f$.
  Vertices of $T_f^0$ are labeled as the corresponding vertices of $T_f$.
  Bottom left: attach the opposite hemisphere $U^1_f$ with another copy $T_f^1$ of $T_f$.
  Then $G_f$ is the union of $T_f^0$ and $T_f^1$.
  This construction shows that $G_f$ is uniquely defined as a ribbon graph once $T_f$ is given
  and the critical values are distinguished among the vertices of $T_f$.
  Vertices of $G_f$ are now labeled as they appear in $G_f$
  (new labels are preimages of the former labels).
  Bottom right: we now removed the edges of $G_f$ that do not appear in $T_f$.
  What remains is a tree that identifies with $T_f$
  (one should rotate the sphere and deform the trees to attain the coincidence).
  }
  \label{fig:cut}
\end{figure}

The proof will consist of two steps.
The first step is to define a ribbon graph isomorphism between $f^{-1}(T_f)$ and $g^{-1}(T_g)$.
In other words, if just $T_f$ is given (in which the critical values are marked), then
  $f^{-1}(T_f)$ can be recovered as a ribbon graph, even without knowing $f$.
In order to recover $f^{-1}(T_f)$, a classical construction of the Riemann surface for $f^{-1}$ helps.
(This construction is essentially the same as for the Riemann surface of $z\mapsto \sqrt{z}$.)
We make a cut between two critical values of $f$, and then glue two copies of the slitted sphere along the slits.
If the cut is disjoint from $T_f$ (except the endpoints), then it suffices to see
  how copies of $T_f$ in the two slitted spheres are glued together.
To translate this process to combinatorics, we need some terminology related to cyclic sets and pseudoaccesses.
The next definition follows the terminology of \cite{poi93}.

\begin{dfn}[Pseudoaccess]
Let $A$ be a \emph{cyclic set}, i.e., a set with a distinguished cyclic order of elements.
A \emph{pseudoaccess} of $A$ is an (ordered) pair $(a,b)$ of elements of $A$ such that
 $b$ is the immediate successor of $a$ in the cyclic order.
The terminology is motivated by the following picture.
Suppose that $A$ consists of Jordan arcs in the plane that share an endpoint and are otherwise disjoint
(the cyclic order on $A$ follows the counterclockwise direction around the endpoint).
A Jordan arc disjoint from all elements of $A$ except for the same endpoint defines a pseudoaccess of $A$.
This is illustrated by the figure below, in which the pseudoaccess $(a,b)$ is represented by the dashed segment.
$$
\xymatrix{
{} \ar@{-} `/20pt[r]^c [dr] & {} & {} \ar@{-} [dl]_b \\
{} \ar@{-}[r]^d &{\bullet} \ar@{<--}[r] & {}\\
{} & {} & \ar@{-} `l[l]^a [ul] {}
}
$$
The cyclic set $A$ here is represented by the four arcs $a$, $b$, $c$, $d$ in this cyclic order.
\end{dfn}

\subsection{A homeomorphism between $f^{-1}(T_f)$ and $g^{-1}(T_g)$}
\label{ss:ext-finvT}
Consider a Thurston map $f$ of degree 2 with an invariant spanning tree $T_f$,
and a Thurston map $g$ of degree 2 with an invariant spanning tree $T_g$.
We will work under the assumptions of Theorem A.
In particular, we consider a homeomorphism $\tau:T_f\to T_g$ with the properties listed there.
The first step in the proof of Theorem A is to extend $\tau$ to $G_f=f^{-1}(T_f)$.

Note that we view $G_f$ not only as a subset of $\S^2$ but also as a graph.
Vertices of $G_f$ are defined as preimages of vertices of $T_f$.
Edges of $G_f$ are defined as components of $f^{-1}(T_f-V(T_f))$.

\begin{dfn}[Pseudoaccesses of graphs]
Let $G$ be a graph in the sphere.
A \emph{pseudoaccess} of $G$ at a vertex $a$ is defined as a pseudoaccess of $E(G,a)$.
Here $E(G,a)$ is the cyclic set of all edges of $G$ incident to $a$.
Recall that the cyclic order of edges incident to $a$ is induced by the orientation of $\S^2$.
\end{dfn}

Consider a Jordan arc $C_f$ that connects $v_1(f)$ with $v_2(f)$ and is otherwise disjoint from $T_f$,
  see Figure \ref{fig:cut}, top left.
Then $C_f$ defines two pseudoaccesses of $T_f$, one at each of the critical values.
Since $\tau$ preserves the cyclic order of edges at every vertex,
it defines a correspondence between pseudoaccesses of $T_f$ and pseudoaccesses of $T_g$.
The two pseudoaccesses of $T_f$ defined by $C_f$ give rise to two distinguished pseudoaccesses of $T_g$.
Since $\tau$ maps the critical values of $f$ to the critical values of $g$,
the two distinguished pseudoaccesses of $T_g$ are at $v_1(g)$ and $v_2(g)$.
Clearly, there exists a Jordan arc $C_g$ connecting $v_1(g)$ with $v_2(g)$,
otherwise disjoint from $T_g$ and defining the two distinguished pseudoaccesses of $T_g$.

The set $U_f=\S^2-C_f$ is a disk.
The restriction of $f$ to $f^{-1}(U_f)$ is an unbranched covering since both critical values of $f$ are in $C_f$.
Therefore, $f^{-1}(U_f)$ is a disjoint union of two open disks $U_f^0$ and $U_f^1$.
These disks are shown as hemispheres in Figure \ref{fig:cut}, bottom left.
There is an ambiguity in labeling $U_f^0$ and $U_f^1$.
One of the two disks has to be labeled $U_f^0$, and the other $U_f^1$.
However, which disk gets which label is up to us.
Similarly, $g^{-1}(U_g)$ is a disjoint union of two disks $U_g^0$ and $U_g^1$.
Again, the labeling of these disks should be specified somehow.

The common boundary of the disks $U_f^0$, $U_f^1$ is the Jordan curve $f^{-1}(C_f)$.
Consider the closure $T_f^i$ (in $\S^2$) of the $f$-pullback of $T_f-C_f$ in $U_f^i$, where $i=0,1$.
Clearly, $T_f^i$ is a tree isomorphic to $T_f$.
Moreover, $T_f^i$ and $T_f$ have isomorphic ribbon graph structures.
We may view $T_f^0$ and $T_f^1$ as two copies of $T_f$.
Observe that these two copies are glued at the critical points of $f$ to form the graph $G_f=f^{-1}(T_f)$.
Observe also that the critical points of $f$ are the vertices of $T_f^i$ that correspond,
under the natural isomorphism between $T_f^i$ and $T_f$, to the critical values $v_1(f)$ and $v_2(f)$.
Thus there is an abstract description of the ribbon graph $G_f$.
It involves making two copies of $T_f$ and gluing them at the vertices corresponding to $v_1(f)$, $v_2(f)$.
See again Figure \ref{fig:cut}.
Note that the representation of $G_f$ as a union $T_f^0\cup T_f^1$ depends on the choice of $C_f$.
More precisely, it depends on the choice of the two pseudoaccesses of $T_f$.
A similar representation can be obtained for $G_g=g^{-1}(G_g)$.

\begin{lem}
\label{l:lev-corr}
Either $\tau(T_f\cap T_f^i)\subset T_g\cap T_g^i$ for every $i=0,1$ or
 $\tau(T_f\cap T_f^i)\subset T_g\cap T_g^{1-i}$ for every $i=0,1$.
\end{lem}

Lemma \ref{l:lev-corr} is a manifestation of the fact that the construction shown in Figure \ref{fig:cut} is essentially unique.
Define the \emph{label} $\ell(e)$ of an edge $e\in E(\ol T_f)$ so that $e\subset T_f^{\ell(e)}$.
Thus the label of an edge can take values $0$ or $1$.
Labels are defined on edges of $\ol T_f$ and on edges of $f^{-1}(T_f)$ but may not be well-defined on edges of $T_f$.
(Recall that $\ol T_f$ was defined above as a subdivision of $T_f$, in which critical points of $f$ in $T_f$ become vertices.)
Lemma \ref{l:lev-corr} asserts that $\tau$ either preserves all labels or reverses all labels.
We will choose the labeling of $U_g^0$, $U_g^1$ so that all labels are preserved by $\tau$.

\begin{proof}[Proof of Lemma \ref{l:lev-corr}]
We have to show that, if $\tau(e_r)\subset T_g^{\ell(e_r)}$ for some $e_r\in E(\ol T_f)$,
 then the same holds for every edge $e$ of $\ol T_f$.
In other words, $\tau$ preserves all labels.
To this end, we compare every edge $e$ with $e_r$.
The latter will be called the \emph{reference edge}.
Consider critical points of $f$ in $\ol T_f$.
Note, however, that $T_f$ does not have to contain all critical points of $f$.
Suppose that some critical point $c$ (which is necessarily $c_1(f)$ or $c_2(f)$) lies in $\ol T_f$.
Let $v=f(c)$ be the corresponding critical value.
The curve $f^{-1}(C_f)$ defines two pseudoaccesses of $\ol T_f$ at $c$, not necessarily different.
We will call these distinguished pseudoaccesses \emph{critical pseudoaccesses}.
Clearly, critical pseudoaccesses depend only on the ribbon graph structure of $\ol T_f$ and
 on the choice of $C_f$, more precisely, on the two pseudoaccesses defined by $C_f$.
The latter two pseudoaccesses will be referred to as \emph{post-critical pseudoaccesses}.
The two critical pseudoaccesses of $\ol T_f$ at $c$ may separate some pairs of edges incident to $c$.

The values of the function $\ell$ can be computed step by step,
	starting at $e_{r}$ and passing from edges to adjacent edges.
Suppose that $\ell(e)$ is known, and $e'$ shares a vertex $a$ with $e$.
If $a$ is not critical, then $\ell(e')=\ell(e)$.
If $a$ is critical, then $\ell(e')\ne \ell(e)$ if and only if $e$ and $e'$ are separated by the critical pseudoaccesses at $a$.
The just described computational description of $\ell$ follows from the observation that
  edges $e$, $e'\in E(T_f,a)$ are separated by the critical pseudoaccesses at $a$
  if and only if they are separated by $f^{-1}(C_f)$, i.e., lie in different components of $f^{-1}(U_f)$.

We now need to prove that $\ell(\tau(e))=\ell(e)$.
To this end, it is enough to observe that $\tau$ maps $C(T_f)$ to $C(T_g)$ and that $\tau$ maps
 critical pseudoaccesses of $\ol T_f$ to critical pseudoaccesses of $\ol T_g$.
Indeed, suppose that $c$ is a critical point in $\ol T_f$.
Then $v=f(c)$ is a critical value, and $\tau(v)$ is also a critical value by property $(3)$ of $\tau$
	listed in the statement of Theorem A.
On the other hand, by property $(2)$ of $\tau$, we have $\tau(v)=\tau\circ f(c)=g(\tau(c))$.
Since $g(\tau(c))$ is a critical value, and $g$ has degree two, $\tau(c)$ is a critical point.
Thus, $\tau$ maps critical points of $f$ in $T_f$ to critical points of $g$ in $T_g$. 
For $c\in C(T_g)$, the critical pseudoaccesses at $c$ map under $\tau$ to critical pseudoaccesses at $\tau(c)$.
This follows from our assumption that $\tau$ preserves the cyclic order between edges of $f^{-1}(T_f)$ incident to $c$. 
We conclude that $\tau$ preserves the labels, as desired.
\end{proof}

Recall that $G_f=f^{-1}(T_f)$ and $G_g=g^{-1}(T_g)$.
Recall also that $V(G_f)\supset V(\ol T_f)\supset V(T_f)$, and similarly for $g$.

\begin{prop}
 \label{p:ext-finvT}
 There is a ribbon graph isomorphism $\tau_*:G_f\to G_g$ with the following properties:
 \begin{enumerate}
  \item We have $\tau_*=\tau$ on $V(\ol T_f)$.
  \item We have $\tau_*\circ f=g\circ\tau_*$ on all vertices of $G_f$.
 \end{enumerate}
\end{prop}

\begin{proof}
Recall that, by Lemma \ref{l:lev-corr}, the map $\tau:T_f\to T_g$ preserves labels.
This map lifts to $T_f^i$, where $i=0,1$, by the homeomorphisms $f:T_f^i\to T_f$ and $g:T_g^i\to T_g$.
In other words, we can define a map $\tau^i:T_f^i\to T_g^i$ by the formula $\tau^i=g^{-1}_i\circ\tau\circ f$,
where $g^{-1}_i$ is the inverse of $g:T_g^i\to T_g$.
We set $\tau_*$ to be the map from $G_f$ to $G_g$, whose restriction to $T_f^i$ is $\tau^i$.
Then we need to prove that properties $(1)-(2)$ hold for $\tau_*$.

Let us first prove that $\tau_*=\tau$ on $V(\ol T_f)$.
On $V(\ol T_f)\cap T_f^i$, the map $\tau$ satisfies the property $\tau\circ f=g\circ\tau$ by the assumptions of Theorem A.
By Lemma \ref{l:lev-corr}, under $\tau$ the set $T_f\cap T_f^i$ maps to $T_g\cap T_g^i$
(note that $\ol T_f=T_f$ as sets, and hence $\ol T_f\cap T_f^i=T_f\cap T_f^i$ set-theoretically).
Therefore, we have $\tau=g^{-1}_i\circ\tau\circ f$ on $V(\ol T_f)\cap T_f^i$.
It remains to note that the right hand side coincides with the definition of $\tau^i$.

We now prove that the cyclic order of edges incident to a vertex $a_*\in V(G_f)$ is preserved by $\tau_*$.
Let $f^{-1}_i$ be the inverse of $f:T_f^i\to T_f$.
If $a_*=f^{-1}_i(a)$, where $a$ is not a critical value, then the statement is obvious since both
$f:U_f^i\to \S^2$ and $g:U_g^i\to \S^2$ preserve the orientation.
Now, if $a$ is a critical value, then the restriction of $\tau_*$ to the union of edges incident to $a_*$
is glued from the two maps $\tau^0$ and $\tau^1$.
The cyclic order of edges of $G_f$ at $a_*$ is as follows.
First come all edges of $T_f^0$ incident to $a_*$ that are mapped by $\tau^0$ in an order preserving fashion.
Then come all edges of $T_f^1$ incident to $a_*$ that are mapped by $\tau^1$ in an order preserving fashion.
It follows that $\tau_*$ preserves the cyclic order on edges of $G_f$ incident to $a_*$.

It remains to prove that $\tau_*\circ f=g\circ\tau_*$ on all vertices of $G_f$.
Indeed, let $a_*$ be a vertex of $G_f$.
Then $a_*=f^{-1}_i(a)$ for $i=0$ or $1$.
We have
$$
\tau_*\circ f(a_*)=\tau(a)=g\circ \tau^i\circ f^{-1}_i(a)=g\circ\tau_*(a_*).
$$
In the first equality, we used that $\tau^*=\tau$ on $V(T_f)$.
In the second equality, we used the definition of $\tau^i$.
\end{proof}

\subsection{An extension of $\tau$ to the sphere}
\label{ss:tau-ext}
We keep the notation of Theorem A.
Consider the homeomorphism $\tau_*:G_f\to G_g$ constructed in Proposition \ref{p:ext-finvT}.
The restriction of $\tau_*$ to $\ol T_f$ is in general different from $\tau$.
However, these two maps match on $V(\ol T_f)=V(T_f)\cup C(T_f)$.
Moreover, $\tau_*$ restricted to $T_f$ also satisfies assumptions $(1)$--$(3)$ of Theorem A.
Thus we may consider $\tau_*$ in place of $\tau$.

We will now extend $\tau_*$ to the entire sphere.
Such an extension is possible due to the following result.

\begin{thm}[Corollary 6.6 of \cite{BBH92}]
\label{t:graph-ext}
 Let $G$ and $G'$ be two connected graphs embedded into $\S^2$.
 Consider a homeomorphism $h:G\to G'$ that induces an isomorphism of ribbon graphs.
 Then there is an orientation preserving homeomorphism $h_*:\S^2\to\S^2$ whose restriction to $G$ is $h$.
\end{thm}

Applying Theorem \ref{t:graph-ext} to our specific situation, we obtain the following corollary.

\begin{cor}
 \label{c:tau-ext}
 Suppose that $\tau_*:G_f\to G_g$ satisfies the properties listed in Proposition \ref{p:ext-finvT}.
 Then $\tau_*$ extends to an orientation preserving homeomorphism $\tau_*:\S^2\to\S^2$.
\end{cor}

The homeomorphism $\tau_*$ maps complementary components of $G_f$ to complementary components of $G_g$.
The following notion helps to say which components are mapped to which components in combinatorial terms:

\begin{dfn}[Boundary Circuits]
Let $G$ be a graph in $\S^2$.
If an orientation of an edge $e\in E(G)$ is fixed, then $e$ is called an \emph{oriented edge} of $G$.
The endpoints of $e$ form an ordered pair $(a,b)$, where $a$ is the \emph{initial endpoint} and
$b$ is the \emph{terminal endpoint} of $e$.
We also say that $e$ \emph{originates} at $a$ and \emph{terminates} at $b$.
The same edge equipped with different orientations gives rise to two different oriented edges.
A \emph{boundary circuit} of $G$ (also known as a left-turn path in $G$) is a cyclically ordered sequence $[e_0,\dots,e_{n-1}]$ of oriented edges of $G$ with the following property: if $e_i$ terminates at a vertex $a$, then $e_{i+1\pmod n}$ originates at $a$, and $e_{i+1\pmod n}$ is the immediate \textbf{predecessor} of $e_i$ in the cyclic order on $E(G,a)$.
Clearly, any oriented edge belongs to some boundary circuit.
\end{dfn}

As above, let $W$ be some complementary component of $G$.
There is a boundary circuit $\Sigma_W=[e_0,\dots,e_{n-1}]$ associated with $W$.
Informally, it is obtained by tracing the boundary of $W$ counterclockwise.
The correspondence $W\mapsto \Sigma_W$ between components of $\S^2-G$ and boundary circuits of $G$ is one-to-one.
Observe that the same edge may enter $\Sigma_W$ twice with different orientations.
Observe also that the rotation from $e_k$ to $e_{k+1}$ around the terminal point of $e_k$ is clockwise.

\subsection{Homotopy}
Theorem A will be deduced from Theorem \ref{t:ThmA-def} stated below.
Theorem \ref{t:ThmA-def} is not new:
Proposition 3.4.3 of \cite{H17} contains a more general fact; it is based in turn on a similar statement from \cite{BM17}.
However, since notation and terminology in \cite{BM17,H17} are somewhat different, we sketch a proof here.
The proof will be based on a technical lemma from \cite{BBH92}.

\begin{thm}
 \label{t:ThmA-def}
 Suppose that two Thurston maps $f$ and $g$ of degree two share an invariant spanning tree $T$.
 Moreover, suppose that $f^{-1}(T)=g^{-1}(T)=G$, that $f=g$ on $V(G)$, and that
  the critical values of $f$ coincide with the critical values of $g$.
 Then there is an orientation preserving homeomorphism $\psi$ isotopic to the identity relative to $V(T)$
 and such that $f=g\circ\psi$.
\end{thm}

Note that the equality $f^{-1}(T)=g^{-1}(T)$ means the equality of graphs rather than just sets.
In particular, we assume that the two graphs have the same vertices.
Note also that all critical points of $f$ are among vertices of these graphs.
Theorem \ref{t:ThmA-def} implies that $f$ and $g$ are Thurston equivalent.
In fact, they are even homotopic.

\begin{proof}[Proof of Theorem A using Theorem \ref{t:ThmA-def}]
Let $f$, $g$ and $\tau:T_f\to T_g$ be as in Theorem A.
As before, set $G_f=f^{-1}(T_f)$ and $G_g=g^{-1}(T_g)$.
By Proposition \ref{p:ext-finvT}, there is a homeomorphism $\tau_*:G_f\to G_g$
 that induces an isomorphism of ribbon trees and is such that
\begin{enumerate}
 \item we have $\tau_*=\tau$ on $V(T_f)$;
 \item we have $g\circ\tau_*=\tau_*\circ f$ on all vertices of $G_f$.
\end{enumerate}
Replacing $\tau$ with $\tau_*$ if necessary, we may assume that $\tau$ satisfies these properties.
In particular, $\tau$ maps $G_f$ to $G_g$.

By Corollary \ref{c:tau-ext}, the map $\tau$ extends to an orientation preserving homeomorphism $\tau:\S^2\to\S^2$.
Set $g_*=\tau^{-1}\circ g\circ\tau$.
Clearly, this is a Thurston map of degree two.
Then $T_f$ is an invariant spanning tree for $g_*$.
Since $\tau$ maps the critical values of $f$ to the critical values of $g$,
  the maps $f$ and $g_*$ share the critical values.
Finally,
$$
g_*^{-1}(T_f)=\tau^{-1}\circ g^{-1}\circ\tau(T_f)=\tau^{-1}(G_g)=G_f=f^{-1}(T_f).
$$
Thus all assumptions of Theorem \ref{t:ThmA-def} hold for $f$ and $g_*$.
By Theorem \ref{t:ThmA-def}, the map $g_*$ is
homotopic to $f$.
Since $g$ is topologically conjugate to $g_*$, we conclude that $g$ is Thurston equivalent to $f$.
\end{proof}

Consider two graphs $G$ and $T$ in the sphere.
Let $h:G\to T$ be a continuous map that is injective on the edges of $G$ and
is such that the forward and inverse images of the vertices are vertices.
Such a map is called a \emph{graph map} in \cite{BBH92}.
Suppose that a graph map $h$ has an extension $\ol h:\S^2\to\S^2$.
If $\ol h$ is an orientation preserving branched covering injective on every complementary component of $G$, then $\ol h$ is called a \emph{regular extension} of $h$.
This terminology also follows \cite{BBH92}.

\begin{thm}[Corollary 6.3 of \cite{BBH92}]
  \label{t:ext-hom}
  Consider two graph maps $h$, $h':G\to T$ admitting regular extensions $\ol h$, $\ol h'$.
  Suppose that $h=h'$ on $V(G)$ and $h(e)=h'(e)$ for every $e\in E(G)$.
  Then there is a homeomorphism $\psi:\S^2\to\S^2$ such that $\ol h=\ol h'\circ\psi$, and
  $\psi$ is isotopic to the identity relative to $V(G)$.
\end{thm}

We are now ready to deduce Theorem \ref{t:ThmA-def} from Theorem \ref{t:ext-hom}.

\begin{proof}[Proof of Theorem \ref{t:ThmA-def}]
Apply Theorem \ref{t:ext-hom} to $h=f:G\to T$ and $h'=g:G\to T$.
These are clearly graph maps admitting regular extensions.
All assumptions of Theorem \ref{t:ext-hom} are satisfied.
It follows that there exists a homeomorphism $\psi:\S^2\to\S^2$ such that $f=g\circ\psi$ on $\S^2$,
and $\psi$ is isotopic to the identity relative to $V(G)$.
\end{proof}

Thus we proved Theorem \ref{t:ThmA-def}, and the latter implies Theorem A.

\section{No dynamics: spanning trees}
In this section, we associate certain combinatorial objects with a spanning tree.
Recall that, given a finite set $P$ of (marked) points in $\S^2$, a \emph{spanning tree} for $P$ is a tree $T\subset\S^2$
 with the property that $V(T)=P\cup B(T)$, where $B(T)$ is the set of branch points of $T$.
Thus the notion of a spanning tree is an non-dynamical notion.
Suppose that the sphere $\S^2$ is glued of a polygon $\Delta$ by identifying some edges of it.
Then the boundary of $\Delta$ becomes a spanning tree for the set $P$ of all vertices of $\Delta$.
Alternatively, some vertices of $\Delta$ can be dropped from $P$ if these give rise to branch points of the tree.

\subsection{The generating set $\Ec_T$ of $\pi_1(\S^2-P)$}
\label{ss:Ec}
Let $T$ be a spanning tree for a finite marked set $P$.
Assume that the base point $y\in\S^2-T$ is fixed once and for all.
We now define a certain generating set $\Ec=\Ec_T$ of $\pi_1(\S^2-P)=\pi_1(\S^2-P,y)$.
(This is the same generating set as in \cite{H17}; Hlushchanka refers to its elements as \emph{edge generators}).

Endow $\S^2$ with some smooth structure.
(It will be clear however that our construction is independent of this structure).
Consider an oriented smooth Jordan arc $A$.
Let $\gamma$ be a smooth path that crosses $A$ only once and transversely.
By a transverse intersection we mean that the tangent lines to $A$ and $\gamma$ at the intersection point are different, and that the intersection point is not an endpoint of $A$.
We say that $\gamma$ approaches $A$ \emph{from the left} if, at the intersection point, the velocity vectors to $\gamma$ and to $A$ (in this order) form a positively oriented basis in the tangent plane to the sphere.
With every oriented edge $e$ of $T$, we associate an element $g_e\in\pi$ as follows.
The homotopy class $g_e$ is represented by a smooth loop $\gamma_e$ that crosses $e$ just once and transversely,
  approaches it from the left, and has no other intersection points with $T$.
(We assume of course that the loop $\gamma_e$ is based at $y$).
A smooth loop $\gamma_e$ with the indicated properties is said to be \emph{adapted} to $T$ at $e$.
Consider the subset $\Ec=\Ec_T\subset\pi_1(\S^2-P)$ consisting of $id$, the neutral element, and elements $g_e$,
  where $e$ ranges through all oriented edges of $T$.
Note that the same edge equipped with different orientations gives rise to two different elements of $\Ec$.
These elements are inverse to each other.

Thus $\Ec$ is a generating set of $\pi_1(\S^2-P)$ that is symmetric ($\Ec^{-1}=\Ec$) and such that $id\in\Ec$.

\begin{lem}
  \label{l:th-diff}
Different oriented edges of $T$ give rise to different elements of $\Ec$.
\end{lem}

\begin{proof}
Consider two different oriented edges $e_1$, $e_2$ of $T$ with $g_{e_1}=g_{e_2}$.
Let $\gamma_{e_i}:[0,1]\to\S^2$ be smooth simple loops as above so that $g_{e_i}=[\gamma_{e_i}]$.
We can also arrange that $\gamma_{e_1}(0,1)$ is disjoint from $\gamma_{e_2}(0,1)$.
Let $D$ be the croissant shaped region bounded by $\gamma_{e_1}[0,1]$ and $\gamma_{e_2}[0,1]$.

Since $g_{e_1}=g_{e_2}$, the loops $\gamma_{e_i}$ are homotopic rel. $P$.
Therefore, there are no points of $P$ in $D$.
We claim that there are also no branch points of $T$ in $D$.
Indeed, if $x$ is such point, then there is a component of $T-\{x\}$ disjoint from both $e_1$ and $e_2$.
This component must end somewhere in $D$.
On the other hand, by definition of a spanning tree, all endpoints of $T$ are in $P$.
A contradiction with the fact that $P\cap D=\0$.

Since $e_1\ne e_2$, there is at least one vertex $x$ of $T$ in $D$.
However, this is impossible since all vertices of $T$ are in $P\cup B(T)$.
\end{proof}

Consider a set $E$ of smooth loops based at $y$ with the following properties.
Firstly, we assume that every $\gamma\in E$ is adapted to $T$ at some oriented edge $e$ of $T$.
Secondly, there is exactly one loop $\gamma_e\in E$ adapted to $T$ at $e$, and,
 as we change the orientation of $e$, the corresponding loop also changes orientation but otherwise remains the same.
Thirdly, we assume that the constant loop belongs to $E$, and that
 different loops from $E$ are either disjoint (except the common basepoint $y$) or the same (up to the change of direction).
If these assumptions are satisfied, then we say that $E$ is an \emph{adapted set of loops} for $T$.
Clearly, any spanning tree admits an adapted set of loops.
The set $\Ec_T$ equals $[E]$, the set of classes in $\pi_1(\S^2-P)$ of all elements from $E$.

The following lemma will help us translate the pullback operations on spanning trees into a combinatorial language.
We will assume that the basepoint $y$ is chosen outside of all spanning trees under consideration.

\begin{lem}
  \label{l:sameE}
Let $H$ be an inner automorphism of $\pi_1(\S^2-P)$.
Suppose that two spanning trees $T$ and $T'$ are such that $\Ec_T=H(\Ec_{T'})$ in $\pi_1(\S^2-P)$.
Then $T$ and $T'$ are isotopic rel. $P$.
\end{lem}

\begin{proof}
Let $h$ be an element of $\pi_1(\S^2-P)$ such that $H$ is the conjugation by $h$.
We will write $\pmcg(\S^2,P)$ for the pure mapping class group of $\S^2$ with marked point set $P$.
Consider the homomorphism $\push:\pi_1(\S^2-P,y)\to\pmcg(\S^2,P\cup\{y\})$ from the Birman exact sequence
(cf. Section 4.2.1 of \cite{FM12}).
It is easy to see that $\psi=\push(h)$ acts on $\pi_1(\S^2-P,y)$ as $H$.
Moreover, the Birman exact sequence implies that $\psi$ is isotopic to the identity rel. $P$ but not rel. $P\cup\{y\}$.
Replacing $T'$ with $\psi^{-1}(T')$, we can arrange that $\Ec_T$ and $\Ec_{T'}$ coincide.
Thus we will assume from now on that $\Ec_T=\Ec_{T'}$.

We may assume that both $T$ and $T'$ are composed of smooth arcs.
Suppose that sets $E$, $E'$ of smooth loops based at $y$ are adapted to $T$, $T'$, respectively.
The sets $E$, $E'$ form embedded graphs $\Gamma$, $\Gamma'$, respectively, in $\S^2$ with the single vertex $y$.
Every complementary component (``face'') of $\Gamma$ contains a single vertex of $T$, and similarly for $\Gamma'$.
There is a homeomorphism $\phi:\Gamma\to\Gamma'$ that is simultaneously a graph map.
Moreover, for every edge of $\Gamma$, there is an isotopy transforming this edge to its $\phi$-image.
(Indeed, two loops are homotopic rel. $P$ if and only if they are isotopic rel. $P$.)
Then $\phi$ can be extended as an orientation preserving homeomorphism $\phi:\S^2\to\S^2$ fixing $P$ pointwise.
This follows from Lemma 2.9 of \cite{FM12}.
Moreover, it follows from the same lemma that $\phi$ is isotopic to the identity rel. $P$.
Applying $\phi^{-1}$ to $T'$ and $E'$, we may now assume that $E=E'$.
The corresponding edges of $T$ and $T'$ connect the same complementary components of $\Gamma$ and
 cross the same edge of $\Gamma$.
 It follows that the corresponding edges of $T$ and $T'$ are homotopic rel. $P$, as desired.
\end{proof}

The converse of Lemma \ref{l:sameE} is also true.
We say that two subsets $\Ec$ and $\Ec'$ of a group $\pi$ are \emph{conjugate} if there is $u\in\pi$ such that
 $\Ec'$ coincides with the set of all elements of the form $uvu^{-1}$, where $v$ runs through $\Ec$.

\begin{prop}
  \label{p:trconj}
  Let $T$ and $T'$ be spanning trees for $P$.
The trees $T$ and $T'$ are homotopic rel. $P$ if and only if the corresponding generating sets $\Ec_T$ and $\Ec_{T'}$ are conjugate.
\end{prop}


\begin{proof}[Proof of Proposition \ref{p:trconj}]
To lighten the notation, we will write $\Ec$ and $\Ec'$ instead of $\Ec_T$ and $\Ec_{T'}$.
We silently assumed that both $\Ec$ and $\Ec'$ are subsets of the same group $\pi=\pi_1(\S^2-P,y)$
 corresponding to a certain basepoint $y$.
Thus the basepoint is fixed.

Suppose first that $\Ec$ and $\Ec'$ are conjugate.
Then $T$ and $T'$ are homotopic rel. $P$, by Lemma \ref{l:sameE}.
Suppose now that $T$ and $T'$ are homotopic.
We may assume that both $T$, $T'$ are formed by smooth arcs.
Let $E$ be a set of smooth loops adapted to $T$.

Now consider a homotopy $T_t$ of $T$ (so that $T_0=T$, $T_1=T'$, and $t$ runs through $[0,1]$).
We may assume that this homotopy is smooth.
Then there is a homotopy $\phi_t:\S^2\to\S^2$ consisting of orientation preserving diffeomorphisms
 such that $\phi_0=id$ and $\phi_t(T)=T_t$.
Clearly, $E_t=\phi_t(E)$ is adapted to $T_t$.
In particular, $E_t$ represents the symmetric generating set $\Ec_t=\Ec_{T_t}$ in $\pi_1(\S^2-P,\phi_t(y))$.

Recall that any homotopy class $c$ of paths connecting two given points $y$, $y'\in\S^2-P$
 gives rise to an isomorphism $H_c:\pi_1(\S^2-P,y)\to \pi_1(\S^2-P,y')$.
Two different isomorphisms of this type differ by an inner automorphism of the target group.
All groups $\pi_1(\S^2-P,\phi_t(y))$ can be identified along the path $t\mapsto \phi_t(y)$.
In particular, $\pi_1(\S^2-P,\phi_1(y))$ identifies with $\pi$.

Modifying the homotopy if necessary, we may arrange that $\phi_1(y)=y$.
Thus, $\Ec_1$ and $\Ec'$ lie in the same group, and, by definition of $\Ec'$, we must have $\Ec'=\Ec_1$.
On the other hand, $\Ec_1$ identifies with $\Ec_0=\Ec$ under the automorphism $H_c$,
 where $c$ is the homotopy class of the loop $t\mapsto \phi_t(y)$.
Since $H_c$ is an inner automorphism, $\Ec$ and $\Ec'$ are conjugate.
Thus the proposition is proved.
\end{proof}

\subsection{Vertex structures}
\label{ss:vertstruc}
Below, we will introduce some formal algebraic/combinatorial notions.
The purpose of these is to translate topological objects, namely, spanning trees, into a symbolic language.

For any finite set $\Ec$, we write $\mathrm{FS}(\Ec)$ for the free semi-group generated by $\Ec$.
The semi-group $\mathrm{FS}(\Ec)$ can also be thought of as the set of all finite words in the alphabet $\Ec$.
The empty word is allowed as an element of $\mathrm{FS}(\Ec)$; it is the neutral element of the semi-group.
For $g$, $h\in \Ec$, the product of $g$ and $h$ in $\mathrm{FS}(\Ec)$ will be written as $g\cdot h$.

Suppose now that $\pi$ is a group and that $\Ec\subset\pi$.
We also suppose that $id\in \Ec$.
Here $id$ means the identity element of $\pi$.
It is not to be confused with the neutral element of $\mathrm{FS}(\Ec)$, which is not an element of $\Ec$ or of $\pi$.
We set $\Ec^\star$ to be the quotient of $\mathrm{FS}(\Ec)$ modulo the relations $id\cdot g=g\cdot id=g$ for all $g\in \Ec$.
Now assume 
 that $\Ec\ni id$ is symmetric, i.e., that $g\in\Ec$ implies $g^{-1}\in\Ec$.
Here $g^{-1}$ is the inverse of $g$ in the group $\pi$.
Then there is a natural map $\Pi:\Ec^\star\to\pi$ that takes every word in the alphabet $\Ec$ to the product of its symbols.
(The latter product is with respect to the group operation in $\pi$.)
We will refer to $\Pi$ as the \emph{evaluation map}.
For example, an element $g_1\cdot g_2\in\Ec^\star$ is mapped to $g_1g_2\in\pi$.
Intuitively, an element $u\in\Ec^\star$ is a way of writing the element $\Pi(u)$ of the subgroup of $\pi$ generated by $\Ec$
 as a product of generators.
Different ways of writing the same element may differ by a sequence of cancellations.
However, we disregard all appearances of $id$.
For example, $g\cdot h\cdot g\cdot g^{-1}$ is different from $g\cdot h$ as an element of $\Ec^\star$.
However, it is the same as $g\cdot id\cdot h\cdot g\cdot id\cdot g^{-1}\cdot id\cdot id$, for example.

A \emph{vertex structure} on $\Ec$ is a subset $\Vc\subset\Ec^\star$ with the following property:
 for every $g\in\Ec$, there is a unique element of $\Vc$ of the form $u_1\cdot g\cdot u_2$ for some $u_1$, $u_2\in\Ec^\star$.
Any vertex structure gives rise to an abstract directed graph $G(\Vc)$ as follows.
The vertices of $G(\Vc)$ are identified with elements of $\Vc$.
The oriented edges of $G(\Vc)$ are labeled by elements of $\Ec$.
Two vertices $v$, $w\in\Vc$ are connected with an oriented edge $g$ (from $v$ to $w$) if
$$
v=v_1\cdot g\cdot v_2,\quad w=w_1\cdot g^{-1}\cdot w_2
$$
for some elements $v_1$, $v_2$, $w_1$, $w_2$ of $\Ec^\star$.
Since $\Ec$ is symmetric, the edges of $G(\Vc)$ always come in pairs so that
  paired edges connect the same vertices but go in different directions.
These pairs of edges correspond to pairs of the form $\{g,g^{-1}\}$ in $\Ec$.
Thus $G(\Vc)$ can also be regarded as an undirected graph,
  by identifying each pair of oppositely directed edges with an undirected edge.
A vertex structure $\Vc$ on $\Ec$ is called a \emph{tree structure} if $G(\Vc)$ is a tree.

Observe that the graph $G(\Vc)$ also carries a natural ribbon graph structure.
Indeed, directed edges of $G(\Vc)$ originating at a given vertex $v\in\Ec^\star$ are linearly ordered.
We refer to the linear order of symbols in words from $\Ec^\star$.
For example, consider a vertex represented by $a\cdot b\cdot c\in\Ec^\star$ with $a$, $b$, $c\in\Ec$.
Then we should think of $a$, $b$, $c$ as appearing in this \textbf{clockwise} order around the given vertex.
That is, the cyclic order of $a$, $b$, $c$ at the given vertex is $[c,b,a]$.

\subsection{Vertex words}
In this section, we explain how a spanning tree $T$ for a finite marked set $P$ defines a tree structure on $\Ec=\Ec_T$.
To this end, we need to equip $T$ with a bit of extra structure.
Namely, we assume that some pseudoaccess is fixed at every vertex of $T$.

Recall that any oriented edge $e$ of $T$ gives rise to a group element (edge generator) $g_e\in\Ec$.
Moreover, by Lemma \ref{l:th-diff}, different edges correspond to different edge generators.
Thus we may think of $\Ec$ as a combinatorial analog for the set of oriented edges of $T$.
We now define a combinatorial analog of a vertex.

\begin{dfn}[Vertex word]
  Let $x$ be a vertex of $T$.
Consider all edges $e_0$, $\dots$, $e_{k-1}$ incident to $x$ and oriented outwards.
The linear order of these edges is well defined if we impose that
\begin{enumerate}
  \item it follows the natural \textbf{clockwise} order around $x$;
  \item the chosen pseudoaccess at $x$ coincides with $(e_{k-1},e_0)$.
\end{enumerate}
Then we define the \emph{vertex word} of $x$ as the product $g_{e_0}\cdot \dots \cdot g_{e_{k-1}}\in\Ec^\star$.
For example, if $k=3$, then $x=g_{e_0}\cdot g_{e_1}\cdot g_{e_2}$
(the product is in $\Ec^\star$, not in $\pi_1(\S^2-P)$!).
Let $\Vc$ be the set of all vertex words associated with the vertices of $T$.
Then $\Vc$ is clearly a tree structure on $\Ec$ such that $G(\Vc)$ is isomorphic to $T$ as a ribbon graph.
\end{dfn}

The construction presented above may seem artificial.
In order to shed some light on it, let us consider an example.
The following is a spanning tree $T$ for a set of three marked points:
$$
\xymatrix{
& {\circ} \ar@{<-}[d]^B& \\
{\circ} \ar@{<-}[r]^A &
{\cdot} \ar@{->}[r]^C &
{\circ}
}
$$
(The marked points, shown as circles, are precisely the endpoints of the tree.)
We write $A$, $B$, $C$ for the oriented edges of $T$ originating at the branch point.
Set $a=g_A$, $b=g_B$, $c=g_C$.
Then the generating set $\Ec_T$ consists of 7 elements $id$, $a$, $a^{-1}$, $b$, $b^{-1}$, $c$, $c^{-1}$.
The vertex word corresponding to the branch point of the tree is $a\cdot b\cdot c$.
Note that this word is different from the neutral element of $\Ec^\star$ even through
 $\Pi(a\cdot b\cdot c)=abc=id$ in $\pi$.
This example explains why we need to consider $\Ec^\star$.
The vertex structure associated with $T$ is
$$
\Vc=\{a\cdot b\cdot c,a^{-1},b^{-1},c^{-1}\}.
$$
Clearly, the combinatorial structure of $G(\Vc)$ represents that of $T$.

\section{Dynamics: computation of the biset}
\label{s:img}
In this section, we consider a Thurston map $f$ of degree two with an invariant spanning tree $T$.
We will find a presentation for the biset of $f$ using only the combinatorics of the map $f:T\to T$.
We start with recalling the terminology.

\subsection{Bisets and automata}
A biset is a convenient algebraic invariant of a Thurston map,
  which fully encodes the Thurston equivalence class.

Fix some basepoint $y\in\S^2-P(f)$.
Define the set $\Xc_f(y)$ as the set of all homotopy classes of paths from $y$ to $f^{-1}(y)$ in $\S^2-P(f)$.
To lighten the notation, we will write $\pi_f$ for the fundamental group $\pi_1(\S^2-P(f),y)$.
There are natural left and right actions of $\pi_f$ on $\Xc_f(y)$.
For this reason, the set $\Xc_f(y)$ is referred to as a \emph{$\pi_f$-biset}.

The left action of $\pi_f$ on $\Xc_f(y)$ is the usual composition of paths.
Let $\gamma$ be a representative of an element $[\gamma]\in\pi_f$, and
  let $\alpha$ be a representative of an element $[\alpha]\in\Xc_f(y)$.
Then $[\gamma][\alpha]$, the left action of the element $[\gamma]\in\pi_f$ on an element $[\alpha]\in\Xc_f(y)$,
  is defined as the element $[\gamma\alpha]$ of $\Xc_f(y)$ represented by
  the composition $\gamma\alpha$ of $\gamma$ and $\alpha$: we first traverse $\gamma$, and then $\alpha$.
According to our convention, paths are composed from left to right.
The right action of $\pi_f$ on $\Xc_f(y)$ is defined as follows.
For $[\gamma]\in\pi_f$ and $[\alpha]\in\Xc_f(y)$ as above, let $\beta$ be the composition of $\alpha$ and the pullback of $\gamma$ originating at the terminal point of $\alpha$.
Then the element $[\alpha].[\gamma]\in\Xc_f(y)$, the right action of $[\gamma]$ on $[\alpha]$, is defined as $[\beta]$.
We will refer to $\Xc_f(y)$ as \emph{the biset of $f$}.
Now that we have a particular example at hand, we give a general algebraic definition of a biset.

\begin{dfn}[Biset]
  Let $\pi$ be a group.
A set $\Xc$ is called a \emph{biset over $\pi$}, or a \emph{$\pi$-biset}, if
  commuting left and right actions of $\pi$ on $\Xc$ are given.
The biset $\Xc$ is said to be \emph{left free} if there exists a subset $\Bc\subset\Xc$ such that
every element $a\in\Xc$ can be uniquely represented as $gb$, where $g\in\pi$ and $b\in\Bc$.
The subset $\Bc$ is then called a \emph{basis} of $\Xc$.
Let $\pi'$ be another group, and $\Xc'$ be a $\pi'$-biset.
A group isomorphism $\rho:\pi\to\pi'$ is said to \emph{conjugate} $\Xc$ with $\Xc'$ if
 there is a bijection $\sigma:\Xc\to\Xc'$ with the property that $\sigma(g_1a.g_2)=\rho(g_1)\sigma(a)\rho(g_2)$
 for all $g_1$, $g_2\in\pi$ and $a\in\Xc$.
If $\rho$ and $\sigma$ with these properties exist, then $\Xc$ and $\Xc'$ are said to be \emph{conjugate}.
If moreover $\pi=\pi'$ and $\rho=id$, we say that $\Xc$ and $\Xc'$ are \emph{isomorphic}.
For more details on these formal notions, we refer the reader to \cite{Nek05,BD17}
(note that bisets are called \emph{bimodules} in \cite{Nek05}, see Chapter 2).
\end{dfn}

Clearly, the biset of a Thurston map is well defined up to conjugation.
Recall the following theorem of Nekrashevich (Theorem 6.5.2 of \cite{Nek05}, see also \cite{Kam01,Pil03}),
which says that, reversely, the conjugacy class of the biset
  determines the Thurston equivalence class of the map:

\begin{thm}
  \label{t:nek}
  Let $f_1$ and $f_2$ be Thurston maps, and $\Xc_{f_i}$ be the corresponding $\pi_{f_i}$-bisets, $i=1$, $2$.
  Here $\pi_{f_i}$ is the fundamental group of $\S^2-P(f_i)$.
\begin{enumerate}
  \item The maps $f_1$ and $f_2$ are Thurston equivalent if and only if there exists
  an orientation preserving homeomorphism $h:\S^2\to\S^2$ such that $h(P(f_1))=P(f_2)$ and
  the induced isomorphism $h_*:\pi_{f_1}\to\pi_{f_2}$ conjugates $\Xc_{f_1}$ with $\Xc_{f_2}$.
  \item Suppose that $P(f_1)=P(f_2)=P$ and the base points chosen for $\Xc_{f_1}$, $\Xc_{f_2}$ coincide.
  The maps $f_1$ and $f_2$ are homotopic rel. $P$ if and only if $\Xc_{f_1}$ and $\Xc_{f_2}$ are isomorphic.
\end{enumerate}
\end{thm}

Let us go back to a degree 2 Thurston map $f$.
A basis of $\Xc_f(y)$ consists of two elements.
These are homotopy classes of two paths connecting $y$ with its preimages $y_0$, $y_1$.
Thus, to choose a basis of $\Xc_f(y)$ is the same as to choose two paths $\alpha_0$, $\alpha_1$, up to homotopy rel. $P(f)$,
so that $\alpha_\eps$ connects $y$ with $y_\eps$, for $\eps=0$, $1$.
Once some basis of $\Xc_f(y)$ is chosen, we can associate an automaton with $\Xc_f(y)$.

\begin{dfn}[Automaton]
\label{d:automaton}
 Let $A$ and $S$ be some sets.
 In practically important cases both $A$ and $S$ are finite.
 The set $A$ is called an \emph{alphabet}, and its elements are called \emph{symbols}.
 The set $S$ is called the \emph{set of states}, and its elements are called \emph{states}.
 An \emph{automaton} can be defined as a map $\Sigma:A\times S\to S\times A$, or rather as a triple $(A,S,\Sigma)$.
 Let $\fs(A)$ be the set of finite words in the alphabet $A$, including the empty word.
 This is a \textbf{f}ree \textbf{s}emi-group generated by $A$, thus the notation.
 If we fix some initial state $s_0\in S$, then we obtain a self-map of $\fs(A)$ as follows.
 Imagine that a machine reads a word $w\in \fs(A)$ symbol by symbol, \textbf{right to left}.
 Suppose, at some point, it reads a symbol $a\in A$ and its state is $s$.
 Set $(t,b)=\Sigma(a,s)$.
 Then the machine writes $b$ in place of $a$, changes the state to $t$, and moves one step left.
 In other words, an automaton $(A,S,\Sigma)$ gives rise to a right action of $S$ on $\fs(A)$.
 If $\Sigma$ is fixed, then it is common to write $\Sigma(a,s)$ simply as $as$.
\end{dfn}

Consider an abstract left free $\pi$-biset $\Xc$.
Assume that some basis $\Bc$ of $\Xc$ is chosen.
Then, for every $a\in\Bc$ and every $g\in\pi$, there are elements $a^*\in\Bc$ and $g^*\in\pi$ with $ag=g^*a^*$.
Thus, we have a well-defined map $\Sigma_\Bc:\Bc\times\pi\to\pi\times\Bc$ taking $(a,g)$ to $(g^*,a^*)$.
By definition, this is an automaton with $\pi$ being the set of states.
We will refer to this automaton as the \emph{full automaton} of $\Xc$ in the basis $\Bc$.
Clearly, the full automaton defines $\Xc$ up to isomorphism.
On the other hand, the full automaton carries excessive information.
It is enough to know the values $\Sigma_\Bc(a,g)$ for all $g$ in some generating set of $\pi$.
If $\Bc$ is finite and $\pi$ is generated by a finite set $S$, then the image of $\Bc\times S$ under $\Sigma_\Bc$ is finite.
In particular, this image lies in $S^*\times\Bc$, where $S^*$ is also a finite subset of $\pi$.
Thus, in order to describe the biset, it suffices to indicate the map
$\Sigma_{\Bc}:\Bc\times S\to S^*\times\Bc$ between finite sets.
This map is called a (finite) \emph{presentation} of $\Xc$.
We see that finitely presented bisets can be efficiently described, and computations with them are easy to implement.
However, the isomorphism problem for bisets is not easy, cf. \cite{BD17}.

We now go back to the biset $\Xc_f(y)$ of a quadratic Thurston map $f$.
In a number of important situations, there is a finite generating set $\Ec\subset \pi_f$
 and a basis $[\alpha_0]$, $[\alpha_1]$ with the following property.
For $\eps\in\{0,1\}$ and any element $a\in\Ec$, we have $[\alpha_\eps].a=a^*[\alpha_{\eps^*}]$
  for some $\eps^*\in\{0,1\}$ and $a^*\in\Ec$ depending on $a$ and $\eps$.
Define an automaton $\Sigma:\{0,1\}\times\Ec\to\Ec\times\{0,1\}$ taking $(\eps,a)$ to $(a^*,\eps^*)$.
This automaton has then a finite set of states.
Such automata are practically important and are called \emph{finite state automata}.
Observe that $\Sigma$ defines a finite presentation of $\Xc_f(y)$.
We will see that a simple presentation of $\Xc_f(y)$ by a finite state automaton can be associated with
 every invariant spanning tree of $f$.
This observation was also made in \cite{H17} in a more general context but with a less explicit description of the automaton.

\subsection{A base edge and labels}
\label{ss:baseedge}
We now assume that $(T^*,T)$ is a dynamical tree pair for $f$.
Let $Z$ be the smallest subarc of $T$ containing both $v_1$ and $v_2$.
(In Figure \ref{fig:cut}, top left, this is the union of the arcs $\infty 0$, $0x_\alpha$, and $x_\alpha v$.)
Then $f^{-1}(Z)$ is a Jordan curve containing the critical points $c_1$ and $c_2$.
(In Figure \ref{fig:cut}, bottom left, this is the only simple cycle in the graph.)
We will regard both $Z$ and $f^{-1}(Z)$ as graphs in the sphere
whose vertices are the vertices of $T$ and $f^{-1}(T)$, respectively, contained in $Z$ and $f^{-1}(Z)$, respectively.
Since the tree $T^*$ cannot contain the Jordan curve $f^{-1}(Z)$,
there is at least one edge $e'_b$ of $f^{-1}(Z)$ not contained in $T^*$.
(In Figure \ref{fig:cut}, we removed an edge of $f^{-1}(Z)$ when passing from the bottom left to the bottom right picture.
We may set $e'_b$ to be this removed edge.)
Choose one such edge, and call $e_b=f(e'_b)$ the \emph{base edge} of $T$.
There may be several ways of choosing a base edge.

The two arcs with endpoints $c_1$, $c_2$ mapping onto $Z$ will be denoted by $Z^0$ and $Z^1$.
Here $Z^1$ is chosen to include $e'_b$.
Then $Z^0$ includes the other pullback of $e_b$.

Set $G=f^{-1}(T)$.
Suppose now that some post-critical pseudoaccesses (i.e., pseudoaccesses at critical values) are chosen for $T$.
Intermediate steps in the computation of an automaton for $\Xc_f(y)$, but not the final result, will depend on this choice.
The choice of the post-critical pseudoaccesses gives rise to a representation $G=T^0\cup T^1$.
Here $T^0$, $T^1$ are two trees mapping homeomorphically onto $T$ under $f$.
In Section \ref{ss:ext-finvT}, we defined $T^0$ and $T^1$ using a Jordan arc $C$ connecting $v_1$ with $v_2$ outside of $T$.
However, it is easy to see that $T^i$ depend only on the pseudoaccesses defined by $C$.
To fix the labeling, we assume that $Z^i\subset T^i$ for $i=0,1$.
In fact, $Z^i$ is an ``invariant'' part of $T^i$, independent of the choice of the pseudoaccesses.

Modify $T^*$ so that the critical points of $f$ lying in $T^*$ become vertices.
To distinguished the new (modified) tree from $T^*$, we denote it by $\ol T^*$.
Let $e$ be an edge of $\ol T^*$.
Then $e$ lies in $T^{\ell(e)}$, where $\ell(e)=0$ or $1$.
The number $\ell(e)$ is called the \emph{label} of $e$, cf. the proof of Lemma \ref{l:lev-corr}.
We now reproduce the combinatorial definition of labels.

\begin{dfn}[The label of an edge]
\label{d:level}
Define the \emph{critical pseudoaccesses} of $G=f^{-1}(T)$ as the preimages of the post-critical pseudoaccesses of $T$.
There is a unique function $\ell:E(G)\to\{0,1\}$ with the following properties:
 \begin{enumerate}
  \item we have $\ell(e'_b)=1$;
  \item suppose that edges $e_1$, $e_2$ share a vertex; then $\ell(e_1)=\ell(e_2)$ if and only if $e_1$, $e_2$ are not separated by the critical pseudoaccesses.
 \end{enumerate}
The function $\ell$ with these properties is called the \emph{labeling}.
For $e\in E(G)$, the value $\ell(e)$ is called the \emph{label of the edge} $e$.
An edge of $\ol T^*$ may consist of several edges of $G$.
These edges have the same label since the critical points of $f$ in $T^*$ are vertices of $\ol T^*$.
The label of an edge of $\ol T^*$ is defined as the label of any edge of $G$ contained in it.
Thus the labeling is also defined on $E(\ol T^*)$.
\end{dfn}

Note that the labeling may not be well defined on $E(T^*)$ if there are edges of $T^*$ subdivided by critical points of $f$.
This was the reason for passing from $T^*$ to $\ol T^*$.

\subsection{Signatures}
\label{ss:signature}
As before, $T$ is a spanning tree for $P(f)$ with specified pseudoaccesses at the critical values.
We also need a function on the edges of $T$.

\begin{figure}
  \centering
  \includegraphics[height=4cm]{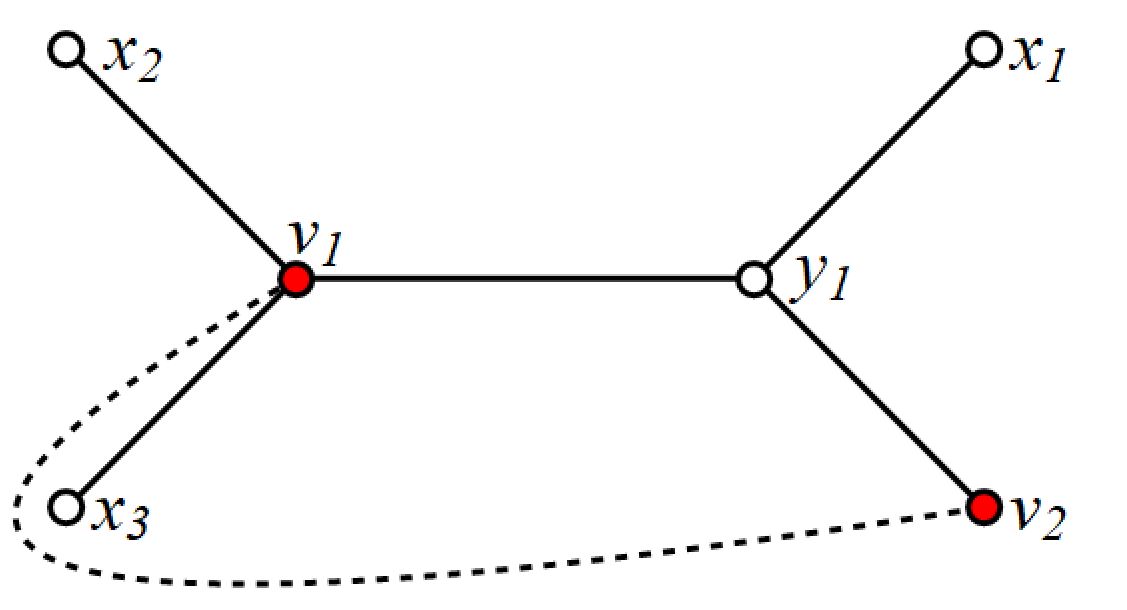}
  \caption{\footnotesize The sequences $S^0(T)$ and $S^1(T)$.
  In this example, the sequence $S^0(T)$ consists of the oriented edges
  $v_1x_2$, $x_2v_1$, $v_1y_1$, $y_1x_1$, $x_1y_1$, $y_1v_2$, taken in this order.
  The sequence $S^1(T)$ consists of the edges $v_2y_1$, $y_1v_1$, $v_1x_3$, $x_3v_1$, taken in this order.
  The signature of $v_1x_2$ and $y_1x_1$ is $(0,0)$.
  The signature of $v_1y_1$ and $y_1v_2$ is $(0,1)$.
  The opposite edges $y_1v_1$ and $v_2y_1$ have signature $(1,0)$.
  The signature of $v_1x_3$ is $(1,1)$.
  The dashed line is the curve $C$ corresponding to the chosen pseudoaccesses at the critical values.}
  \label{fig:circuit}
\end{figure}

\begin{dfn}[Signatures of edges]
\label{d:signature}
Let $C(T)$ be the only boundary circuit of $T$.
Informally: if a particle $x$ loops around $T$ in a small neighbourhood of $T$ so that $T$ is kept on the \textbf{right},
  then the cyclically ordered sequence of oriented edges, along which $x$ moves, coincides with $C(T)$.
Even more informally: $C(T)$ corresponds to walking around $T$ \textbf{clockwise}.
The choice of the direction is explained as follows: as we walk around $T$ clockwise, we walk around $\S-T$ counterclockwise.
The postcritical pseudoaccesses divide all oriented edges from $C(T)$ into two groups (segments) $S^0(T)$ and $S^1(T)$.
The labeling of $S^0(T)$ and $S^1(T)$ is chosen as follows.
By definition, $S^0(T)$ originates at $v_1$ and terminates at $v_2$.
Then $S^1(T)$ originates at $v_2$ and terminates at $v_1$.
See Figure \ref{fig:circuit} for an illustration.
We can now assign \emph{signatures} to all edges of $T$.
We say that an oriented edge $e$ of $T$ is of signature $(i,j)$ if $e$ appears in $S^i(T)$, and $e^{-1}$ appears in $S^j(T)$.
Here, for an oriented edge $e$, we let $e^{-1}$ denote the same edge with the opposite orientation.
Thus there are four possible signatures: $(0,0)$, $(0,1)$, $(1,0)$, and $(1,1)$.
\end{dfn}

\begin{figure}
  \centering
  \includegraphics[width=7cm]{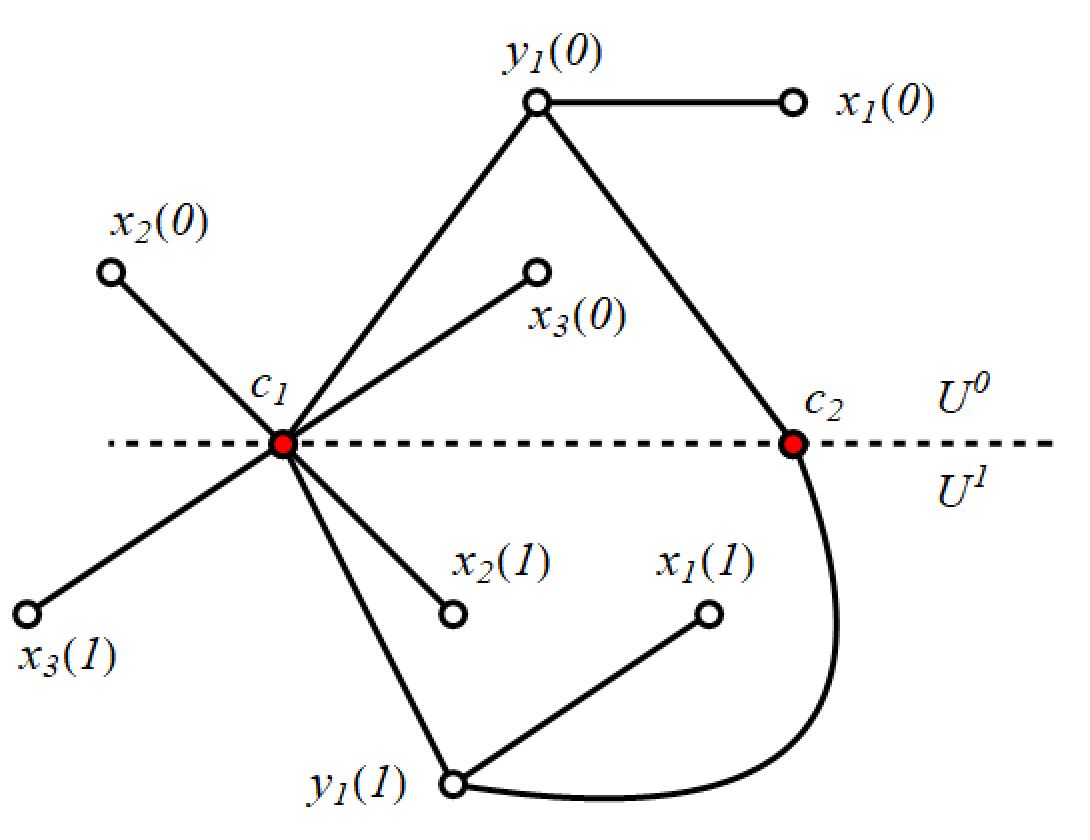}
  \caption{\footnotesize The graph $G$ and the disks $\Omega^0$ and $\Omega^1$.
  Here $G=f^{-1}(T)$, where $T$ is the tree from Figure \ref{fig:circuit}.
  The preimages of $v_1$, $v_2$ are the critical points $c_1$, $c_2$, respectively.
  The preimages of other vertices of $T$ are denoted as their images followed by a label $0$ or $1$ in the parentheses.
  All preimages in $U^0$ (the upper half-plane) are labeled $0$ and
  all preimages in $U^1$ (the lower half-plane) are labeled $1$.
  Then $T^0$ and $T^1$ are the copies of $T$ in $U^0$ and $U^1$, respectively.
  These copies are deformed but topologically the same as $T$.
  The sequence $S^0(T^0)$ goes through the vertices $c_1$, $x_2(0)$, $c_1$, $y_1(0)$, $x_1(0)$, $y_1(0)$, $c_2$.
  The sequence $S^1(T^1)$ goes through the vertices $c_2$, $y_1(1)$, $c_1$, $x_3(1)$, $c_1$.
  The disk $\Omega^0$ bounded by $S^0(T^0)\cup S^1(T^1)$ is the exterior of the quadrilateral $c_1y_1(0)c_2y_1(1)$
  with the arcs $c_1x_3(1)$, $c_1x_2(0)$ and $y_1(0)x_1(0)$ removed.
  }\label{fig:G}
\end{figure}

For $i,j=0,1$, we write $S^i(T^j)$ for the pullback of $S^i(T)$ in $T^j$.
The complement of $G$ in $\S^2$ consists of two disks $\Omega^0$ and $\Omega^1$.
These disks are bounded by $S^0(T^0)\cup S^1(T^1)$ and $S^0(T^1)\cup S^1(T^0)$.
We assume that $\Omega^0$ to be the disk bounded by $S^0(T^0)\cup S^1(T^1)$.
See Figure \ref{fig:G} for an illustration.
More precisely, the oriented boundary of $\Omega^0$,
regarded as a chain of oriented edges of $G$,
is the concatenation of $S^0(T^0)$ and $S^1(T^1)$.
Then $\Omega^1$ is bounded by $S^0(T^1)\cup S^1(T^0)$ in a similar sense.
We will assume that $\Omega^0\ni y$.
The problem, however, is that the two assumptions
\begin{enumerate}
  \item that $y\in\Omega^0$, and
  \item that $\Omega^0$ is bounded by the concatenation of $S^0(T^0)$ and $S^1(T^1)$
\end{enumerate}
may not be compatible.
There are two ways of making them both hold.
On the one hand, we can choose $y$ differently.
Although this is easy in theory, we will not do this in practice.
A basepoint will be fixed once and for all
(see Assumption \ref{as:alpha0} for the principle of choosing the basepoint).
On the other hand, we may relabel $T^0$ and $T^1$ by choosing $e'_b$ differently.
The edge $e'_b$ of $G$ is one of the two pullbacks of $e_b$; the one not in $T^*$.
If we change $T^*$, then we can also replace $e'_b$ with the other pullback of $e_b$.
In this way, we can satisfy both assumptions.
This is how we will act in practice.
At each step of our iterative process, we will define $T^*$ (and $e'_b$) so that both assumptions hold.
The exact procedure will be described later.
For now, we just assume that both assumptions are satisfied.

\subsection{The choice of paths $\alpha_0$, $\alpha_1$}
\label{ss:alpha}
We assume that $(T^*,T)$ is a dynamical tree pair for $f$.
As before, $T$ comes with a specific choice of pseudoaccesses at the critical values.
Recall that $f^{-1}(y)=\{y_0,y_1\}$.
We label the preimages $y_0$, $y_1$ of $y$ so that $y_i\in\Omega^i$ for $i=0,1$.
Choose a path $\alpha_0$ connecting $y$ with $y_0$ outside of $T^*$.
Similarly, choose a path $\alpha_1$ connecting $y$ with $y_1$ outside of $T^*$.
Then $\Bc=\{[\alpha_0],[\alpha_1]\}$ is a basis of $\Xc_f(y)$.
The basis $\Bc$ is well defined and depends only on $T^*$ and $y$.
This description of $\Bc$ is sufficient for now.
However, for later use, we will need a more accurate description of $\alpha_0$ and $\alpha_1$.
We describe them up to a homotopy rel $f^{-1}(V(T))$ rather than rel $V(T)$.
Choose a path $\alpha_0$ connecting $y$ to $y_0$ so that it is disjoint from $G$.
This is possible.
Indeed, according to assumption $(1)$ made in Section \ref{ss:signature}, we have $y\in\Omega_0$.
Recall also that $\Omega_0$ is a topological disk, and that $y_0\in\Omega_0$ by definition of $y_0$.
Therefore, $y$ can be connected with $y_0\in \Omega_0$ by a path in $\Omega_0$.
This path is automatically disjoint from $G$; and we take this path as $\alpha_0$.
The path $\alpha_1$ should be chosen so that it crosses $G$ only once in a point of $e'_b$.
We may arrange that $\alpha_1$ is smooth and that the intersection is transverse.
Since $e'_b$ is not included into $T^*$, this description of $\alpha_1$ is consistent with the earlier description.

For example, in Figure \ref{fig:G}, the path $\alpha_1$ goes from the outside of the quadrilateral $c_1y_1(0)c_2y_1(1)$
  to the inside.
It may cross either $c_2y_1(1)$ or $y_1(1)c_1$; thus there are two possible choices for $e'_b$.

\subsection{A more precise statement of Theorem B}
\label{ss:ThmB-aut}
In this section, we restate Theorem B more precisely and in a greater generality.
Recall that $f:\S^2\to\S^2$ is a Thurston map of degree two.
We assume that $f$ has a dynamical tree pair $(T^*,T)$.
Let $\Ec_T$ and $\Ec_{T^*}$ be the generating sets of $\pi_f$ defined as in Section \ref{ss:Ec}.
We will describe a map $\Sigma:\{0,1\}\times\Ec_T\to \Ec_{T^*}\times\{0,1\}$.
By definition, $\Sigma(\eps,[\gamma])$ is $([\alpha_\eps\gamma^*\alpha_{\eps^*}^{-1}],\eps^*)$,
  where $\gamma^*$ is an $f$-pullback of $\gamma$ originating at $y_\eps$ and terminating at $y_{\eps^*}$.
Note that $\eps^*$ and $\gamma^*$ are determined by $\eps$ and $\gamma$.
The map $\Sigma$ defines a presentation of $\Xc_f(y)$.

Recall our assumption on the basepoint $y$:
the complementary component $\Omega^0$ of $f^{-1}(T)$ containing $y$ is bounded by $S^0(T^0)$ and $S^1(T^1)$.

\begin{thm}
\label{t:ThmB-hom}
We use the terminology and notation introduced above.
Suppose that $\eps\in\{0,1\}$ and $g\in\Ec_{T}$.
Then we have
$$
\Sigma(\eps,g)=(g^*,\eps^*),
$$
where $\eps^*$ and $g^*$ are defined as follows.
If $g=id$, then $\eps^*=\eps$ and $g^*=id$.
Suppose now that $g=g_e$, where $e$ is an oriented edge of $T$ of signature $(\eps+\delta,\eps^*+\delta)$.
Here the addition is mod 2; observe that $\delta$ and $\eps^*$ are determined by $\eps$ and the signature of $e$.
If there is an oriented edge $e^*$ of $\ol T^*$ labeled $\delta$ that maps over $e$ preserving the orientation, then $g^*=g_{e^*}$.
If there is no such edge, then $g^*=id$.
\end{thm}

An edge $e^*$ \emph{mapping over} $e$ preserving the orientation means that $f(e^*)$ is an oriented Jordan arc containing $e$, and that the orientation of $f(e^*)$ is consistent with that of $e$.
Note that an element $g_{e^*}$ for $e^*\in E(\ol T^*)$ is also an element of $\Ec_{T^*}$.
If, say, a critical point divides an edge of $T^*$ into two edges of $\ol T^*$, then
 these two edges give rise to the same pair of mutually inverse elements of $\Ec_{T^*}$.
Suppose that $T^*=T$, then $T$ is an invariant spanning tree for $f$.
In this case, we obtain a finite state automaton $(\{0,1\},\Ec_T,\Sigma)$.
Theorem \ref{t:ThmB-hom} gives an explicit description of this automaton.
Thus it provides a specification of Theorem B.

\begin{cor}
  \label{c:ThmB-hom}
  A dynamical tree pair $(T^*,T)$ for $f$ determines the map $\Sigma:\{0,1\}\times\Ec_T\to\Ec_{T^*}\times\{0,1\}$
  that provides a presentation for the biset of $f$.
  In particular, the isomorphism class of the biset $\Xc_f(y)$ and hence the homotopy class of $f$ are determined by $(T^*,T)$.
\end{cor}

Note that the description provided in Theorem \ref{t:ThmB-hom} depends on the choice of pseudoaccesses.
However, the end result, i.e., the map $\Sigma:\{0,1\}\times\Ec_T\to \Ec_{T^*}\times\{0,1\}$,
 is obviously independent of these choices.

\begin{ex}[An automaton for the basilica polynomial]
\label{ex:bas}
Recall that the basilica polynomial is $p(z)=z^2-1$.
It is easy to find a presentation for the biset of $p$ directly (cf. \cite[Section 5.2.2]{Nek05}).
However, we will use Theorem \ref{t:ThmB-hom} in order to illustrate its statement.
Let $T$ be the invariant spanning tree for $p$ defined in Example \ref{ex:quadpoly}:
$$
\xymatrix{
{\stackrel{-1}{\bullet}} \ar@{->}[r]^A &{\stackrel{0}{\circ}} \ar@{->}[r]^B &{\stackrel{\infty}{\bullet}}
}
$$
Since $T$ is invariant, we may take $T^*=T$.
The tree $T$ has three vertices $-1$, $0$, $\infty$ and two edges: $A=[-1,0]$ and $B=[0,\infty]$.
Note that $A$ is the Hubbard tree for $p$.
Orient these two edges from left to right (in the picture, the orientations are represented by arrows).
Observe that $A$ maps onto $A$ reversing the orientation, and $B$ maps onto $A\cup B$ preserving the orientation.
We may represent this symbolically as
$$
A\to A^{-1},\quad B\to A,B.
$$
Observe also that $B$ contains the $\beta$-fixed point $x_\beta$ of $p$, i.e.,
the landing point of the invariant external ray (the latter ray is also a part of $B$).
Set $a=g_{A}$ and $b=g_{B}$.
Thus we have $\Ec=\Ec_T=\{id,a,b,a^{-1},b^{-1}\}$.

Theoretically, we have to make some choices.
Observe that the location of the basepoint is irrelevant since the complement of $T$ in the sphere is simply connected.
The only possible choice for a base edge is $B$ since $f^{-1}(A)\subset T$.
Also, we need to choose two pseudoaccesses at the critical values $v_1=-1$ and $v_2=\infty$.
However, since both critical values are endpoints of $T$, these pseudoaccesses are unique.
In order to implement the algorithm described in Theorem \ref{t:ThmB-hom}, we need to find labels and signatures.
Since $B$ is the base edge, and $B$ maps over $B$, we have $\ell(B)=0$.
Indeed, the edge $e'_b$ of $f^{-1}(G)$ not in $T$ but mapping also over $B$ has label $1$ by definition.
The two critical pseudoaccesses at $0$ separate $A$ from $B$, hence we have $\ell(A)=1$.
The boundary circuit $C(T)$ is $[A,B,B^{-1},A^{-1}]$ (square brackets denote a cyclically ordered set).
Here $A^{-1}$, $B^{-1}$ stand for the edges $A$, $B$ equipped with the opposite orientation.
The two post-critical pseudoaccesses divide $C(T)$ into $S^0(T)=(A,B)$ and $S^1(T)=(B^{-1},A^{-1})$.
By Definition \ref{d:signature}, both $A$ and $B$ have signature $(0,1)$.
The oriented edges $A^{-1}$, $B^{-1}$ have signature $(1,0)$.
We can now compute $(g^*,\eps^*)$ for each pair $(\eps,g)\in\{0,1\}\times\Ec$ according to Theorem \ref{t:ThmB-hom}.

The computations can be organized as follows.
Draw the following table:

\medskip

\bgroup
\def\arraystretch{1.5}
\centerline{
\begin{tabular}{|c|c|c|c|c|}
 \hline
 & $a(0,1)$& $b(0,1)$ & $a^{-1}(1,0)$& $b^{-1}(1,0)$\\
 \hline
0& & & & \\
\hline
1& & & & \\
\hline
& $a^{-1}(1)$& $b(0)$& $a(1)$& $b^{-1}(0)$\\
& $b(0)$ & & $b^{-1}(0)$& \\
\hline
\end{tabular}
}
\egroup

\medskip

In the top row, we list all elements of $\Ec-\{id\}$.
After each element, we indicate its signature.
Thus, columns of the table (except for the leftmost one) are marked by oriented edges of $T$.
These are the $A$-column, then the $B$-column, etc.
The last row is temporarily filled as follows.
In the $A$-column, we write all elements of the form $g_e$, where $e$ is mapped over $A$ preserving the orientation.
In our case, these elements are $a^{-1}=g_{A^{-1}}$ and $b=g_B$.
We proceed similarly with other columns.
After each element of $\Ec-\{id\}$ in the last row, we indicate in the parentheses the label of the corresponding edge.

Now we can fill the second and the third rows of the table.
For example, look at the $A$-column.
Take one of the entries in the last row, say, $a^{-1}(1)$.
Here $a^{-1}$ is an element of $\Ec$ and $1$ is the label.
Add the label to both components of the signature written in the same column.
In our case, we obtain $(0+1,1+1)=({\bf 1},{\bf 0})$.
This means that ${\bf 1}.a=a^{-1}{\bf 0}$ by Theorem \ref{t:ThmB-hom}.
We write $a^{-1}0$ at the intersection of the $A$-column with the row marked $1$.
Now take the remaining entry in the last row, $b(0)$.
Adding the label to the signature, we obtain $(0+0,1+0)=({\bf 0},{\bf 1})$.
This means that ${\bf 0}.a=b\,{\bf 1}$.
We write $b\,1$ at the intersection of the $A$-column with the row marked $0$.

More generally, consider the column marked by an oriented edge $e$ of $T$.
In the first row, we indicated the signature $(\eps+\delta,\eps^*+\delta)$ of this edge,
right after the corresponding element $g_e$.
In the last row, we indicated an edge $e^*$ mapping over $e$ in an orientation preserving fashion,
and the label $\delta$ of $e^*$.
We add $\delta$ to both components of $(\eps+\delta,\eps^*+\delta)$ to obtain $(\eps,\eps^*)$.
Then we have $\eps.g_e=g_{e^*}\eps^*$ by Theorem \ref{t:ThmB-hom}.
We write $g_{e^*}\eps^*$ at the intersection of the $e$-column with the row marked $\eps$.
If there is no edge $e^*$ of label $\delta$ mapping over $e$ preserving orientation, then we write $\eps^*$.

Acting in this way, we obtain the following table
(from which we removed the last row as it was not needed anymore).

\medskip

\bgroup
\def\arraystretch{1.5}
\centerline{
\begin{tabular}{|c|c|c|c|c|}
 \hline
 & $a$& $b$ & $a^{-1}$& $b^{-1}$\\
 \hline
0& $b\, 1$ & $b\, 1$ & $a\, 1$ & $1$\\
\hline
1& $a^{-1}0$ & $0$ & $b^{-1}0$ & $b^{-1}0$\\
\hline
\end{tabular}
}
\egroup

\medskip

Here, in order to evaluate $\Sigma(\eps,g_e)$, one has to look at the intersection of the row marked $\eps\in\{0,1\}$ with the column marked $g_e\in\Ec$.
The cell of the table at the given position contains $g_{e'}\eps'$ or $\eps'$.
In the former case, we have $\Sigma(\eps,g_e)=(g_{e'},\eps')$.
In the latter case, we have $\Sigma(\eps,g_e)=(id,\eps')$.

The following is the Moore diagram for the obtained automaton.

$$
\xymatrix{
*+<1.5pc>[o][F]{a}\ar@{->}[rr]^{(0,1)} \ar@/_/[dd]_{(1,0)}   && *+<1.3pc>[o][F]{b}\ar@{->}[dr]^{(1,0)}\ar@(r,u)[]_{(0,1)} & \\
 && & *+<1.3pc>[o][F]{id}\\
*+[o][F]{a^{-1}}\ar@/_/[uu]_{(0,1)} \ar@{->}[rr]_{(1,0)}          &&*+<0.8pc>[o][F]{b^{-1}}\ar@{->}[ur]_{(0,1)}\ar@(r,d)[]^{(1,0)} &
}
$$
\end{ex}

\subsection{Proof of Theorem \ref{t:ThmB-hom}}
Set $G=f^{-1}(T)$.
Since $f(T^*)\subset T$, we have $T^*\subset G$ set-theoretically.
Moreover, all vertices of $T^*$ are also vertices of $G$ but, in general, not the other way around.
Recall that complementary components of $G$ correspond to boundary circuits of $G$.
We will write $C^i(G)$ for the boundary circuit corresponding to $\Omega^i$.
Recall that the labeling of $\Omega^i$ was defined so that $C^0(G)$ is the concatenation of $S^0(T^0)$ and $S^1(T^1)$.
The boundary circuit $C^1(G)$ is then the concatenation of $S^0(T^1)$ and $S^1(T^0)$.

Let $e$ be an edge of $T$.
Then $f^{-1}(e)$ can be represented as a union $e^0\cup e^1$, where $e^0\in E(T^0)$ and $e^1\in E(T^1)$.
The following proposition is an alternative description of the boundary circuits $C^i(G)$, where $i=0,1$.

\begin{prop}
\label{p:signatures}
 Let $e$ be an oriented edge of $T$ of signature $(i,j)$.
Then $e^0$ belongs the boundary circuit $C^i(G)$ and $e^1$ belongs to the boundary circuit $C^{1-i}(G)$.
In other words, $e^\delta$ belongs to $C^{i+\delta}(G)$,
where $\delta=0,1$ and the addition is mod 2.
\end{prop}

\begin{proof}
Suppose that the signature of $e$ is $(i,j)$.
It follows by definition of a signature that $e\in S^i(T)$.
The edge $e^0$ of $G$ is a part of $T^0$ hence also of $S^i(T^0)$.
By definition, this means that that $e^0\in C^i(G)$.
The proof of the claim that $e^1\in C^{1-i}(G)$ is similar.
\end{proof}

We are now ready to prove Theorem \ref{t:ThmB-hom}.

\begin{proof}[Proof of Theorem \ref{t:ThmB-hom}]
Suppose that we are given $\eps\in\{0,1\}$ and $g\in\Ec_{T}$.
If $g=id$, then the conclusion is obvious.
Thus we may assume that $g=g_e$ for some $e\in E(T)$.
Let $g^*$ and $\eps^*$ be as in the statement of Theorem \ref{t:ThmB-hom}.
Namely, let $(i,j)$ be the signature of $e$.
Set $\delta=\eps+i$ mod 2.
Then we also have $i=\eps+\delta$ mod 2.
Define $\eps^*$ as $j+\delta$ mod 2.
Then we also have $j=\eps^*+\delta$.
Thus the signature of $e$ can be written as $(\eps+\delta,\eps^*+\delta)$.
If there is an edge $e^*$ of $\ol T^*$ labeled $\delta$ that maps over $e$ preserving the orientation,
 then we set $g^*=g_{e^*}$.
Note that, if an edge $e^*$ exists, then it is unique.
Indeed, there is only one edge $e^\delta$ of $T^\delta$ mapping to $e$.
This edge $e^\delta$ may or may not be a subset of $T^*$.
If it is, then it is contained in a unique edge $e^*$ of $\ol T^*$.
We equip $e^*$ with the orientation induced from the orientation of $e$ by the map $f:e^\delta\to e$.
Thus $e^*$ is uniquely determined as an oriented edge of $\ol T^*$.
If $e^\delta$ is not a subset of $T^*$, then we set $g^*=id$.

We now need to prove that $\Sigma(\eps,g)=(g^*,\eps^*)$, i.e., that $[\alpha_\eps].g=g^*\,[\alpha_{\eps^*}]$.
Let $\gamma_e$ be a smooth loop based at $y$, crossing $T$ just once transversely and approaching it from the left.
Thus $[\gamma_e]=g$.
By definition $[\alpha_\eps].g$ is (the homotopy class of) the concatenation of $\alpha_\eps$ and a pullback $\gamma^*_e$ of $\gamma_e$.
The pullback $\gamma^*_e$ should start at $y_\eps$, where $\alpha_\eps$ ends.
The path $\gamma^*_e$ approaches some boundary edge $e'$ of $\Omega^\eps$.
Thus $e'$ is an edge of $G$.
Equip $e'$ with an orientation such that $\gamma^*_e$ approaches $e'$ from the left.
Then $e'$ is an element of the boundary circuit $C^\eps(G)$ corresponding to the boundary of $\Omega^\eps$.
Observe that $e'$ must be a pullback of $e$, hence it must coincide with $e^0$ or with $e^1$.
We need to find which one.
By Proposition \ref{p:signatures}, the edge $e^\delta$ belongs to $C^{i+\delta}(G)=C^\eps(G)$.
Therefore, we have $e'=e^\delta$.
Since $e$ is of signature $(\eps+\delta,\eps^*+\delta)$, the two sides of the arc $e'$ belong to $\Omega^\eps$ and $\Omega^{\eps^*}$.
Indeed, the left side of $e'$ is $\Omega^\eps$, as we already know.
On the other hand, by Proposition \ref{p:signatures},
  the opposite edge $(e')^{-1}$ belongs to the boundary circuit $C^{j+\delta}(G)=C^{\eps^*}(G)$.
It follows that the right side of $e'$ is $\Omega^{\eps^*}$.
When crossing $e'$, the path $\gamma^*_e$ leaves $\Omega^\eps$ and enters $\Omega^{\eps^*}$ (it may be that $\eps=\eps^*$).

It follows that the path $\alpha_\eps.\gamma_e$ terminates in $\Omega_{\eps^*}$.
We must have then
$$
[\alpha_\eps].[\gamma]=[\alpha_\eps\gamma^*_e\alpha_{\eps^*}^{-1}]\,[\alpha_{\eps^*}],
$$
and it remains to show that $[\alpha_\eps\gamma^*\alpha_{\eps^*}^{-1}]=g^*$.

Suppose first that $e'$ is not a subset of $T^*$ (then $g^*=id$).
Then $\gamma^*_e$ is disjoint from $T^*$.
Since, by our assumption, $\alpha_0$, $\alpha_1$ are also disjoint from $T^*$, the loop $\alpha_\eps\gamma^*_e\alpha_{\eps^*}^{-1}$ lies entirely in $\S^2-T^*$.
The set $\S^2-T^*$ is simply connected, therefore, this loop is contractible in $\S^2-T^*$ and in $\S^2-P(f)\supset\S^2-T^*$.
Thus both sides of the equality $[\alpha_\eps\gamma^*_e\alpha_{\eps^*}^{-1}]=g^*$ equal $id$, and the equality holds.

Finally, suppose that $e'$ is a subset of $T^*$.
Then, since the edge $e^*$ of $\ol T^*$ has label $\delta$, we have $e^*\supset e'$.
The path $\gamma^*_e$ intersects $G$ once, and approaches $e^*$ from the left.
Therefore, $[\alpha_\eps\gamma^*_e\alpha_{\eps'}^{-1}]=g_{e^*}$, which proves the desired.
\end{proof}

\section{The ivy iteration}
We start with a geometric explanation of the process, after which we provide a formal combinatorial implementation.

\subsection{A geometric description of the iterative process}
\label{ss:ivy-geom}
Consider a marked Thurston map $f:\S^2\to\S^2$ of degree 2 with critical values $v_1$ and $v_2$.

We now describe a procedure that, given a spanning tree $T$ for $f$, allows to recover a dynamical tree pair $(T^*,T)$.
Take the full preimage $G=f^{-1}(T)$.
The basic idea is to select a spanning tree $T^*$ in $G$.
More precisely, we select some subtree of $G$ containing $P(f)$ and then erase some of its vertices.
Thus the choice of $T^*$ is in general not unique.
Below, we will give more precise comments on what it involves to make this choice, in terms of combinatorics.

Recall that the basepoint $y$ is assumed to be outside of $T\cup G$.
We also assume that $y$ and $y_0$ are in the same component of $\S^2-G$.
Choose a base edge $e_b$ of $T$.
As above, we assume that $e_b$ separates $v_1$ from $v_2$.
Having chosen a base edge $e_b$, we can recover $T^*$.
There are two pullbacks of $e_b$ in $G=f^{-1}(T)$.
We choose one of the two pullbacks $e'_b$ so that the following properties hold:
\begin{itemize}
 \item The edge $e'_b$ of $G$ is oriented so that $f:e'_b\to e_b$ preserves the orientation.
 \item Consider a path in $\S^2-G$ originating at $y$ and approaching $e'_b$.
 This path approaches $e'_b$ from the left.
\end{itemize}
Recall that $Z$ is the smallest arc in $T$ connecting $v_1$ with $v_2$.
Then the Jordan curve $f^{-1}(Z)$ consists of two pullbacks of $Z$.
Both pullbacks of $e_b$ are in $f^{-1}(Z)$.
They are oriented both from $c_1$ to $c_2$ or both from $c_2$ to $c_1$.
Thus, one of them, $e'_b$, is oriented as the boundary of the component of $\S^2-f^{-1}(Z)$ containing $y$.
This shows that $e'_b$ is well defined.

We can now define a spanning tree $T^*$.
Clearly, $f^{-1}(Z)$ is the only simple loop in $G$.
Thus, removing $e'_b$ from $G$ leads to a tree.
We set $\widehat T$ to be the smallest subtree of this tree containing $P(f)$.
(In particular, all endpoints of $\widehat{T}$ must be in $P(f)$).
Finally, define $T^*$ as the tree obtained from $\widehat{T}$ by
erasing all vertices of $\widehat T$ that are not in $P(f)$ and are not branch points of $\widehat{T}$.
The erased vertices become points in the edges of $T^*$.
Then $(T^*,T)$ is a dynamical tree pair for $f$.
By definition of labels given in Section \ref{ss:baseedge}, we have $\ell(e'_b)=1$, equivalently, $e'_b\in T^1$.
By the properties of $e'_b$ listed above, the orientation of $e'_b$ corresponds to the orientation of $S^1(T^1)$.
Hence the boundary circuit $C^0(G)$ corresponding to $\Omega^0$ is the concatenation of $S^0(T^0)$ and $S^1(T^1)$.
Thus our assumption made in Section \ref{ss:signature} is fulfilled, and Theorem \ref{t:ThmB-hom} applies to $(T^*,T)$.

The \emph{topological ivy iteration} is aimed at finding an invariant spanning tree for $f$, up to homotopy, or,
  more generally, at finding periodic (also up to homotopy, to be made precise later) spanning trees.
Note that an invariant (or periodic), up to homotopy, spanning tree for $f$ yields  a genuine invariant (or periodic)
  spanning tree for some map homotopic to $f$.
Consider a dynamical tree pair $(T^*,T)$ as above.
Since there are finitely many choices for $e_b$, there are also finitely many choices for $T^*$.

\begin{dfn}[Topological ivy object]
A (topological) \emph{ivy object} is defined as a homotopy class of spanning trees for $P(f)$.
We will write $\ivy(f)$ for the set of all ivy objects for $f$.
For a fixed $f$ with $|P(f)|\ge 4$, there are countably many ivy objects.
There is a free action of the pure mapping class group of $(\S^2,P(f))$ on the set of ivy objects of $f$.
In general, this action is not transitive as, for example, spanning trees may have different combinatorics.
\end{dfn}

Note that, as a set, $\ivy(f)$ depends only on $P(f)$, not on $f$.
However, we will introduce a relation on $\ivy(f)$ that will depend on the dynamics of $f$.
If $T$ is a spanning tree for $f$, then $[T]$ will denote the corresponding ivy object.
Suppose that $(T^*,T)$ is a dynamical tree pair as above.
Then we will write $[T]\multimap [T^*]$, and call the thus defined relation on $\ivy(f)$ the \emph{pullback relation}.
The pullback relation on $\ivy(f)$ can be represented by a structure of an abstract directed graph.
There are finitely many arrows originating at each element of $\ivy(f)$.
The set $\ivy(f)$ equipped with the graph structure just described is called the \emph{ivy graph} of $f$.

Recall that, a subset $C\subset\ivy(f)$ is \emph{pullback invariant} if the following property holds:
whenever $[T]\in C$ and $[T]\multimap [T^*]$, we also have $[T^*]\in C$.
In other words, in the associated directed graph, there are no edges originating in $C$ and terminating outside of $C$.
Finding pullback invariant subsets of $\ivy(f)$ is obviously related to finding periodic spanning trees.
By the way, we can now rigorously define a \emph{periodic object} $\tau\in\ivy(f)$ as
  an element of some directed cycle in $\ivy(f)$.
The length of any simple cycle containing $\tau$ is called a \emph{period} of $\tau$.
(Note that $\tau$ may have several periods according to this definition.
Moreover, as a rule, periodic ivy objects do have several different periods.)
If $T$ is a spanning tree such that $[T]$ is periodic of period $p$, then we say that
  the spanning tree $T$ is periodic of period $p$, up to homotopy.

There is no way of defining the forward image of an ivy object $[T]$ under $f$ simply because
 homotopies rel. $P(f)$ are not preserved by $f$.
This is also confirmed by the existence of a periodic ivy object with several different periods.

\subsection{A formal description of the ivy iteration}
\label{ss:formal}
The purpose of this section is to give a compact and precise description of the computational scheme.
The scheme has been implemented (as a Wolfram Mathematica code) according to this description.
Motivations and geometric explanations are given above.
However, we will still need some work to relate the geometric story to the combinatorial story.

\subsubsection*{Push forward of a generating set}
Let $\Xc$ be an abstract left free biset over a group $\pi$.
Consider a basis $\Bc$ of $\Xc$ and a symmetric generating set $\Ec$ of $\pi$.
(Recall that \emph{symmetric} means that $g^{-1}\in\Ec$ for every $g\in\Ec$).
We will assume that $1\in\Ec$.
Consider the full automaton $\Sigma:\Bc\times\pi\to\pi\times\Bc$ of $\Xc$.
Set $\Sigma=(\sigma,\iota)$, so that $\Sigma(a,g)=(\sigma(a,g),\iota(a,g))$ for all $a\in\Bc$, $g\in\pi$.
Define the \emph{push forward} $P\Ec$ of $\Ec$ as the set $\sigma(\Bc,\Ec)$.
In other words, $P\Ec$ consists of all elements of the form $\sigma(a,g)$, where $a\in\Bc$ and $g\in\Ec$.
Note also that the set $P\Ec$ coincides with the set of all restrictions (sections)
of generators from $\Ec$ under the associated wreath recursion, in Nekrashevych's \cite{Nek05} terminology.
It is easy to see that $P\Ec$ is also a symmetric set containing $1$.
Note that a push forward of $\Ec$ will correspond to a pullback of a tree and, in our case,
it will also be a generating set.
For our purposes, we need to combine the push forward operation with a change of a basis.

\subsubsection*{Push forwards with simultaneous basis changes}
As before, $\Xc$ is an abstract left free biset over a group $\pi$.
Let $\Bc$ be a basis of $\Xc$.
Any function $\lambda:\Bc\to\pi$ defines a basis change: the new basis consists of $\lambda(a)a$, where $a\in\Bc$.
Denote this basis change by $C_\lambda$.
We will also write $a_\lambda$ for $\lambda(a)a$.
The new basis obtained from $\Bc$ through $C_\lambda$ will be denoted by $C_\lambda\Bc$.
Let $C_\lambda\Sigma$ be the full automaton of $\Xc$ in the new basis $C_\lambda\Bc$.
We will now see how the maps $\Sigma$ and $C_\lambda\Sigma$ are related to each other.
The following simple computation solves the problem.
Take any $a\in\Bc$ and any $g\in\pi$, and set $(g^*,a^*)=\Sigma(a,g)$.
Then the following equalities hold in $\Xc$:
$$
a_\lambda.g=\lambda(a)a.g=\lambda(a)g^*a^*=\lambda(a)g^*\lambda(a^*)^{-1}a^*_\lambda.
$$
This computation shows that $C_\lambda\Sigma=(C_\lambda\sigma,C_\lambda\iota)$ is given by
$$
C_\lambda\sigma(a_\lambda,g)=\lambda(a)g^*\lambda(a^*)^{-1},\quad
C_\lambda\iota(a_\lambda,g)=a^*_\lambda.
$$
Given a symmetric generating set $\Ec$ of $\pi$ containing $1$, set $P_\lambda\Ec$ to be the generating set
  obtained as the push forward of $\Ec$ under $C_\lambda\Sigma$.

\subsubsection*{Push forwards of tree structures}
In contrast to the above, we now explicitly assume that $\Bc$ consists of two elements.
As always, we assume that $\Ec$ is symmetric and contains $1$.
We say that $\Ec$ is \emph{tree-like} if there exists a tree structure $\Vc$ on $\Ec$, cf. Section \ref{ss:vertstruc}.
Fix a tree-like generating set $\Ec$.
Below, we will describe a certain set of basis changes $\lambda$ for which $P_\lambda\Ec$ are also tree-like.
For each such $\lambda$, the corresponding tree structure $P_\lambda\Vc$ on $P_\lambda\Ec$
  is defined below.
The map $\Sigma:\Bc\times\Ec\to P\Ec\times\Bc$ can be extended to a map
$$
\Sigma^\star=(\sigma^\star,\iota^\star):\Bc\times\Ec^\star\to (P\Ec)^\star\times\Bc.
$$
The map $\Sigma^\star$ is defined inductively as follows.
If $\0$ denotes the empty word in $\Ec^\star$, then we set $\Sigma^\star(a,\0)=(\0,a)$.
Suppose now that an element of $\Ec^\star$ has the form $g\cdot w$, where $g\in\Ec$ and $w\in\Ec^\star$.
Set $(g^*,a^*)=\Sigma(a,g)$.
Then we set
$$
\sigma^\star(a,g\cdot w)=g^*\sigma^\star(a^*,w),\quad
\iota^\star(a,g\cdot w)=\iota^\star(a^*,w).
$$
Replacing $\Bc$ with $C_\lambda\Bc$, we may assume that $\lambda\equiv 1$.
Suppose that $\Bc=\{a,b\}$.
In order to define the new vertex set $P\Vc\subset (P\Ec)^\star$,
 we first consider the following three sets:
\begin{align*}
  \Vc(a) =&  \{\sigma^\star(a,v)\,|\, v\in\Vc,\ \iota^\star(a,v)=a\},\\
  \Vc(b) =&  \{\sigma^\star(b,v)\,|\, v\in\Vc,\ \iota^\star(b,v)=b\},\\
  \Vc(a,b) =& \{\sigma^\star(a,v)\cdot\sigma^\star(b,v)\,|\, v\in\Vc,\ \iota^\star(a,v)=b\}.
\end{align*}
Now take the union of these three sets, remove the trivial element $1\in\Ec^\star$ form it
  as well as all elements of the form $g\cdot g^{-1}$, where $g\in\Ec$.
The remaining set is $P\Vc$.

\subsubsection*{The combinatorial ivy iteration}
Thus, for every tree-like generating set $\Ec$ of $\pi$, there are several tree-like generating sets of the form $P_\lambda\Ec$.
Set $\Bc=\{a,b\}$.
We will only consider basis changes associated with functions $\lambda:\Bc\to\pi$ such that $\lambda(a)=1$.
This is equivalent to saying that the first basis element $a$ will be fixed once and for all.
In our implementation, this will be the class of a constant loop.

Suppose that $g\in\Ec$ is an element with the property that $\iota(0,g)=1$.
Any such element $g$ is called a \emph{base element}.
With any base element $g$, we associate the function $\lambda_g:\Bc\to\pi$ such that $\lambda_g(b)=\sigma(0,g)$.
(Recall that $\lambda(a)=1$.)
The \emph{combinatorial ivy iteration} is the process of passing from $\Ec$ to $P_{\lambda_g}\Ec$.

Define a \emph{combinatorial ivy object} as a tree-like generating set, up to conjugacy.
More precisely, a combinatorial ivy object is a conjugacy class of tree-like generating sets.
If $\Ec$ is such a generating set, then its conjugacy class will be denoted by $[\Ec]$.
Let $\ivy_c(f)$ be the set of all combinatorial ivy objects with $\pi=\pi_f$.
For a pair of generating sets $\Ec$ and $\Ec'=P_{\lambda_g}\Ec$ as above, connect $[\Ec]$ with $[\Ec']$ by a directed edge.
Each pair $([\Ec],[\Ec'])$ yields only one directed edge,
 no matter in how many ways $[\Ec']$ can be represented in the form $[P_{\lambda_g}\Ec]$.
In this way, $\ivy_c(f)$ becomes a directed graph.
We will show that the combinatorial ivy iteration represents the topological ivy iteration.
In particular, the graph $\ivy(f)$ is isomorphic to a subgraph of $\ivy_c(f)$.
A more precise statement is given in the following theorem:

\begin{thm}
  \label{t:ivy-iso}
There is an isomorphic embedding of $\ivy(f)$ into $\ivy_c(f)$.
This embedding takes a class $[T]$ of a spanning tree $T$ to the conjugacy class of the corresponding generating set $\Ec_T$.
Let $T$ and a base edge $e_b$ of $T$ define a dynamical tree pair $(T^*,T)$ as in Section \ref{ss:ivy-geom}.
If $g$ is the element of $\Ec_T$ corresponding to $e_b$, then $T^*$ corresponds to $P_{\lambda_g}\Ec_T$.
There is a canonical tree structure $\Vc$ on $\Ec_T$ such that $G(\Vc)$ is isomorphic to $T$.
The tree structure $\Vc^*$ on $\Ec_{T^*}$ corresponding to $T^*$ is obtained as $P_{\lambda_g}\Vc$.
\end{thm}

\subsection{Translation from geometry to combinatorics}
\label{ss:geom-algo}

Proposition \ref{p:trconj} implies that there is at least a set-theoretic embedding of $\ivy(f)$ into $\ivy_c(f)$.
Our symbolic implementation of the ivy iteration will rely on the following assumption.

\begin{ass}
\label{as:alpha0}
Whenever we consider a spanning tree $T$ for $P(f)$, we assume that there is a fixed point $y_0$ of $f$ outside of $T$.
Moreover, as long as we deal with spanning trees that eventually map to $T$, the point $y_0$ is kept the same.
The point $y_0$ will be used as the basepoint for $\pi_f$.
\end{ass}

Assumption \ref{as:alpha0} can always be fulfilled if we replace $f$ with a homotopic Thurston map.
Indeed, choose a point $y_0\notin T$ with $f(y_0)\notin T$ and a path $\beta$ connecting $f(y_0)$ to $y_0$ outside of $T$.
Then the assumption is satisfied if we replace $f$ with $\sigma_{\beta}\circ f$.
Here $\sigma_\beta$ is a path homeomorphism introduced in Section \ref{ss:cap}.
Now, if $y_0$ is fixed under $f$, then $y_0$ is disjoint from $f^{-n}(T)$ for all $n\ge 0$.
In particular, if $(T^*,T)$ is a dynamical tree pair for $f$, then $y_0\in \S^2-T^*$.
Thus we can keep the same $f$-fixed base point during the ivy iteration.

\subsubsection*{The basis change associated with a dynamical tree pair.}
Consider a spanning tree $T$ for $f$, and choose a base edge $e_b$ of $T$.
The choice of $e_b$ defines a dynamical tree pair $(T^*,T)$ as in Section \ref{ss:ivy-geom}.
Consider the basis $\Bc=\{[\alpha_0],[\alpha_1]\}$ of $\Xc_f(y_0)$ associated with $T$.
According to the convention introduced in Section \ref{ss:alpha}, the element $[\alpha_0]$ is the class of the constant path.
The path $\alpha_1$ connects $y_0$ with another preimage $y_1$ of $y_0$ outside of $T$.
This property defines $\Bc$ uniquely.
Let now $\Bc^*=\{[\alpha_0],[\alpha^*_1]\}$ be the basis associated with $T^*$.
We need to show that $[\alpha^*_1]=\sigma(0,g)[\alpha_1]$, where $g$ is the element of $\Ec_T$ corresponding to $e_b$.
Then we will have $\Bc^*=C_{\lambda_g}\Bc$.
This is a part of the correspondence between the geometric and the combinatorial ivy iterations.

Set $h=\sigma(0,g)$.
Let $e'_b$ be the edge of $f^{-1}(T)$ of label 1 that is not in $T^*$.
Recall that $\alpha^*_1$ is defined as a path from $y_0$ to $y_1$ crossing $e'_b$ and disjoint from $f^{-1}(T)$ otherwise.
In Section \ref{ss:ivy-geom}, we have chosen $e'_b$ so that $\alpha^*_1$ approaches it from the left.
It follows that $[\alpha_0].g=[\alpha^*_1]$ in $\Xc_f(y_0)$.
Thus we can define $\alpha^*_1$ as $\alpha_0.\gamma_{e_b}$.
On the other hand, we have $\Sigma(0,g)=(h,1)$, therefore, $[\alpha^*_1]=[\alpha_0].g=h[\alpha_1]$, as desired.

\subsection{Pullback and vertex words}
\label{ss:vs}
In this section, we study the effect of the pullback relation on vertex words.

\begin{prop}
  \label{p:v-crival}
Let $x$ be a vertex of $T$, and $v\in\Ec^\star$ be the corresponding vertex word.
The vertex $x$ is a critical value of $f$ if and only if $\iota^\star(0,v)=1$.
In this case, we also have $\iota^\star(1,v)=0$.
\end{prop}

\begin{proof}
Let $\Pi:\Ec^\star\to\pi_f$ be the evaluation map.
Note that $\iota^\star(\eps,v)=\iota(\eps,\Pi(v))$ for each $\eps\in\{0,1\}$.
The element $\Pi(v)\in\pi_f$ is represented by a loop around $x$ that crosses $T$ only in a small neighborhood of $x$.
We may assume that this loop $\gamma$ is smooth and simple.
Then it bounds a disk $D$ such that $D\cap V(T)=\{x\}$.
The two $f$-pullbacks of $\gamma$ are loops or not loops depending on whether $x$ is a critical value or not.
On the other hand, these pullbacks are loops if and only if $\iota(\eps,\Pi(v))=\eps$ for all $\eps=0,1$.
\end{proof}

Suppose now that $v=a_0\cdots a_{k-1}$.
Consider elements $\sigma^\star(0,v)=b_0\cdots b_{k-1}$ and $\sigma^\star(1,v)=c_0\cdots c_{k-1}$ of $(P\Ec)^\star$.
Here $b_i$ and $c_i$ are elements of $\pi_f$, for $i=0$, $\dots$, $k-1$.
Set $\eps_0=0$ and $\eps_i=\iota^\star(0,a_0\cdots a_{i-1})$ for $i=1$, $\dots$, $k$.
Similarly, we set $\delta_0=0$ and $\delta_i=\iota^\star(1,a_0\cdots a_{i-1})$ for $i=1$, $\dots$, $k$.
Suppose first that $\eps_{k}=1$ (then also $\delta_{k}=0$).
Then we define the word $w(0,v)=w(1,v)\in (P\Ec)^\star$ as
$$
\sigma^\star(0,v)\sigma^\star(1,v)=b_0\cdots b_{k-1}\cdot c_0\cdots c_{k-1}.
$$
Suppose now that $\eps_{k}=0$ (then also $\delta_{k}=1$).
Then we define
$$
w(0,v)=\sigma^\star(0,v)=b_0\cdots b_{k-1},\quad  w(1,v)=\sigma^\star(1,v)=c_0\cdots c_{k-1}.
$$

\begin{prop}
  \label{p:vs}
Suppose that the basis $\Bc$ of $\Xc_f(y_0)$ corresponds to $T^*$.
Then the tree structure $\Vc^*$ on $P\Ec$ corresponding to $T^*$ coincides with the set of $w(\eps,v)\in (P\Ec)^\star$,
  where $\eps$ runs through $\{0,1\}$, and $v$ runs through $\Vc$,
  except that we omit $w(\eps,v)$ if it is empty or it has the form $a\cdot a^{-1}$ for some $a\in\Ec$.
In other words, we have $\Vc^*=P\Vc$.
\end{prop}

\begin{proof}
Let $x$ be a vertex of $T$, and $v$ be the corresponding vertex word.
Let $A_0$, $\dots$, $A_{k-1}$ be the oriented edges of $T$ such that $a_i=g_{A_i}$ for $i=0$, $\dots$, $k-1$.
Define $\beta_i$ as a smooth path disjoint from $T$, starting at $y$, and ending in a small neighborhood of $x$.
Choose $\beta_i$ so that it approaches $x$ between $A_{i-1}$ and $A_i$.
Here $i-1$ is understood modulo $k$ so that for $i=0$ we have $A_{i-1}=A_{k-1}$.
Choose the paths $\gamma_{A_i}$ with $a_i=[\gamma_{A_i}]$ as $\beta_i\gamma_i\beta_{i+1}^{-1}$,
  where $\gamma_i$ is a short path in a small neighborhood of $x$ crossing $A_i$ just once and transversely.
We may assume that $\gamma_i$ is disjoint from all other edges of $T$.
For $\eps\in\{0,1\}$, consider the pullback $\beta_i^{\eps}$ of $\beta_i$ originating at $y_\eps$.

Suppose first that $x$ is not a critical value of $f$.
Then there are two preimages $x_0$ and $x_1$ of $x$.
We will prove that the vertex words of $x_0$, $x_1$ are $w(0,v)$, $w(1,v)$ (not necessarily in this order),
  provided that $x_0$ and $x_1$ are vertices of $T^*$.
Without loss of generality, we may assume that $\beta_0^0$ ends near $x_0$.
(Otherwise, simply swap $x_0$ and $x_1$.)
Set $\gamma_i^\eps$ to be the pullback of $\gamma_i$ near $x_\eps$.
An important observation is that a pullback of a short path is short.
Therefore, $\gamma_i^\eps$ must indeed stay near one preimage of $x$ rather than wander between the two preimages.
Then the induction on $i$ shows that $b_i$ is represented by
$\alpha_{\eps_i}\beta^{\eps_i}_i\gamma_i^0(\beta^{\eps_{i+1}}_{i+1})^{-1}\alpha_{\eps_{i+1}}^{-1}$
and that all $\beta_i^{\eps_i}$ end near $x_0$.
Note also that $b_i=id$ if and only if $\gamma_i^0$ does not cross any edge of $T^*$.
Otherwise it crosses exactly one edge.
The composition of all $\gamma_i^0$ is a small loop around $x_0$.
Moreover, this loop is a pullback of the small loop $\gamma$ around $x$, which is the composition of all $\gamma_i$.
It follows that the oriented edges of $T^*$ coming out of $x_0$ correspond precisely to non-identity elements $b_i$.
This means that $w(0,v)$ is the vertex word for $x_0$; the proof of $w(1,v)$ being the vertex word for $x_1$ is similar.

Suppose now that $x$ is a critical value of $f$.
Then there is just one preimage $x_0=x_1$ of $x$.
We will prove that $w(0,v)=w(1,v)$ is the vertex word for $x_0=x_1$ provided that $x_0$ is a vertex of $T^*$.
Set $\gamma_i^0$ be the pullback of $\gamma_i$ originating where $\beta_i^{\eps_i}$ ends, and
  $\gamma_i^1$ be the pullback of $\gamma_i$ originating where $\beta_i^{\delta_i}$ ends.
Since $\delta_i\ne\eps_i$, the paths $\gamma_i^0$ and $\gamma_i^1$ are always different pullbacks of $\gamma_i$.
Similarly to the above, $b_i$ is represented by
$\alpha_{\eps_i}\beta^{\eps_i}_i\gamma_i^0(\beta^{\eps_{i+1}}_{i+1})^{-1}\alpha_{\eps_{i+1}}^{-1}$
and $c_j$ is represented by
$\alpha_{\delta_j}\beta^{\delta_j}_j\gamma_j^1(\beta^{\delta_{j+1}}_{j+1})^{-1}\alpha_{\delta_{j+1}}^{-1}$.
The composition of all $\gamma_i^0$ is a pullback of $\gamma$ but it is not a loop; it is only a ``half'' of a loop.
The other half is the composition of all $\gamma_i^1$, which is also the other pullback of $\gamma$.
It follows that oriented edges of $T^*$ coming out of $x_0$ correspond precisely to non-identity elements $b_i$ or
  non-identity elements $c_j$.
This means that $w(0,v)=w(1,v)$ is the vertex word for $x_0$.

To conclude the proof, we observe that any vertex of $T^*$ is mapped to a vertex of $T$.
Thus any vertex of $T^*$ can be obtained as described above.
If $w(\eps,v)$ is empty, then obviously, the corresponding point $x_\eps$ of $f^{-1}(T)$ does not belong to $T^*$.
If 
$w(\eps,v)$ has the form $a\cdot a^{-1}$, then $x_\eps$ belongs to an edge of $T^*$ corresponding to $a$
(thus, in particular, $x_\eps$ is not a vertex of $T^*$).
Conversely, if $x_\eps$ is not a vertex of $T^*$, then this may be due to one of the following reasons.
Firstly, we may have $x_\eps\notin T^*$, then $w(\eps,v)$ is empty.
Secondly, $x_\eps$ may belong to some edge of $T^*$.
In this case, $w(\eps,v)$ must have the form $a\cdot a^{-1}$, where $a\in P\Ec$ corresponds to this edge.
\end{proof}

Proposition \ref{p:vs} concludes the proof of Theorem \ref{t:ivy-iso}.

\section{Examples of the ivy iteration}
In this section, we describe some particular computations of the ivy graphs made according to the ivy iteration.
We consider only the simplest examples, for which other, sometimes more efficient, computational approaches to distinguishing Thurston equivalence classes are available.
In particular, in most examples, particular invariant spanning trees are known.
We find (conjecturally) all periodic spanning trees in these examples.
Also, we find some pullback invariant sets of ivy objects and show their combinatorial structure.

In \cite{KL18}, all non-Euclidean Thurston maps with at most 4 post-critical points are classified, and
 an algorithm is suggested for solving the twisting problem for such maps.
However, invariant spanning trees for rational maps from \cite{KL18} are not immediate from the provided description.

More complicated examples will be worked out in a separate publication.

\subsection{The basilica polynomial}
\label{ss:bas}
Let us go back to Example \ref{ex:bas}.
This example deals with the basilica polynomial $f(z)=z^2-1$.
We started with an invariant spanning tree
$$
\xymatrix{
{\stackrel{-1}{\bullet}} \ar@{->}[r]^A &{\stackrel{0}{\circ}} \ar@{->}[r]^B &{\stackrel{\infty}{\bullet}}
}
$$
and deduced the corresponding presentation of the biset $\Xc_f(y)$, see Figure \ref{fig:cl-bas}.
This is enough to start the combinatorial ivy iteration.
This process leads to a pullback invariant subset of $\ivy_c(f)$ consisting of 3 combinatorial ivy objects.
Two of these objects correspond to invariant spanning trees.

The two (up to homotopy) invariant spanning trees of $f$ are easy to describe.
The first one is the Hubbard tree connected to infinity as described in Section \ref{ex:quadpoly}.
The second one is a spider in the sense of Hubbard--Schleicher \cite{HS94}.
There is only one remaining ivy object in the given pullback invariant subset of $\ivy(f)$.
It does not correspond to an invariant spanning tree for $f$.

\begin{figure}
  \centering
  \includegraphics[width=12cm]{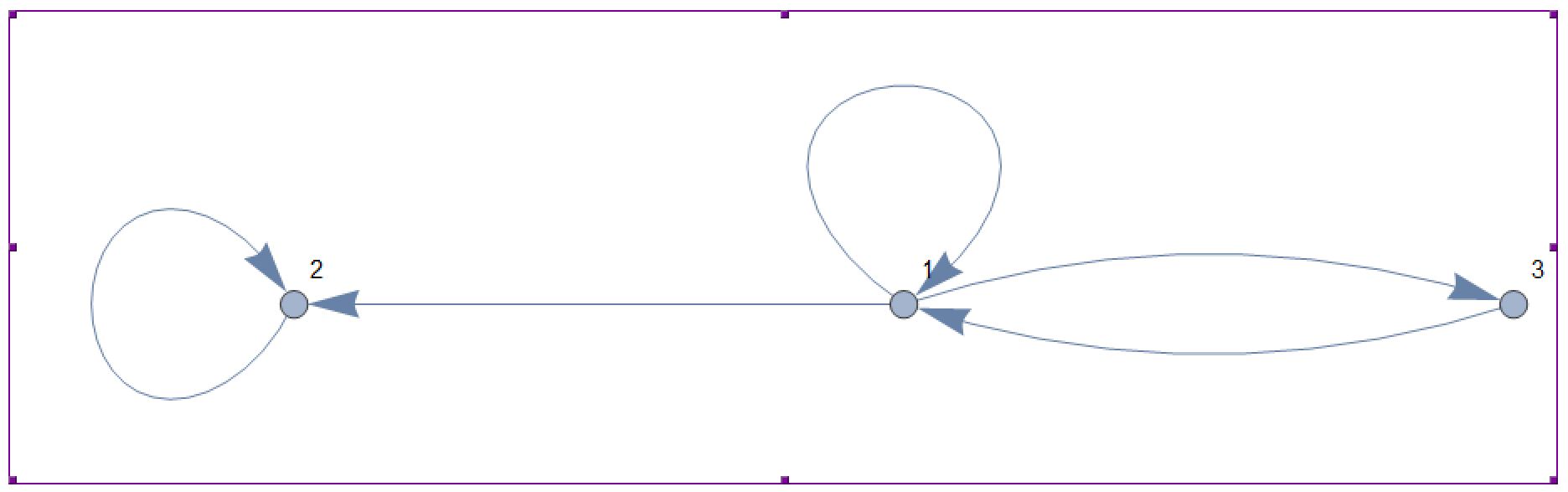}
  \caption{\footnotesize A pullback invariant subset of 3 elements in $\ivy(f)$, where $f(z)=z^2-1$ is the basilica polynomial.
  Arrows represent the pullback relation.
  Vertex 1 corresponds to the tree $T$.
  Vertex 2 corresponds to an invariant spider for $f$.
  Vertex 3 corresponds to a spanning tree that is not invariant up to homotopy;
  it is periodic of period 2.
  Note that vertex 1 has periods 1 and 2.}
  \label{fig:cl-bas}
\end{figure}

Note that $\ivy(f)$ consists of 5 objects.
These objects are the same for all $f$ with a given post-critical set $P(f)$ of 3 elements.
However, the pullback relations (in particular, pullback invariant subsets) defined by $f$ are different.
Three of the elements of $\ivy(f)$ correspond to the unions of two sides of the triangle with vertices in $P(f)$.
The remaining two elements are stars with endpoints in $P(f)$.
The two stars differ by the cyclic order of edges at the only branch point.
Recall that Teichm\"uller theory provides powerful invariants of Thurston equivalence classes in form of certain
 spaces, groups, correspondences between spaces, and virtual homomorphisms between groups.
However, all these invariants are trivial in the case $|P(f)|=3$.

\subsection{The rabbit polynomial}
\label{ss:rab}
The rabbit polynomial $p(z)=z^2+c$ is such that $0$ is periodic of period 3, and $\Im(c)>0$.
These conditions determine $c$ uniquely.
Indeed, the period 3 assumption leads to a cubic equation on $c$, which has one real and two complex conjugate roots.
An invariant spanning tree $T$ for $f$ constructed as in Section \ref{ex:quadpoly} looks as follows.
$$
\xymatrix{
 & {\stackrel{v}{\bullet}} \ar@{<-}^{B}[d] & & \\
{\stackrel{w}{\circ}} \ar@{<-}^{C}[r] & {\stackrel{x_\alpha}{\circ}} \ar@{->}^{A}[dr]& &\\
 & &{\stackrel{0}{\circ}} \ar@{->}^{D}[r] & {\stackrel{\infty}{\bullet}}
}
$$

The choice of the basepoint $y$ for the fundamental group $\pi_{f}=\pi_1(\S^2-P(f),y)$ is irrelevant.
Indeed, the complement of $T$ is simply connected.
The elements of $\pi_f$ associated with the edges of $T$ will be denoted by $a$, $b$, $c$, $d$,
 so that a small letter denotes $g_e$, where $e$ is the edge denoted by the corresponding capital letter.
Thus $\Ec_T$ consists of $id$, $a$, $b$, $c$, and their inverses.
The edges of $T$ map forward as follows:
$$
A\to B,\quad B\to C,\quad C\to A,\quad D\to B^{-1}AD.
$$

Let us compute the map $\Sigma:\{0,1\}\times\Ec_T\to \Ec_T\times\{0,1\}$ from Theorem \ref{t:ThmB-hom}.
We choose $D$ as the base edge.
Then we have the following labels:
$$
\ell(A)=\ell(B)=\ell(C)=1,\quad \ell(D)=0.
$$
The two pseudoaccesses of $T$ at the critical values $v_1=v$ and $v_2=\infty$ are unique.
We have
$$
S^0(T)=(B^{-1},A,D),\quad S^1(T)=(D^{-1},A^{-1},C,C^{-1},B).
$$
By Definition \ref{d:signature}, the edges $A$ and $D$ have signature $(0,1)$, the edge $B$ has signature $(1,0)$,
the edge $C$ has signature $(1,1)$.

By Theorem \ref{t:ThmB-hom}, a presentation for $\Xc_f(y_0)$ looks as follows

\medskip

\bgroup
\def\arraystretch{1.5}
\centerline
{
\begin{tabular}{|c|c|c|c|c|c|c|c|c|}
 \hline
 & $a$& $b$ & $c$& $d$& $a^{-1}$& $b^{-1}$& $c^{-1}$& $d^{-1}$\\
 \hline
0 & $d\,1$ & $a\,1$ & $b\,0$& $d\,1$& $c^{-1}1$& $d\,1$& $b^{-1}0$& $1$\\
\hline
1 & $c\,0$ & $d^{-1}0$ & $1$ & $0$ & $d^{-1}0$ & $a^{-1}0$ & $1$ & $d^{-1}0$ \\
\hline
\end{tabular}
}
\egroup

\medskip

Note that $a=b^{-1}c^{-1}$, so that it is enough to use only $b$, $c$ and $d$ as generators of $\pi_f$.
We will identify $\pi_f$ with the free group generated by $b$, $c$, and $d$.
Then the tree structure on $\Ec_T$ consists of $b^{-1}c^{-1}\cdot c\cdot b$, $cb\cdot d$, $b^{-1}$, $c^{-1}$, and $d^{-1}$.
Using the combinatorial ivy iteration, we found a pullback invariant subset of $\ivy(p)$ consisting of 10 ivy objects,
 see Figure \ref{fig:cl-rab}.
This is the pullback invariant subset containing the class of the tree $T$.

\begin{figure}
  \centering
  \includegraphics[width=10cm]{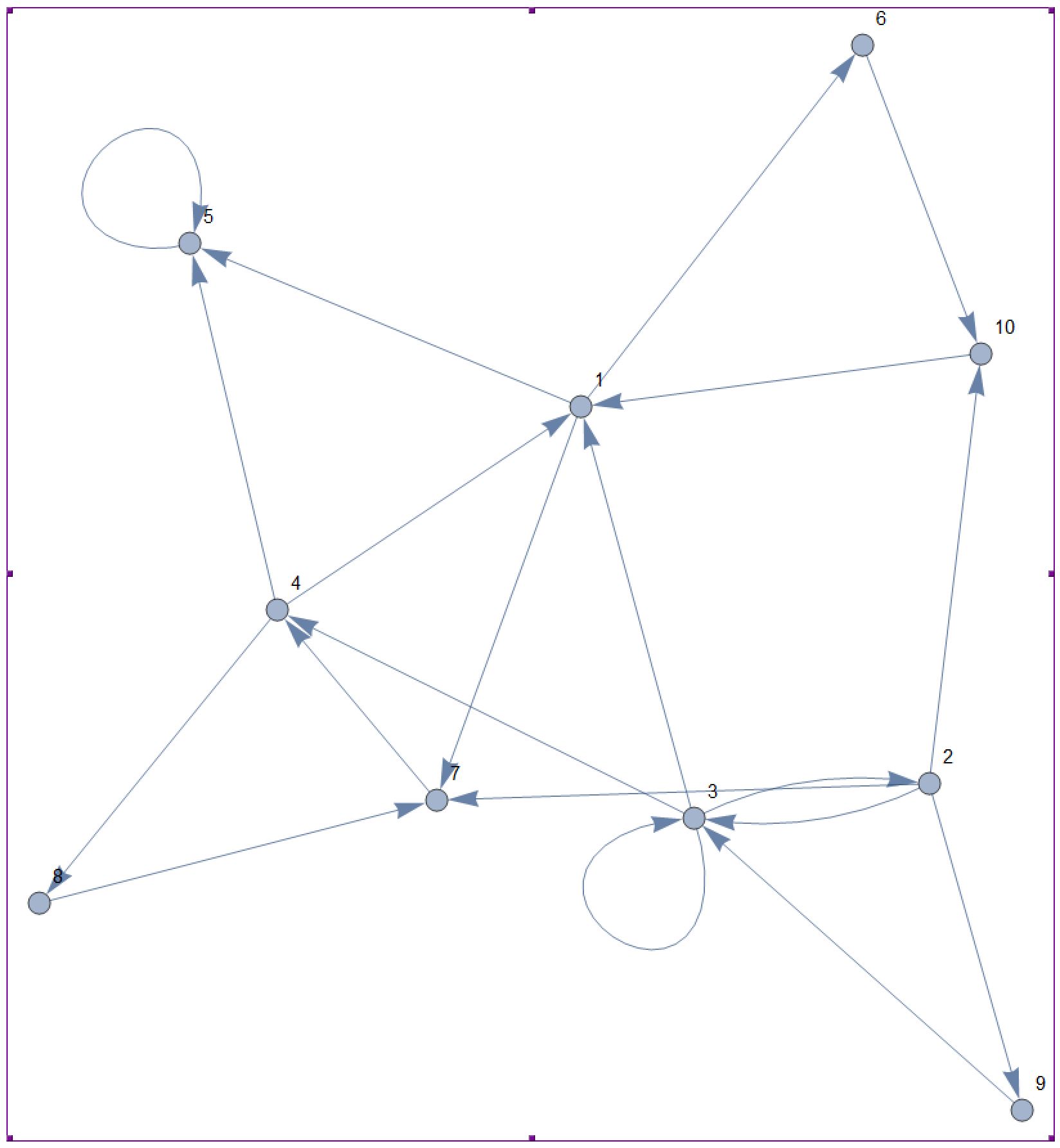}
  \caption{\footnotesize The pullback invariant subset of $\ivy(f)$ containing $[T]$, where $f$ is the rabbit polynomial,
  and $T$ is the invariant spanning tree for $f$ obtained by connecting the Hubbard tree to $\infty$.
  This subset consists of 10 elements.
  Vertex 5 represents an invariant spider, and vertex 3 represents $T$.
  }
  \label{fig:cl-rab}
\end{figure}

We see that, similarly to the basilica, there are two invariant spanning trees for $f$, up to homotopy,
 among the trees representing objects in the found pullback invariant subset.
One tree corresponds to the Hubbard tree connected to $\infty$.
The other tree is an invariant spider.

\subsection{Simple capture of the basilica at $\sqrt{2}$}
\label{ss:cbsqrt2}
Let $p(z)=z^2-1$ be the basilica polynomial.
The point $\sqrt 2$ is preperiodic under $p$ of preperiod 2: it maps to $1$, and $1$ maps to $-1$.
It follows that the simple capture $f$ of $p$ at $\sqrt 2$ has the following invariant spanning tree $T$:
$$
\xymatrix{
{\stackrel{-1}{\bullet}} \ar@{->}[r]^{A} &
{\stackrel{0}{\circ}} \ar@{->}[r]^{B} &
{\stackrel{1}{\circ}} \ar@{->}[r]^{C} &
{\stackrel{\sqrt{2}}{\bullet}}
}
$$
We oriented the edges of $T$ from left to right, and labeled them $A$, $B$, $C$.
The corresponding symmetric generating set of $\pi_f$ is $\Ec_T=\{1,a^{\pm 1},b^{\pm 1},c^{\pm 1}\}$,
  where $a=g_A$, $b=g_B$, $c=g_C$.

\begin{figure}
  \centering
  \includegraphics[width=12cm]{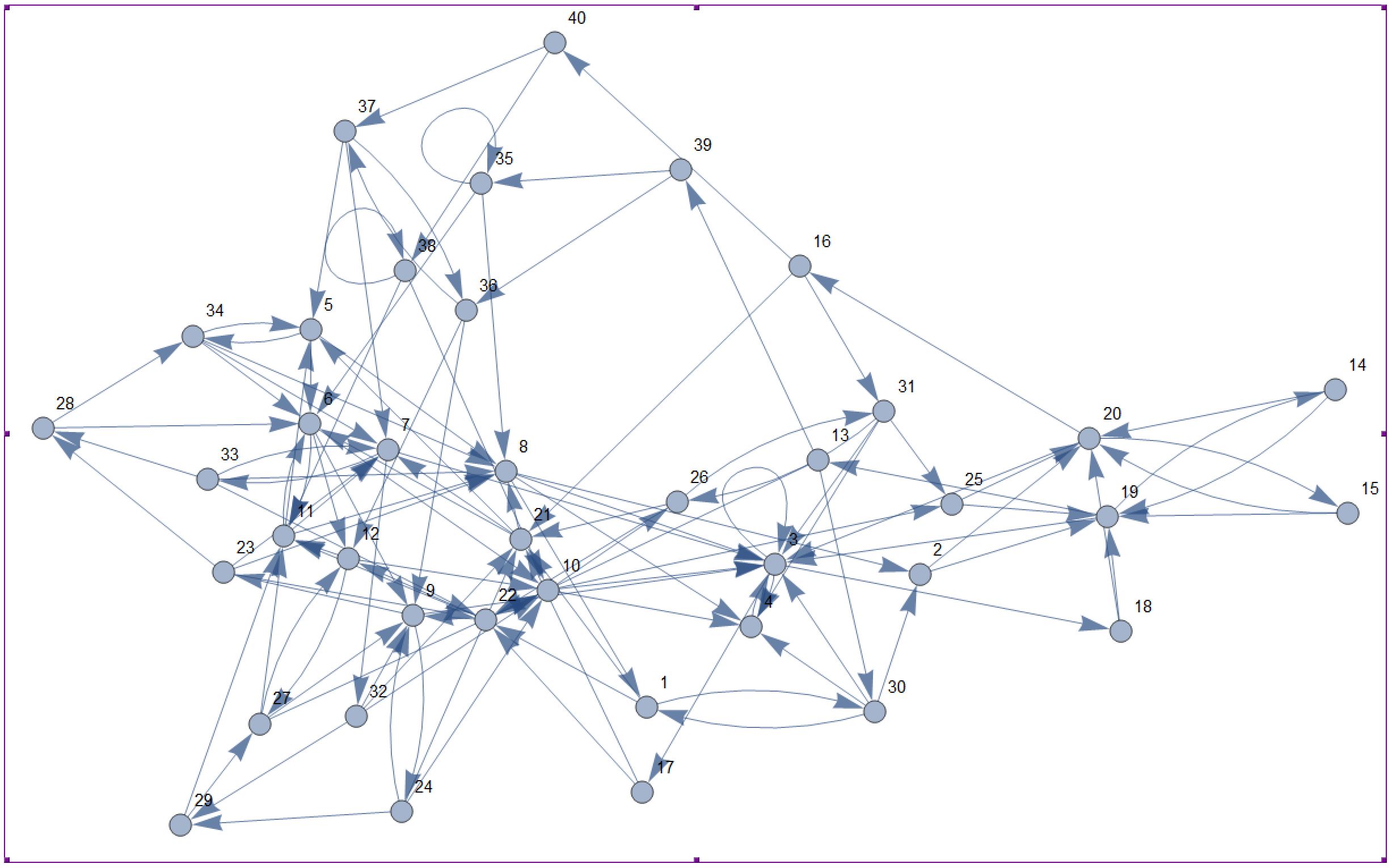}
  \caption{\footnotesize The pullback invariant subset of 40 elements in $\ivy(f)$ containing $[T]$.
  Here $f$ is a simple capture of the basilica at $\sqrt{2}$, and $T$ is the invariant spanning tree for $f$
  introduced in Section \ref{ss:cbsqrt2}.
  Vertices 3, 35 and 38 represent invariant spanning trees for $f$.
  }
  \label{fig:cl-cb}
\end{figure}

Since $-1$ and $\sqrt{2}$ are incident each to a unique edge of $T$, there are unique pseudoaccesses at $-1$ and $\sqrt 2$.
As the base edge of $T$, we take $C$.
Note that $0$ is the only critical point in $T$, thus it separates edges of different labels.
We may assume that
$$
\ell(A)=0,\quad \ell(C)=\ell(B)=1
$$
although the opposite assignment of labels is also possible.
In fact, there are two edges of $G=f^{-1}(T)$ mapping onto $C$.
They are separated by a critical point mapping to $\sqrt 2$.
One or the other assignment of labels depends on which of the two edges is chosen as $e'_b$.
The latter, in turn, depends on which complementary component of $G$ contains the point $y$.
The oriented edges $A$, $B$, $C$ have signature $(0,1)$.
The opposite oriented edges $A^{-1}$, $B^{-1}$, $C^{-1}$ have signature $(1,0)$.

By Theorem \ref{t:ThmB-hom}, the biset of $f$ is represented as follows:
\medskip

\bgroup
\def\arraystretch{1.5}
\centerline{
\begin{tabular}{|c|c|c|c|c|c|c|}
 \hline
 & $a$& $b$ & $c$& $a^{-1}$& $b^{-1}$& $c^{-1}$\\
 \hline
0& $a^{-1}1$ & $1$ & $1$ & $b^{-1}1$& $c^{-1}1$& $1$\\
\hline
1& $b\,0$ & $c\,0$ & $0$ & $a\,0$ & $0$ & $0$\\
\hline
\end{tabular}
}
\egroup

\medskip

With the help of the ivy iteration, we found the pullback invariant subset of $\ivy(f)$ of order 40 containing $[T]$.
Within this subset, there are three invariant spanning trees for $f$, up to homotopy, see Figure \ref{fig:cl-cb}.
Vertex 3 corresponds to the invariant spanning tree
$$
\xymatrix{
{\stackrel{-1}{\bullet}} \ar@{-}[r] &
{\stackrel{0}{\circ}} \ar@{-}[r] &
{\stackrel{1}{\circ}} \ar@{-}[r] &
{\stackrel{\sqrt{2}}{\bullet}}
}
$$
Vertex 35 corresponds to the invariant spanning tree
$$
\xymatrix{
{\stackrel{-1}{\bullet}} \ar@{-}[dr] &
{\stackrel{0}{\circ}} \ar@{-}[d] &
{\stackrel{1}{\circ}} \ar@{-}[dl] &
{\stackrel{\sqrt{2}}{\bullet}} \ar@{-}[dll]\\
& {\circ} & &
}
$$
Finally, vertex 38 corresponds to the invariant spanning tree
$$
\xymatrix{
& {\circ} & &\\
{\stackrel{-1}{\bullet}} \ar@{-}[ur] &
{\stackrel{0}{\circ}} \ar@{-}[u] &
{\stackrel{1}{\circ}} \ar@{-}[ul] &
{\stackrel{\sqrt{2}}{\bullet}} \ar@{-}[ull]
}
$$

\subsection{A capture of the Chebyshev polynomial}
Finally, we consider an example, where an invariant spanning tree is not known a priori.
Namely, we take a simple capture of the Chebyshev polynomial $p(z)=z^2-2$ whose post-critical set has cardinality $4$.
There are two preimages of $0$ under $p$, namely, $\pm\sqrt 2$.
We restrict our attention to a simple capture of $p$ at $\sqrt{2}$.
There are two simple captures of $p$ at $\sqrt{2}$ corresponding to capture paths $\beta_u$ and $\beta_d$
(``u'' and ``d'' are from ``\textbf{u}p'' and ``\textbf{d}own'').
We may define $\beta_u$ as a path along the external ray of argument $\frac 18$, and
  $\beta_d$ as a path along the external ray of argument $\frac 78$.
Clearly, any other simple capture path for $p$ ending at $\sqrt{2}$ is homotopic to $\beta_u$ or $\beta_d$
  relative to the set $\{\sqrt{2},0,-2,2\}$ (which is the post-critical set of the captures).
The extended Hubbard tree $T$ with vertices in $\{\sqrt{2},0,-2,2\}$ is the following:
$$
\xymatrix{
{\stackrel{-2}{\bullet}} \ar@{-}[r] &
{\stackrel{0}{\circ}} \ar@{-}[r] &
{\stackrel{\sqrt 2}{\bullet}} \ar@{-}[r] &
{\stackrel{2}{\circ}}
}
$$
(We have marked post-critical points of the capture rather than of $p$.)

Consider $f_u=\sigma_{\beta_u}\circ p$.
To lighten the notation, we will write $\sigma_u$ instead of $\sigma_{\beta_u}$.
The full preimage $G_u=f_u^{-1}(T)$ can be obtained as the full preimage under $p$ of $\sigma_u^{-1}(T)$.
The tree $\sigma_u^{-1}(T)$ can be represented (up to homotopy) as follows.
Comparing $\sigma_u^{-1}(T)$ to $T$: the edge $[-2,0]$ is preserved in $\sigma_u^{-1}(T)$.
The edges $[0,\sqrt{2}]$ and $[\sqrt{2},2]$ are replaced with the external rays of arguments $\frac 14$ and $0$, respectively.
See Figure \ref{fig:tcap}, top right, for an illustration of $\sigma_u^{-1}(T)$.

\begin{figure}
  \centering
  \includegraphics[height=5cm]{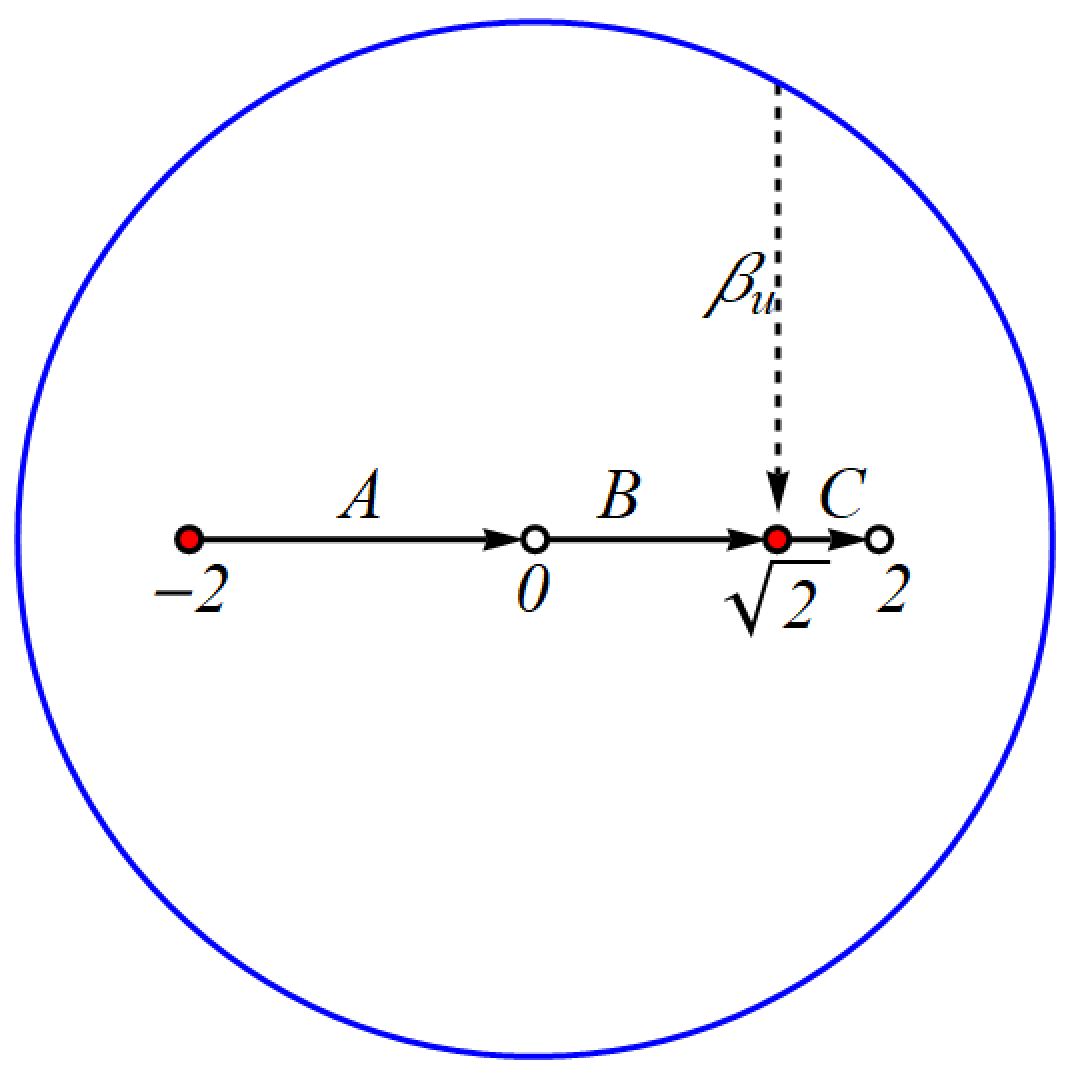}\quad
  \includegraphics[height=5cm]{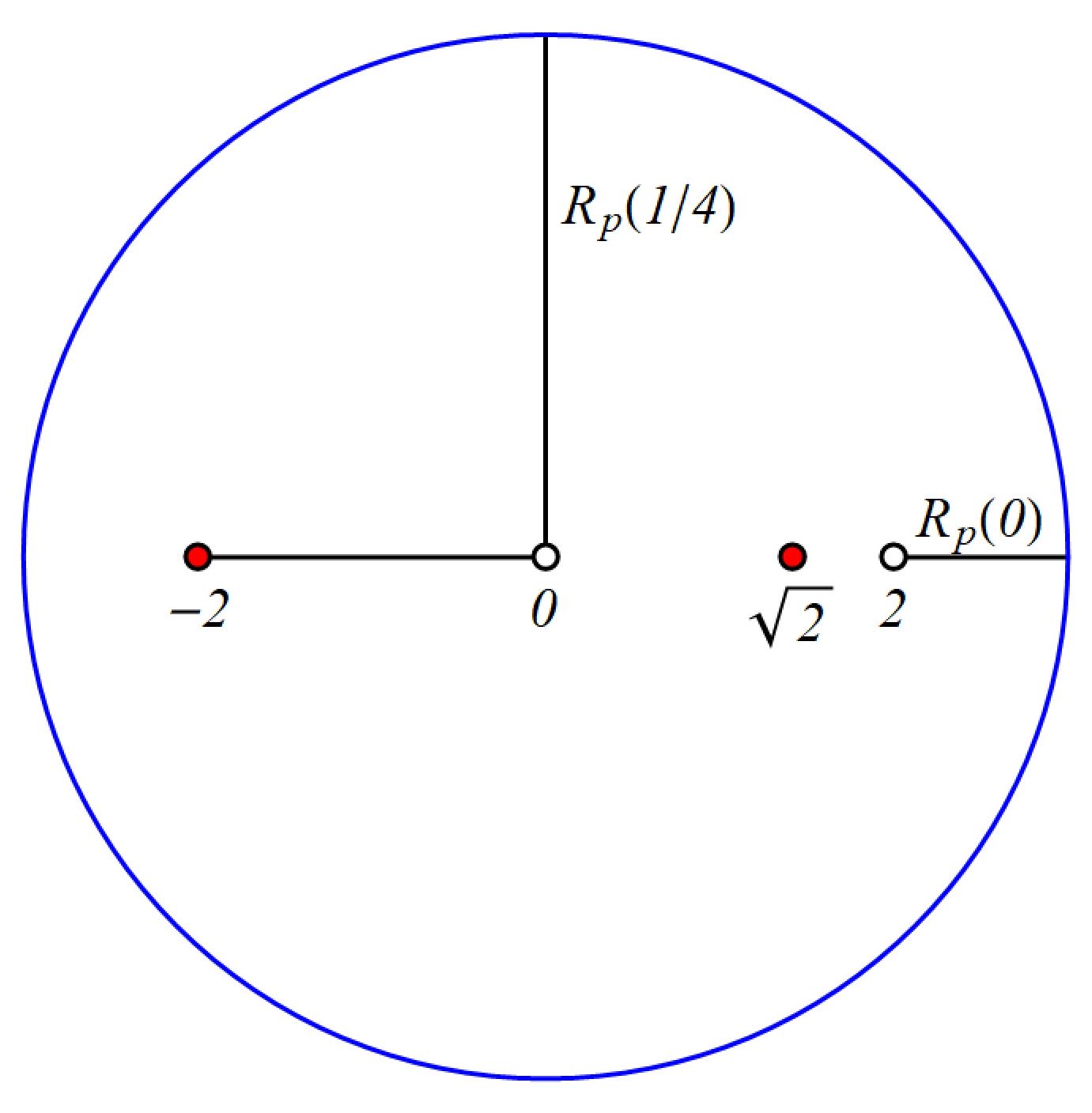}\\
  \includegraphics[height=5cm]{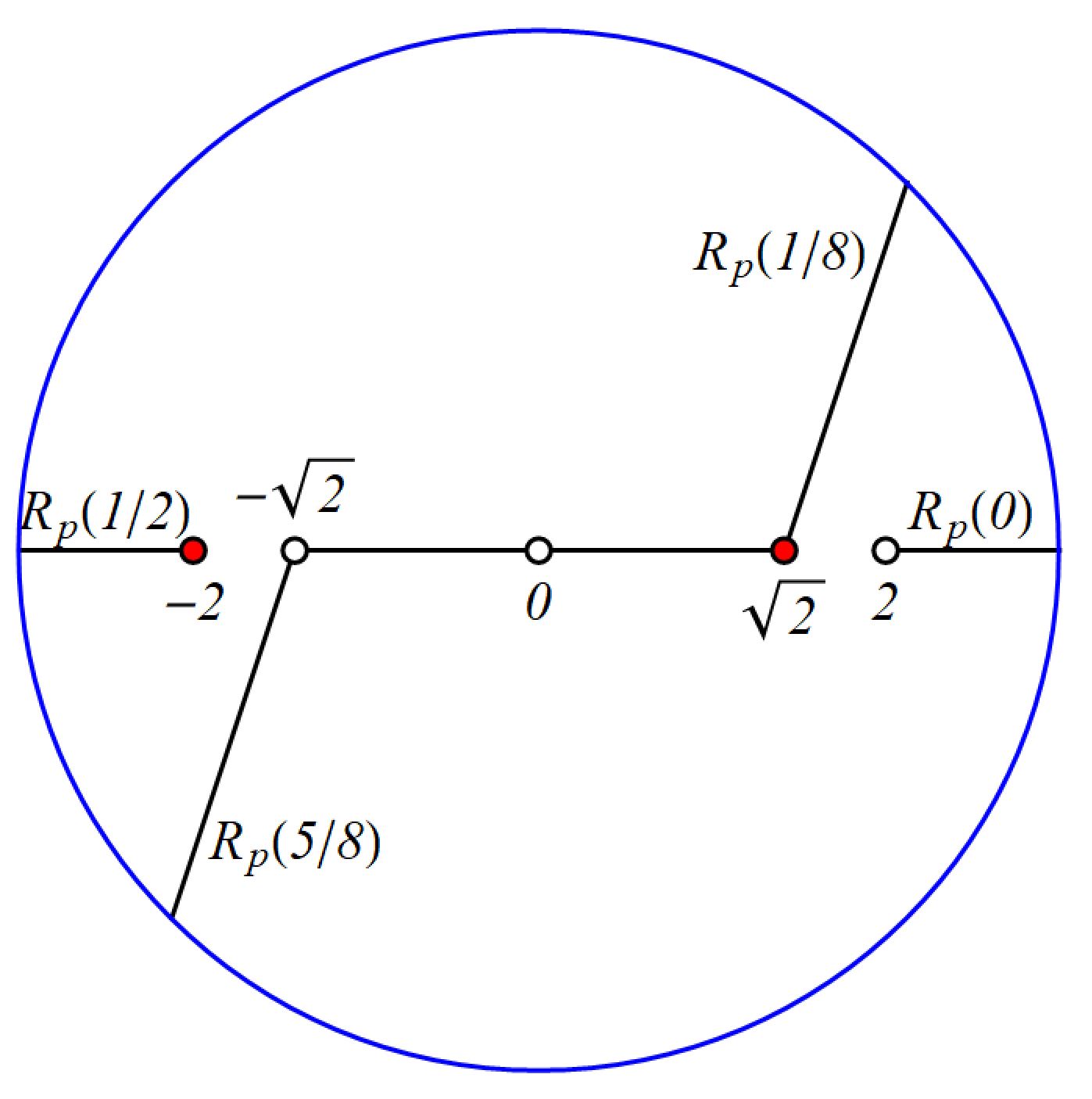}\quad
  \includegraphics[height=5cm]{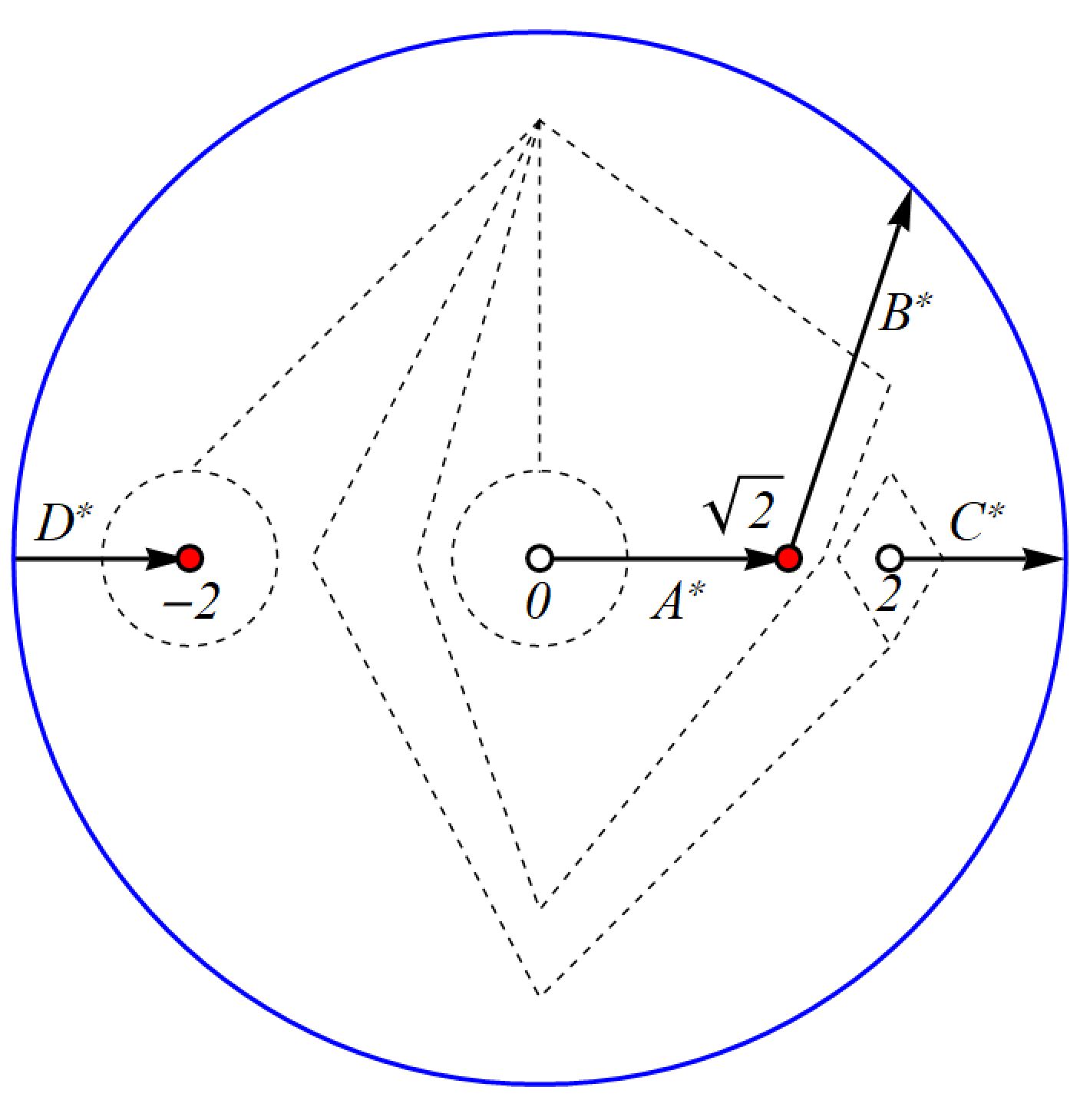}
  \caption{\footnotesize The simple capture $f_u=\sigma_u\circ p$ of $p(z)=z^2-2$.
  Copies of the sphere are represented as disks, in which the boundary circles are assumed to be collapsed.
  Thus, in each of the four pictures, one should think of the entire boundary circle as one point.}
  \label{fig:tcap}
\end{figure}

Next, $G_u$ is obtained as the full $p$-preimage of $\sigma_u^{-1}(T)$, see Figure \ref{fig:tcap}, bottom left.
The graph $G_u$ consists of the line segments $[-\sqrt{2},0]$, $[0,\sqrt{2}]$,
 and the external rays $R_p(0)$, $R_p(\frac 12)$, $R_p(\frac 18)$, $R_p(\frac 58)$.
Here $R_p(\theta)$ stands for the external ray of argument $\theta$ in the dynamical plane of $p$.
We want to find a spanning tree $T^*\subset G_u$ so that to make $(T^*,T)$ into a dynamical tree pair.

Denote the edges of $T$ as $A$, $B$, $C$ and orient them as in Figure \ref{fig:tcap}, top left.
Then $Z=A\cup B$ is a simple arc in $T$ connecting the two critical values
(this is consistent with the meaning of the symbol $Z$ in Sections \ref{ss:baseedge} and \ref{ss:ivy-geom}).
The full preimage $f_u^{-1}(Z)$ is the simple closed curve consisting of the segments $[-\sqrt 2,0]$, $[0,\sqrt 2]$,
 the rays $R_p(\frac 18)$, $R_p(\frac 58)$, and the point $\infty$.
We need to chose an edge $e'_b$ in $f_u^{-1}(Z)$ that will not be included into $T^*$.
Take $e'_b$ to be the edge that goes along the ray $R_p(\frac 58)$.
Then the segment $[-\sqrt 2,0]$ should also be removed from $T^*$ since $-\sqrt{2}$ is not in $P(f_u)$.
Thus $T^*$ is as shown in Figure \ref{fig:tcap}, bottom right.
Denote the edges of $T^*$ by $A^*$, $B^*$, $C^*$, $D^*$ and orient them as in Figure \ref{fig:tcap}, bottom right.
The edges of $T^*$ map over the edges of $T$ as follows:
$$
A^*\to A,\quad B^*\to B,\quad C^*\to C^{-1},\quad D^*\to C.
$$
Thus, in this example, every edge of $T^*$ maps over just one edge of $T$, which is not the case in general.

We will write $a$, $b$, $c$ for the elements of $\pi_{f_u}$
 corresponding to the edges $A$, $B$, $C$, respectively.
Thus the generating set $\Ec=\Ec_T$ consists of $id$, $a^{\pm 1}$, $b^{\pm 1}$, $c^{\pm 1}$.
Similar convention will apply to $T^*$,
 so that $\Ec^*=\Ec_{T^*}$ consists of $id$, $a^*$, $b^*$, $c^*$, $d^*$, and their inverses.
Elements of $\Ec_{T^*}$ are shown through their representatives in Figure \ref{fig:tcap}, bottom right (dashed loops).
Inspecting how the dashed loops cross the edges of $T$, we can express elements of $\Ec^*$ through those of $\Ec$:
$$
a^*=ba^{-1},\quad b^*=ca^{-1},\quad c^*=ac^{-1}a^{-1},\quad d^*=a^{-1}.
$$

\begin{figure}
  \centering
  \includegraphics[width=12cm]{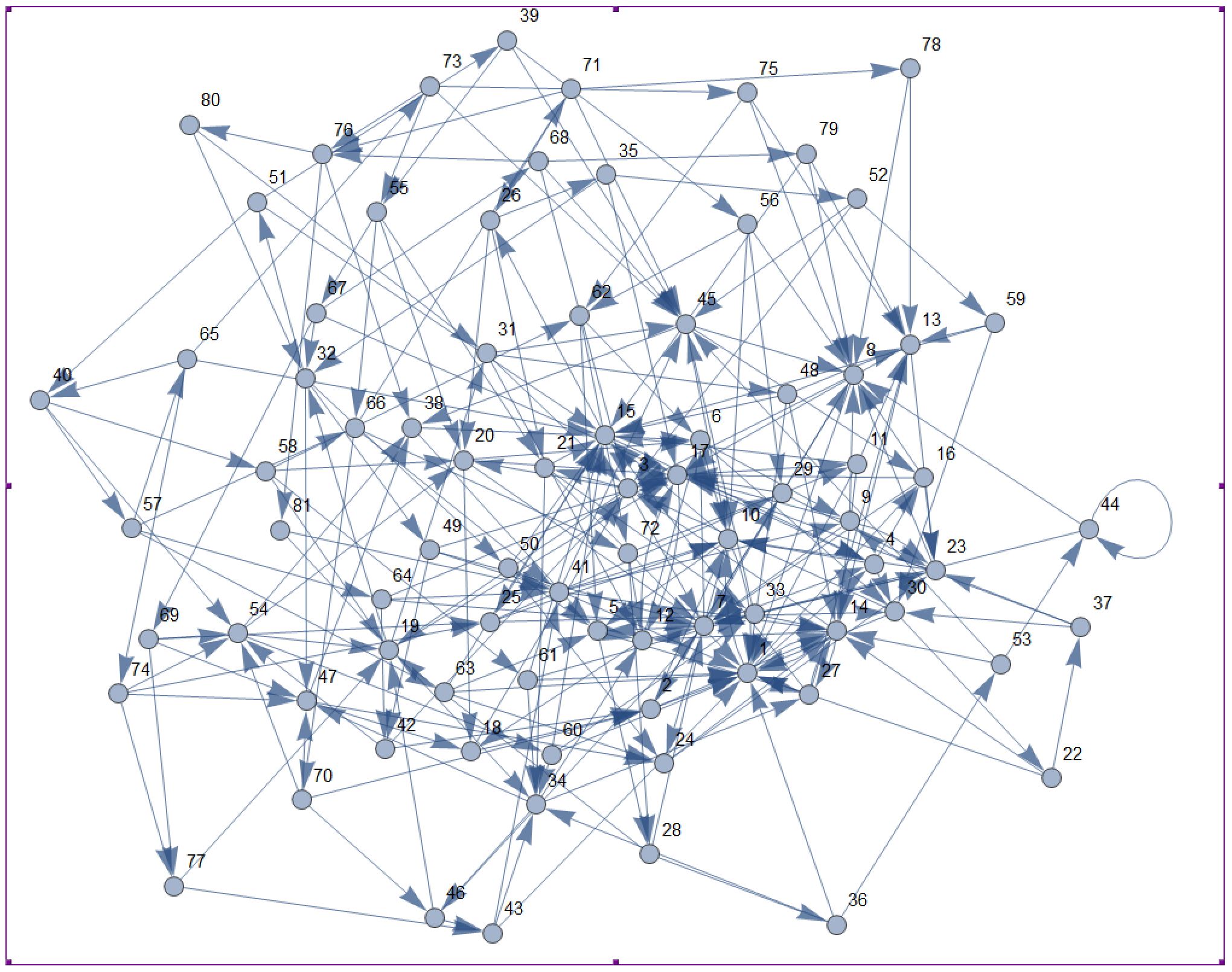}
  \caption{\footnotesize The pullback invariant subset of $\ivy(f_u)$ containing $[T]$.
  There are 81 objects in this subset.
  Vertex 44 represents an invariant ivy object.}
  \label{fig:cl-ccheb}
\end{figure}

We now compute the presentation of the biset of $f_u$ associated with $(T^*,T)$ as in Theorem \ref{t:ThmB-hom}.
To this end, we first need to choose post-critical pseudoaccesses for $T$.
The critical values of $f_u$ are $v_1=-2$ and $v_2=\sqrt 2$.
Since $v_1$ is an endpoint of $T$, there is only one pseudoaccess at $v_1$.
However, $v_2$ admits to pseudoaccesses: one is from above, and the other is from below.
We choose the one from above.
According to this choice of pseudoaccesses, we have $S^0(T)=(A,B)$ and $S^1(T)=(C,C^{-1},B^{-1},A^{-1})$.
Therefore, the edges $A$ and $B$ have signature $(0,1)$, and $C$ has signature $(1,1)$.
Next, we need to compute the labels for all edges of $T^*$.
Since $e'_b$ by definition has label 1, we have $\ell(A^*)=\ell(B^*)=0$
(indeed, we change the label as we pass through the critical point $0$).
It is slightly harder to figure out the labels of $C^*$ and $D^*$ since we need to look at the critical pseudoaccesses at $\infty$.
The pseudoaccesses at $\infty$ separate $B^*$ and $C^*$ from $D^*$.
Therefore, we have $\ell(C^*)=0$ and $\ell(D^*)=1$.

We can now write down the presentation of the biset $\Xc_{f_u}(y)$ associated with $(T^*,T)$:

\medskip

\bgroup
\def\arraystretch{1.5}
\centerline{
\begin{tabular}{|c|c|c|c|c|c|c|}
 \hline
 & $a$& $b$ & $c$& $a^{-1}$& $b^{-1}$& $c^{-1}$\\
 \hline
0& $a^*1$ & $b^*1$ & $d^*0$ & $1$& $1$& ${d^*}^{-1}0$\\
\hline
1& $0$ & $0$ & ${c^*}^{-1}1$ & ${a^*}^{-1}0$ & ${b^*}^{-1}0$ & $c^*1$\\
\hline
\end{tabular}
}
\egroup

\medskip

With the help of a computer, we found a pullback invariant subset of size 81 in $\ivy(f)$ containing $[T]$,
 see Figure \ref{fig:cl-ccheb}.
This subset contains an invariant ivy object.
Thus, we found an invariant (up to homotopy) spanning tree for $f$.
This invariant tree is a star.

\subsection{Some open questions}
The following are open questions about the pullback relation on spanning trees that seem important:
\begin{itemize}
  \item For a quadratic rational Thurston map $f$, can there be an infinite sequence of pairwise different ivy objects $[T_n]$ such that $[T_{n}]\multimap [T_{n+1}]$?
  \item Is there a uniform upper bound on the number of invariant spanning trees, up to isotopy, for a quadratic rational Thurston map?
  \item Can there be two disjoint pullback invariant subsets of $\ivy(f)$,
   for a quadratic rational Thurston map $f$?
\end{itemize}

\subsection{Acknowledgements}
The authors are grateful to D. Dudko and M. Hlush\-chan\-ka for useful discussions,
to D. Schleicher and Jacobs University Bremen for hospitality and inspiring working conditions during the workshop ``Dynamics, Geometry and Groups'' in May 2017, where these and other enlightening discussions took place.
We are also grateful to the referee for valuable remarks and suggestions.

\end{document}